\documentclass{ws-ccm}

\usepackage{amsmath}
\usepackage{amsfonts}
\usepackage{amssymb}
\usepackage{amsthm}
\usepackage{hyperref}
\usepackage{mathtools}
\usepackage{graphicx}

\title{\Large Higher-Order Phase Transitions with Line-Tension Effect \vspace{-0.35cm}}
\author{Bernardo Galv\~{a}o-Sousa}
\address{Department of Mathematics and Statistics, McMaster University\\
Hamilton ON L8S 4K1, Canada
\\
\emph{\texttt{beni@mcmaster.ca}}}

\def\Per {\text{Per}}

\def\keywords {\paragraph{Keywords} }

\def\msc {\paragraph{AMS subject classification} }

\def\ack {\paragraph{Acknowledgements} }

\numberwithin{equation}{section}

\def\R {\mathbb{R}}

\def\N {\mathbb{N}}
\def\eps {\varepsilon}
\def\bs {\backslash}

\def\loc {\text{loc}}
\def\grad {\nabla}
\def\divf {\text{div }}
\def\mint {\ds -\hspace{-11pt}\int}
\def\mmint {\text{-}\hspace{-5.5pt}\int}
\newcommand{\fpar}[2]{\frac{\partial #1}{\partial #2}}
 
\DeclareMathOperator*{\esssup}{ess\, sup}

\DeclareMathOperator*{\argmin}{arg\, min} 
\def\supp {\text{supp }}
\def\xto {\xrightarrow}

\def\ul #1{\underline{#1}}
\def\ol #1{\overline{#1}}

\def\til{\widetilde}

\def\Leb #1{{\cal L}^{#1}}
\def\Haus #1{{\cal H}^{#1}}

\def\restr #1{\lfloor_{#1}}

\def\ssubset {\Subset}

\def\leq {\leqslant}
\def\geq {\geqslant}

\def\ds {\displaystyle}
\def\ts {\textstyle}

\def\dist {\text{dist }}
\def\I {\mathbb{I}}
\def\supp {\text{supp }}

\renewcommand{\theequation}{\thesection.\arabic{equation}}
\newtheorem{defi}{Definition}[section]
\newtheorem{thm}[defi]{Theorem}
\newtheorem*{thm*}{Theorem}
\newtheorem{prop}[defi]{Proposition}
\newtheorem*{prop*}{Proposition}
\newtheorem{lem}[defi]{Lemma}
\newtheorem{cor}[defi]{Corollary}
\newtheorem*{conj}{Conjecture}

\theoremstyle{definition}
\newtheorem*{rem}{Remark}

\def\subpar#1{#1 \quad }

\begin{document}

\maketitle

\begin{abstract}

The behavior of energy minimizers at the boundary of the domain is of great importance in the Van de Waals-Cahn-Hilliard theory for fluid-fluid phase transitions, since it describes the effect of the container walls on the configuration of the liquid. 
This problem, also known as the liquid-drop problem, was studied by Modica in \cite{Modica_wb}, and in a different form by Alberti, Bouchitt\'e, and Seppecher in \cite{ABS} for a first-order perturbation model.
This work shows that using a second-order perturbation Cahn-Hilliard-type model, the boundary layer is intrinsically connected with the transition layer in the interior of the domain. Precisely, considering the energies
$$
{\cal F}_{\eps}(u) := \eps^{3} \int_{\Omega} |D^{2}u|^{2} + \frac{1}{\eps} \int_{\Omega} W (u) + \lambda_{\eps} \int_{\partial \Omega} V(Tu),
$$
where $u$ is a scalar density function and $W$ and $V$ are double-well potentials, the exact scaling law is identified in the critical regime, when $\eps \lambda_{\eps}^{\frac{2}{3}} \sim 1$.
\keywords Gamma limit \and functions of bounded variations \and functions of bounded variations on manifolds \and phase transitions
\msc 49Q20, 49J45, 58E50, 76M30
\end{abstract}

\section{Introduction}

In this paper we seek to estimate the asymptotic behavior of the family of energies
$$
\eps^{3} \int_{\Omega} |D^{2}u|^{2} \, dx + \frac{1}{\eps} \int_{\Omega} W(u) \, dx + \lambda_{\eps} \int_{\partial \Omega} V(Tu) \, d\Haus{N-1},
$$
where $u \in H^{2}(\Omega)$, $\Omega$ is a bounded open set in $\R^{N}$ of class $C^2$, $Tu$ is the trace of $u$ on $\partial\Omega$, $W$ and $V$ are continuous and non-negative double-well potentials with quadratic growth at infinity, and $\ds \lim_{\eps \rightarrow 0^{+}} \lambda_{\eps}=\infty$.

It is known that the transition layer in the interior of the domain has width of order $\eps$ (see \cite{ModicaMortola}, \cite{Modica_wob}, \cite{Modica_wb}, \cite{ABS}, \cite{FM}, \cite{CFL}, \cite{Garroni}).
To formally find the order of the width of the transition layer on the boundary, it suffices to study the case $N=2$.
Therefore, by focusing on a neighborhood of a point on the boundary (assuming the boundary is flat), consider a $2-D$ energy in the half ball of radius $\delta$ centered at that point $x_0$ of the boundary, and changing variables to a fixed domain, e.g. the unit ball, we obtain
$$
\frac{\eps^{3}}{\delta^{2}} \iint_{B^{+}} |D^{2}u|^{2} \, dx\,dy + \frac{\delta^{2}}{\eps} \iint_{B^{+}} W(u) \, dx\,dy + \lambda_{\eps}\delta \int_{E} V(Tu) \, d\Haus{1}.
$$

Equi-partition of energy between the first and last terms leads to $\delta \approx \eps \lambda_{\eps}^{-\frac{1}{3}}$ which, in turn, yields $\frac{\delta^{2}}{\eps} \approx \eps\lambda_{\eps}^{-\frac{2}{3}}$, which vanishes with $\eps$, which seems to indicate that the middle term will not contribute for the transition on the boundary
.
One also concludes that on the boundary, the energy will scale as $\frac{\eps^{3}}{\delta^{2}} \approx \lambda_{\eps} \delta \approx \eps \lambda_{\eps}^{\frac{2}{3}}$.
Hence there are three essential regimes for this energy depending on how the quantity $\eps \lambda_{\eps}^{\frac{2}{3}}$ behaves as $\eps\to0^+$. 

In this paper we study the case in which $\eps \lambda_{\eps}^{\frac{2}{3}}$ converges to a finite and strictly positive value. The other two regimes will be treated in a forthcoming paper.

%
%

Consider the functional
\begin{equation}\label{eq:F}
{\cal F}_{\eps}(u) :=
	\begin{cases}
		\ds \eps^{3} \int_{\Omega} |D^{2}u|^{2} \, dx + \frac{1}{\eps} \int_{\Omega} W(u) \, dx + \lambda_{\eps} \int_{\partial \Omega} V(Tu) \, d\Haus{N-1}
			& \text{if } u \in H^{2}(\Omega), \\
		\infty
			& \text{otherwise.}
	\end{cases}
\end{equation}

\begin{thm}[Compactness]\label{thm:compactness}

Let $\Omega \subset \R^{N}$ be a bounded open set of class $C^{2}$ and let $W:\R \rightarrow [0, \infty)$ be such that 
\begin{align*}
(H^W_1)  \quad & W \text{ is continuous and } W^{-1}(\{0\}) = \{a, b\} \text{ for some } a,b\in \R, a<b; \\
(H^W_2) \quad & W(z) \geq C|z|^{2} - \frac{1}{C} \text{ for all } z \in \R \text{ and for some } C > 0.
\end{align*}
Let $V:\R \rightarrow [0, \infty)$ be such that 
\begin{align*}
(H^V_1)  \quad & V \text{ is continuous and } V^{-1}(\{0\}) = \{\alpha, \beta\} \text{ for some } \alpha,\beta\in \R, \alpha<\beta; \\
(H^V_2)  \quad & V(z) \geq C|z|^{2} - \frac{1}{C} \text{ for all } z \in \R \text{ and for some } C > 0; \\
(H^V_3)  \quad & V(z) \geq \frac{1}{C} \min \bigl\{ |z - \beta|, |z - \alpha| \bigr\}^2 \text{ for all } z \in (\alpha-\rho,\alpha+\rho) \cup (\beta-\rho,\beta+\rho)\\
			&  \text{ and for some } C, \rho > 0. 
\end{align*}
Assume that $\eps \lambda_{\eps}^{\frac23} \to L \in (0,\infty)$ as $\eps \to 0^+$ and consider a sequence $\{u_{\eps}\} \subset H^2(\Omega)$ such that $\sup_{\eps>0} {\cal F}_{\eps}(u_{\eps}) < \infty$.
Then there exist a subsequence $\{u_{\eps}\}$ (not relabeled), $u \in BV\bigl(\Omega;\{a,b\}\bigr)$, and $v \in BV\bigl(\partial \Omega; \{\alpha, \beta\}\bigr)$ such that $u_{\eps} \to u$ in $L^2(\Omega)$ and $Tu_{\eps} \to v$ in $L^2(\partial\Omega)$.

\end{thm}

The next theorem concerns the critical regime where $\eps$ and $\lambda_{\eps}$ are ``balanced'', i.e. $\eps \lambda_{\eps}^{\frac23} \sim 1$, and all terms play an important role.
Here $\lambda_{\eps}$ is large enough to render the energy sensitive to the transition that occurs on the boundary, but not too big as to force the value on the boundary to converge to a constant.

We define
\begin{enumerate}
\item $E_{a} := \left\{ x \in \Omega: u(x) = a \right\}$ for all $u \in BV\bigl(\Omega;\{a,b\}\bigr)$;

\item $m$ is the energy density per unit area on the transition interfaces between the interior potential wells, precisely,
\begin{equation}\label{eq:def_m}
m  := \inf \biggl\{ \int_{-R}^{R} \bigl( W(f(t)) + |f''(t)|^{2} \bigr) \, dt : 
		 f\in H^{2}_{\loc}(\R), f(-t) = a, f(t) = b \text{ for all } t \geq R, R >0 \biggr\};
\end{equation}

\item $\sigma$ is the interaction energy on the transition interface between bulk wells and boundary wells, i.e.,
\begin{equation}\label{eq:def_sigma}
\sigma(z,\xi) 
	:= \inf \biggl\{ \int_{0}^{R} \bigl( W(f(t)) + |f''(t)|^{2} \bigr) \, dt:
			f \in H^{2}_{\loc}\bigl((0,\infty)\bigr), f(0)=\xi, f(t) = z \text{ for all } t \geq R, R > 0 \biggr\};
\end{equation}

\item $F_{\alpha} := \left\{ x \in \partial\Omega: v(x) = \alpha \right\}$ for all $v \in BV\bigl(\partial \Omega;\{\alpha,\beta\}\bigr)$;

\item $\ul{c}$ is a lower bound to the energy on a transition interface between the wells of the boundary potential,
\begin{multline}\label{eq:def_under_c}
\ul{c} := \inf \left\{ \frac{1}{8} \int_{-R}^{R} \int_{-R}^{R} \frac{|f'(x) - f'(y)|^{2}}{|x-y|^{2}} \, dx \, dy + \int_{-R}^{R} V\bigl(f(x)\bigr) \, dx: \right. f \in H_{\loc}^{\frac{3}{2}}(\R),\\
	 \left.  f' \in H^{\frac{1}{2}}(\R), f(-t)= \alpha, f(t)=\beta \text{ for all } t \geq R, R > 0 \right\};
\end{multline}

 \item $\ol{c}$ is an upper bound to the energy on a transition interface between the wells of the boundary potential,
\begin{multline}\label{eq:def_over_c}
\ol{c} := \inf \left\{ \frac{7}{16} \int_{-\infty}^{\infty} \int_{-\infty}^{\infty} \frac{|f'(x) - f'(y)|^{2}}{|x-y|^{2}} \, dx \, dy + \int_{-\infty}^{\infty} V\bigl(f(x)\bigr) \, dx : \right. \\
	 \left. f \in H_{\loc}^{\frac{3}{2}}(\R), f(-t)= \alpha, f(t)=\beta \text{ for all } t \geq R, R > 0 \right\}.
\end{multline}

\end{enumerate}

\begin{thm}[Critical case]\label{thm:critical}

Under the same hypotheses of Theorem \ref{thm:compactness}
the following statements hold:
\begin{enumerate}
\item[(i)] {\bf (Lower bound)} For every $u \in BV(\Omega; \{a,b\})$ and $v \in BV(\partial\Omega;\{\alpha,\beta\})$ and for every sequence $\{u_{\eps}\} \subset H^{2}(\Omega)$ such that $u_{\eps} \rightarrow u$ in $L^{2}(\Omega)$, $Tu_{\eps} \rightarrow v$ in $L^{2}(\partial\Omega)$, we have
\begin{equation*}
\liminf_{\eps \rightarrow 0^{+}}{\cal F}_{\eps}(u_{\eps}) 
	\geq m \Per_{\Omega}(E_{a}) 
	+ \sum_{z =a,b} \sum_{\xi=\alpha,\beta} \sigma(z,\xi) \Haus{N-1}\bigl( \{ Tu = z\} \cap \{v = \xi\} \bigr)
	+ \ul{c} L \Per_{\partial\Omega}(F_{\alpha});
\end{equation*}

\item[(ii)] {\bf (Upper bound)} For every $u \in BV(\Omega; \{a,b\})$ and $v \in BV(\partial\Omega;\{\alpha,\beta\})$, there exists a sequence $\{u_{\eps}\} \subset H^{2}(\Omega)$ such that $u_{\eps} \rightarrow u$ in $L^{2}(\Omega)$, $Tu_{\eps} \rightarrow v$ in $L^{2}(\partial\Omega)$, and
\begin{equation*}
\limsup_{\eps \rightarrow 0^{+}}{\cal F}_{\eps}(u_{\eps}) 
	\leq m \Per_{\Omega}(E_{a}) 
	+ \sum_{z =a,b} \sum_{\xi=\alpha,\beta} \sigma(z,\xi) \Haus{N-1}\bigl( \{ Tu = z\} \cap \{v = \xi\} \bigr)
	+ \ol{c} L \Per_{\partial\Omega}(F_{\alpha}).
\end{equation*}
\end{enumerate}
\end{thm}

The main results, Theorems \ref{thm:compactness} and \ref{thm:critical}, imply, in particular, that
$$
\min_{\substack{a<\mmint_{\Omega}u \, dx<b \\ \alpha< \mmint_{\partial\Omega}v \, d\Haus{N-1} <\beta}} \mathcal{F}_{\eps} = O(1) \quad\text{as } \eps\to 0^+,
$$
where we impose a mass constraint to avoid trivial solutions which yield no energy. Note that these conditions pose no difficulties to the $\Gamma$-convergence due to the strong convergence of $u_\eps$ and $Tu_\eps$. 
Thus we identify the precise scaling law for the minimum energy in the parameter regime $\eps\lambda_{\eps}^{\frac23}\sim 1$.

Observe that, although Theorem \ref{thm:critical} does not prove that the sequence $\{\mathcal{F}_{\eps}\}_{\eps>0}$ $\Gamma$-converges as $\eps\to 0^+$, since the constants of the lower and upper bounds for the last transition term do not match, we can apply Theorem 8.5 from \cite{DalMaso} to prove that there exists a subsequence $\eps_n \to 0^+$ such that the corresponding subsequence of functionals $\Gamma$-converges.

Hence Theorem \ref{thm:critical} shows that the limiting functional concentrates on the three different kinds of transition layers: an interior transition layer of dimension $N-1$, where the limiting value of $u$ makes the transition between $a$ and $b$; the boundary of the domain, also of dimension $N-1$, where there is the transition between the interior phases $a$ and $b$ and the boundary phases $\alpha$ and $\beta$; and a transition interface on the boundary, of dimension $N-2$, where the limiting value of the trace $Tu$ makes the transition between $\alpha$ and $\beta$.

The difficulties in proving a $\Gamma$-convergence result arise mainly from the nature of the functional under consideration. 
On one hand, the energy involves second-order derivatives, which prevents us from following the usual techniques in phase transitions, such as truncation and rearrangement arguments to obtain monotonically increasing test functions for the constant $\ol{c}$. In \cite{ABS}, these techniques are crucial to find a test function that matches both the lifting constant and the optimal profile problem for the boundary wells.
On the other hand, for the boundary term, the functionals are also nonlocal. Thus the estimates for the recovery sequence have to be sharper, since the nonlocality extends its contribution beyond the characteristic length of the phase transition. The usual methods for localization make use of truncation arguments, which do not apply in this setting due to the fact that the fractional seminorm is of higher-order.

Similar difficulties can also be found in the papers \cite{CK, CKO1,CKO2, CCKO} where, similarly, the $\Gamma$-convergence is not established.

The difference between the constants $\ul{c}$ and $\ol{c}$ arises from two factors. First, from Proposition \ref{prop:lifting} it does not follow that the lifting constant is independent of the value of the trace $g$. And second, when estimating the upper bound for the recovery sequence, the transition between $\alpha$ and $\beta$ is accomplished on a layer of thickness $\delta_{\eps} = o(\eps)$. So we rescale the integrals by $\delta_\eps$, but because of the non-locality of the fractional energy, it obtains a contribution from a layer of thickness $\eps$, which after rescaling becomes of thickness $\eps / \delta_\eps \to \infty$. This accounts for the fact that the integration limits of the constant $\ol{c}$ extend to infinity, while for $\ul{c}$ they are bounded.

\hfil

The proofs of Theorems \ref{thm:compactness} and \ref{thm:critical} are divided through the next sections.  
We begin by studying two auxiliary one-dimensional problems.
More precisely, let $I, J \subset \R$ be two open intervals and define the following functionals
\begin{equation}\label{eq:bulkF}
F_{\eps}(u;I) := \begin{cases}
\ds \eps^3 \int_I |u''(x)|^2 \, dx + \frac{1}{\eps} \int_I W\bigl(u(x)\bigr)\, dx
   & \text{if } u \in H^2(I), \\
\infty
   & \text{otherwise,}
\end{cases}
\end{equation}
and
\begin{equation}\label{eq:bdyG}
G_{\eps}(v;J) := \begin{cases}
\ds \frac{\eps^3}{8} \int_J \int_J \frac{\bigl| v'(x)-v'(y) \bigr|^2}{|x-y|^2} \, dx \, dy + \lambda_{\eps} \int_J V\bigl( v(x)\bigr)\, dx
   & \text{if } v \in H^{\frac{3}{2}}(J), \\
\infty
   & \text{otherwise.}
\end{cases}
\end{equation}

In Sections \ref{ssec:bulkF1D_compc} and \ref{ssec:bulkF1D_LB} we prove a compactness result and a lower bound for $F_{\eps}$ which follows the techniques developed in \cite{FM}.
In Section \ref{ssec:bdyG1D_compc} we will prove a compactness result for $G_{\eps}$, while in Section \ref{ssec:bdyG1D_LB} we will prove a lower bound by finding ``good points'' $x_i^{\pm}$ such that most of the transition energy is concentrated between $x_i^-$ and $x_i^+$ and we modify the original sequence $\{u_n\}$ on a small set to be admissible for $\ul{c}$.
In Section \ref{ssec:Compactness} we will prove Theorem \ref{thm:compactness} in the critical regime using a slicing argument to reduce the compactness in the interior to the auxiliary problem studied in Section \ref{ssec:bulkF1D_compc}, and analogously, we reduce the compactness on the boundary to the one-dimensional problem for $G_{\eps}$ studied in Section \ref{ssec:bdyG1D_compc}.
In Section \ref{ssec:LB_ND} we prove the lower bound result for Theorem \ref{thm:critical} using the fact that the energy concentrates in different mutually singular sets.
Finally, in Section \ref{ssec:UB} we prove the upper bound for Theorem \ref{thm:critical}.

\hfil

From Theorem \ref{thm:critical}, we deduce the following corollary.

\begin{cor}
Under the same hypotheses of Theorem \ref{thm:compactness}, and assuming that $\alpha=\beta$, then the sequence $\{\mathcal{F}_{\eps}\}_{\eps >0}$ $\Gamma$-converges as $\eps \to 0^+$ to 
$$
\mathcal{F}_0(u) := \begin{cases}
	\ds m \Per_{\Omega}(E_{a}) 
	+ \sum_{z =a,b} \sigma(z,\alpha) \Haus{N-1}\bigl( \{ Tu = z\} \bigr)
		& \text{ if } u \in BV(\Omega;\{a,b\}), \\
	\infty	& \text{ otherwise,}
\end{cases}
$$
where $m$ is defined as in \eqref{eq:def_m} and $\sigma$ is defined as in \eqref{eq:def_sigma}.
\end{cor}


From the result of Theorem \ref{thm:critical}, we know that the $\Gamma$-limit of the functionals $\mathcal{F}_{\eps}$ as $\eps \to 0^+$ will concentrate its energy on three surfaces: the discontinuity surface of $u$, the boundary $\partial\Omega$, and the discontinuity surface of $v$. Moreover, we know the precise energy of the first two terms. For the last term, we expect it to be the product of the perimeter of the surface times the value $c$ of the transition between the two boundary preferred phases $\alpha$ and $\beta$. Since the fractional norm on the boundary is non-local, the definition of $c$ should span the whole real line and the lifting constant should be independent of the function $g$, as in the first-order case (see \cite{ABS}). 
We offer the following conjecture.

\begin{conj}
Under the same hypotheses of Theorem \ref{thm:compactness}, then the sequence $\{\mathcal{F}_{\eps}\}_{\eps >0}$ $\Gamma$-converges as $\eps \to 0^+$ to 
$$
\mathcal{F}_0(u,v) := \begin{cases}
	\ds m \Per_{\Omega}(E_{a}) 
	+ \sum_{z =a,b} \sum_{\xi=\alpha,\beta} \sigma(z,\xi) \Haus{N-1}\bigl( \{ Tu = z\} \cap \{v = \xi\} \bigr)
	+ c L \Per_{\partial\Omega}(F_{\alpha})
		& \text{ if } (u,v) \in \mathcal{V}, \\
	\infty	& \text{ otherwise,}
\end{cases}
$$
where $\mathcal{V} := BV(\Omega;\{a,b\}) \times BV(\partial\Omega;\{\alpha,\beta\})$, $m$ is defined as in \eqref{eq:def_m}, $\sigma$ is defined as in \eqref{eq:def_sigma}, and $c$ is defined by
\begin{equation}\label{eq:def_c}
c := \inf \left\{ \zeta \int_{-\infty}^{\infty} \int_{-\infty}^{\infty} \frac{|f'(x) - f'(y)|^{2}}{|x-y|^{2}} \, dx \, dy + \int_{-\infty}^{\infty} V\bigl(f(x)\bigr) \, dx: f \in H_{\loc}^{\frac{3}{2}}(\R),
	  \lim_{x \to \infty} f(-x)= \alpha, \lim_{x\to \infty} f(x)=\beta \right\}, 
\end{equation}
and $\zeta$ is defined by
\begin{equation}\label{eq:def_beta}
\zeta := \inf \left\{ \frac{\iint_{\R \times \R^{+}} \bigl|D^{2} u(x,y)\bigr|^{2} \, dx \, dy}{ \int_{\R} \int_{\R} \frac{\bigl| g'(x) - g'(y) \bigr|^{2}}{|x-y|^{2}} \, dx \, dy} :
	 u \in H^{2}(\R \times \R^+), Tu(\cdot,0)= g \text{ in } \R \right\},
\end{equation}
which is independent of $g \in H^{\frac32}_{\loc}(\R)$ such that $\lim g(-x)=\alpha$ as $x\to\infty$ and $\lim g(x) = \beta$ as $x \to \infty$.
\end{conj}


%
%

\section{Preliminaries}\label{sec:pre}

\subsection{Slicing}\label{ssec:slicing}

We now show a slicing argument introduced by \cite{ABS} and improved in \cite{FM}.
First we fix some notation. Given a bounded open set $A \subset \R^N$, a unit vector $e$ in $\R^N$, and a function $u:A \to \R$, we denote by
\begin{align*}
& M \text{ the orthogonal complement of } e, \\
& A_e \text{ the projection of } A \text{ onto } M,  \\
& A_e^y := \{ t \in \R: y + te \in A \}, \text{ for all } y \in A_e,\\
& u_e^y \text{ the trace of } u \text{ on } A_e^y \text{, i.e., } u_e^y(t) := Tu(y + te), \text{ for all } y \in A_e.
\end{align*}

\begin{defi}\label{def:deltaclose}

For every $\delta>0$, two sequences $\{v_{\eps}\}, \{w_{\eps}\} \subset L^1(E)$ are said to be $\delta-$close if for every $\eps>0$ $\|v_{\eps} - w_{\eps}\|_{L^{1}(E)} < \delta$.

\end{defi}

\begin{prop}\label{prop:slicing}

Assume that $E$ is a Lipschitz, bounded and open subset of $\R^{N-1}$. If $\{ w_{\eps} \} \subset L^1(E)$ is equi-integrable and if there are $N-1$ linearly independent unit vectors $e_i$ such that for every $\delta > 0$ and for every fixed $i=1,\ldots, N-1$, there exist a sequence $\{v_{\eps} \}$ (depending on $i$) that is $\delta-$close to $\{w_{\eps} \}$ with $\{v^{y}_{\eps}\}$ precompact in $L^{1}(E^{y}_{e_i})$ for $\Haus{N-2}$-a.e. $y\in E_{e_i}$, then $\{w_{\eps} \}$ is precompact in $L^{1}(E)$. 

\end{prop}


\subsection{Fractional order Sobolev spaces}\label{ssec:fractional}

We will use the norms and seminorms of several fractional order spaces, introduced by Besov and Nikol'skii and summarized in \cite{Adams} and \cite{Ziemer}.
Consider the following norms and seminorms for the space $W^{\frac{3}{2},2}(J)$ where $J \subset \R$ is an open interval.
\begin{align*}
& |u|_{H^{\frac{1}{2}}(J)}^{2}:= \int_{J} \int_{J} \frac{\bigl|u(x)-u(y)\bigr|^{2}}{|x-y|^{2}} \, dx \, dy, \\
& |u|_{H^{\frac{3}{2}}(J)}^{2}:= \int_{J} \int_{J} \frac{\left|u(x)-2u\bigl(\frac{x+y}{2}\bigr)+u(y)\right|^{2}}{|x-y|^{4}} \, dx \, dy, \\
& \|u\|_{W^{\frac{3}{2},2}(J)}^{2}:= \|u\|_{H^{1}(J)}^{2} + |u'|_{H^{\frac{1}{2}}(J)}^{2}, \\
& \|u\|_{H^{\frac{3}{2}}(J)}^{2}:= \|u\|_{L^{2}(J)}^{2} + |u|_{H^{\frac{3}{2}}(J)}^{2}.
\end{align*}

We will need to compare the two seminorms and for that we invoke an auxiliary result (see \cite{FL,Stein}).

\begin{prop}\label{prop:Hardy}
Let $r>1$ and let $u:(a,b) \longrightarrow [0,\infty]$ be a Borel function.
Then
$$
\int_{a}^{b} \frac{1}{(x-a)^{r}} \left( \int_{a}^{x} u(y) \, dy \right) \, dx
	\leq \frac{1}{r-1} \int_{a}^{b} \frac{u(x)}{(x-a)^{r-1}} \, dx.
$$
\end{prop}
\begin{lem}\label{lem:seminorms}
Let $J \subset \R$ be an open interval and let $u \in H^{\frac{3}{2}}(J)$.
Then
$$
|u|_{H^{\frac{3}{2}}(J)}^{2} \leq \frac{1}{8} |u'|_{H^{\frac{1}{2}}(J)}^{2}.
$$
\end{lem}

\begin{prop}[Gagliardo-Nirenberg-type inequality]
\label{prop:gagliardo}

Let $J \subset \R$ be an open interval.
Then there exists $C = C(J)>0$ such that
$$
\| u \|_{H^{1}(J)} \leq C \left( \|u\|_{L^{2}(J)}^{\frac{1}{3}} | u' |_{H^{\frac{1}{2}}(J)}^{\frac{2}{3}} + \|u\|_{L^{2}(J)} \right)
$$
for all $u \in H^{\frac{3}{2}}(J)$.
\end{prop}

We recall two inequalities due to Gagliardo and Nirenberg (see \cite{Gagliardo,Nirenberg}).

\begin{prop}\label{prop:gagliardo-nirenberg}

Let $\Omega \subset \R^N$ be a bounded open set satisfying the cone property.
If $u \in L^{2}(\Omega)$ and $\grad^2 u \in L^{2}(\Omega)$, then $u \in H^{2}(\Omega)$ and
$$
\|\grad u \|_{L^{2}(\Omega)} \leq C \Leb{N}(\Omega)\left( \|u\|_{L^{2}(\Omega)}^{\frac{1}{2}} \|\grad^2 u \|_{L^{2}(\Omega;\R^{N\times N})}^{\frac{1}{2}} + \|u\|_{L^{2}(\Omega)} \right),
$$
where $C >0$ is independent of $u$ and $\Omega$.
\end{prop}

\begin{prop}\label{prop:GagliardoNirenberg2}
Let $J \subset \R$ be an open bounded interval.
If $u \in L^1(J)$ and $u'' \in L^2(J)$ then $u \in H^2(J)$ and 
$$
\|u'\|_{L^{\frac{4}{3}}(J)} \leq C \left( \|u\|_{L^1(J)}^{\frac{1}{2}} \|u''\|_{L^2(J)}^{\frac{1}{2}} + \|u\|_{L^1(J)} \right),
$$
for some constant $C>0$.
\end{prop}

%
%
%
%
%
%
%
%
%
%
%
%
%
%
\subsection{Lifting inequalities}\label{sec:lifting}
We need to relate the $L^2$ norm of the hessian with its equivalent on the boundary, i.e., the $H^{\frac{1}{2}}$ fractional seminorm of the derivative of the trace.
In this section, we estimate the ratio between these two seminorms.
We start with an auxiliary lemma from \cite{diBenedetto}.
\begin{lem}\label{lem:fubini-like}
Let $1\leq p < \infty$, let $E \subset \R^N$ and $F \subset \R^m$ be measurable sets and let $u \in L^p(E\times F)$.
Then
$$
\left( \int_{F} \left( \int_E |u(x,y)| \, dx \right)^p \,dy \right)^{\frac{1}{p}}
  \leq \int_E \left( \int_{F} |u(x,y)|^p \,dy \right)^{\frac{1}{p}} \, dx.
$$
\end{lem}

\begin{prop}\label{prop:lifting}
Let $g \in H^{\frac{3}{2}}(0,R)$ and consider the triangle $T_R^+ := \{(x,y)\in\R^2: 0 < y < \frac{R}{2}, y < x < R-y \}$.
Then,
\begin{equation}\label{eq:lifting}
\frac{1}{8} \leq \zeta_{R,g} := \inf \left\{ \frac{\iint_{T_{R}^{+}} \bigl|D^{2} u(x,y)\bigr|^{2} \, dx \, dy}{ \int_{0}^{R} \int_{0}^{R} \frac{\bigl| g'(x) - g'(y) \bigr|^{2}}{|x-y|^{2}} \, dx \, dy} :
	 u \in H^{2}(T_R^+), Tu(\cdot,0)= g \text{ in } (0,R) \right\} \leq \frac{7}{16}.
\end{equation}
\end{prop}

\begin{proof}
We divide the proof in two steps.

\paragraph{Step 1:} Upper bound.

Define the diamond 
\begin{equation}\label{eq:diamond}
T_R := \left\{(x,y)\in\R^2: 0 \leq x \leq R, |y| \leq \min\{ x, R-x\} \right\}.
\end{equation}

Given a function $g \in H^{\frac{3}{2}}(0, R)$, we lift it to the diamond $T_R$ by
$$
u(x,y) := \frac{1}{2y} \int_{x-y}^{x+y} g(t) \, dt.
$$

We are only interested in the lifting on the positive part of the diamond, i.e., on the triangle $T_R^+$, but observe that $u(x,\cdot)$ is even, and we will take advantage of that fact for some estimates.
Since $g$ is continuous, one deduces immediately that $u$ is continuous and
$$
Tu'(x,0)
  = \lim_{y \to 0^+} \fpar{u}{x}(x,y)
  = \lim_{y \to 0^+} \frac{g(x+y)-g(x-y)}{2y}
  = g'(x).
$$

Moreover,
\begin{align*}
& \fpar{^2u}{x^2}(x,y) = \frac{g'(x+y)-g'(x-y)}{2y}, \\
& \fpar{^2u}{x \partial y}(x,y) = \frac{g'(x+y)+g'(x-y)}{2y} - \frac{g(x+y)-g(x-y)}{2y^2}, \\
& \fpar{^2u}{y^2}(x,y) = \frac{g'(x+y)-g'(x-y)}{2y} - \frac{g(x+y)+g(x-y)}{y^2}+ \frac{1}{y^3} \int_{x-y}^{x+y} g(t) \, dt.
\end{align*}

We can easily deduce that
$ \left\| \fpar{^2u}{x^2}\right\|_{L^2(T_R^+)}^2 = \frac14 |g'|_{H^{\frac{1}{2}}(0,R)}^2$, and note that
$$
\fpar{^2u}{x \partial y}(x,y) 
	= \frac{1}{2y^2}\int_{0}^{y}\left(g'(x+y) -g'(s+x) + g'(x-y) -g'(s+x-y)\right) \,ds.
$$
Use Hardy's inequality from Proposition \ref{prop:Hardy} to obtain
$$
\left\| \fpar{^2u}{x\partial y} \right\|_{L^2(T_R^+)}^2 \leq \frac{1}{16} |g'|_{H^{\frac{1}{2}}(0,R)}^2.
$$

Finally, notice that
$$
\fpar{^2u}{y^2}(x,y) 
 = \frac{1}{y^3} \int_{0}^{y} f_2(r;x,y)\, dr,
$$
where $\ds f_2(r;x,y):=\int_{r+x}^{x+y} \left( g'(x+y) - g'(s) \right) \, ds + \int_{x-y}^{r+(x-y)} \left( g'(s) - g'(x-y) \right) \, ds$.
Using Hardy's inequality in Proposition \ref{prop:Hardy} again, we deduce that
$$
\left\| \fpar{^2u}{y^2} \right\|_{L^2(T_R^+)}^2 \leq \frac{1}{16} |g'|_{H^{\frac{1}{2}}(0,R)}^2.
$$
We finally put the three estimates for the partial derivatives of $u$ of second order together to obtain
$$
\iint_{T_R^+} | \grad^2 u|^2 \, dx \, dy \leq \frac{7}{16} |g'|_{H^{\frac{1}{2}}(0,R)}^2.
$$

\paragraph{Step 2:} Lower Bound in \eqref{eq:lifting}

{\bf Case 1:} Assume that $v\in L^1(T_R^+;\R^2) \cap C^{\infty}(T_R^+;\R^2)$ is such that $\grad v \in L^2(T_R^+;\R^{2\times 2})$.

First it is easy to prove that
$$
\left|\frac{v(x+y,0) - v(x-y,0)}{2y}\right|^2
  \leq \frac{1}{2} \left( \int_0^1 \bigl| \grad v( x+y - ty, ty) \bigr| \, dt  + \int_0^1 \bigl| \grad v(x-y + ty, ty) \bigr| \, dt \right)^2.
$$

By estimating the right-hand side using Lemma \ref{lem:fubini-like} and Minkowski inequality, we obtain
$$
|v(\cdot,0)|_{H^{\frac{1}{2}}(0,R)}^2 \leq 8 \|\grad v\|_{L^2(T_R^+)}^2.
$$

{\bf Case 2:} Assume that $v\in L^1(T_R^+;\R^2)$ is such that $\grad v \in L^2(T_R^+;\R^{2\times 2})$.

First by reflection, extend the function to $v\in L^1(T_R;\R^2)$ with $\grad v \in L^2(T_R;\R^{2\times 2})$.
Let $\varphi_{\eps}$ be the standard mollifiers and consider $v_{\eps} := v \star \varphi_{\eps}$
defined in $T_R^{\eps} := \left\{ (x,y) \in T_R: d\bigl((x,y),\partial T_R\bigr) > \eps \right\}$. 
Then $v_{\eps} \to u$ in $L^1_{\loc}(A;\R^2)$, $\grad v_{\eps} \to \grad v$ in $L^2(A;\R^{2\times 2})$ and $v_{\eps}(\cdot ,0) \to Tv$ in $L^1\bigl(A \cap (\R \times \{0\});\R^2 \bigr)$ for any open set $A \ssubset T_R$. 
We can find a subsequence (not relabeled) such that $v_{\eps}(x ,0) \to Tv(x)$ for $\Leb{1}$-a.e. $x \in A \cap (\R \times \{0\})$.
Then by Case 1, we have
\begin{align*}
\int_{A \cap (\R \times \{0\})} \int_{A \cap (\R \times \{0\})} \left|\frac{Tv(x) - Tv(y)}{x-y}\right|^2 \, dx \, dy
	& \leq \liminf_{\eps \to 0^+} \int_{A \cap (\R \times \{0\})} \int_{A \cap (\R \times \{0\})} \left|\frac{v_{\eps}(x,0) - v_{\eps}(y,0)}{x-y}\right|^2 \, dx \, dy  \\
	& \leq 8 \lim_{\eps \to 0^+} \iint_{A \cap T_R^+} |\grad v_{\eps}|^2 \, dx \, dy
	= 8 \iint_{A \cap T_R^+} |\grad v|^2 \, dx \, dy.
\end{align*}

Let $A_n \subset A_{n+1} \ssubset T_R$ be such that $T_R = \bigcup A_n$. Then one deduces that
$$
\int_0^R \int_0^R \left|\frac{Tv(x) - Tv(y)}{x-y}\right|^2 \, dx \, dy \\
  \leq 8 \iint_{T_R^+} |\grad v|^2 \, dx \, dy.
$$

Apply this result to $v := \grad u$ to deduce
$$
\int_0^R \int_0^R \left|\frac{g(x) - g(y)}{x-y}\right|^2 \, dx \, dy
  \leq 8 \iint_{T_R^+} |\grad^2 u|^2 \, dx \, dy,
$$
which proves the lower bound in \eqref{eq:lifting}.
\end{proof}


\subsection{Slicing on BV}\label{sec:BV}

We use here the same notation as in section \ref{ssec:slicing}.

\begin{thm}[slicing of $BV$ functions]\label{thm:slicing_BV}
Let $u \in L^1(\Omega)$.
Then $u \in BV(\Omega)$ if and only if there exist $N$ linearly independent unit vectors $e_i$ such that $u^y_{e_i} \in BV(\Omega^y_{e_i})$ for $\Leb{N-1}$-a.e. $y \in \Omega_{e_i}$ and
$$
\int_{\Omega^y_{e_i}} |Du^y_{e_i}|(\Omega^y_{e_i}) \, dy < \infty
$$
for all $i=1,\ldots,N$.
\end{thm}

We state an immediate corollary of Theorem 1.24 from \cite{Giusti}.

\begin{prop}\label{prop:giusti}

Let $\Omega \subset \R^N$ be a bounded open Lipschitz set and let $E \subset \Omega$ be a set of finite perimeter.

Then there are sets $E_n \subset \Omega$ of class $C^2$ such that
\begin{equation} \label{eq:approxSuSv}
\begin{cases}
\Leb{N}(E \triangle E_n) \to 0, \\
\Haus{N-1}(\partial E \triangle  \partial E_n) \to 0.
\end{cases}
\end{equation}

\end{prop}

\begin{prop}[see section 5.10 in \cite{EG}]
\label{prop:Bv_bdy}

Let $A \subset \R^N$ be an open set, let $E \subset A$ be a Borel set, let $e$ be an arbitrary unit vector, and $E$ has finite perimeter in $A$.
Then $E^y_e$ has finite perimeter in $A^y_e$ and $\partial E^y_e \cap A^y_e = (\partial E \cap A)^y_e$, and
$$
\int_{A_e} \Haus{0}(\partial E^y_e \cap A^y_e) \, dy
	= \int_{A\partial E \cap A}  \left< \nu_E, e\right> \, d\Haus{N-1}.
$$
Conversely, $E$ has finite perimeter in $A$ if there exist $N$ linearly independent unit vectors $e_i$, $i=1,\ldots,N$ such that
$$
\int_{A_{e_i}} \Haus{0}(\partial E^y_{e_i} \cap A^y_{e_i}) \, dy < \infty
$$
for all $i=1,\ldots,N$.
\end{prop}


\subsection{Functions of bounded variation on a manifold} \label{ssec:manifold}

We consider several spaces of functions with domains $A \subset \R^{N}$ which are not open.  
Specifically, $A$ will be the boundary of an open and bounded set $\Omega$ of class $C^{2}$ and so it will be a compact Riemannian manifold (without boundary) of class $C^{2}$ and dimension $N-1$ in $\R^{N}$. Such a manifold is endowed with a unit normal field $\nu$ which is continuous and defined for every $x \in A$.
 In this section we give a brief definition of these spaces. For more details see \cite{AFP,EG,Hebey2}.

\paragraph{The space of integrable functions on a manifold.} 

Let $A \subset \R^N$ be a compact Riemannian manifold (without boundary) of class $C^1$ and dimension $N-1$ and define the restriction measure $\Haus{N-1}\restr{A}(E) := \Haus{N-1}(E \cap A)$.
A function $v$ is said to be integrable on $A$, and we write $v \in L^1(A;\Haus{N-1}\restr{A})$, if and only if $v$ is $\Haus{N-1}\restr{A}$-measurable and $\Haus{N-1}\restr{A}$-summable, precisely
\begin{align*}
& v^{-1}(J) \text{ is } \Haus{N-1}\restr{A}\text{-measurable for every open set } J \subset \R; \\
& \int_A |v(x)| \, d\Haus{N-1}(x) < \infty.
\end{align*}

\paragraph{The space of functions of bounded variation on a manifold.}

We give a short introduction to the space of functions of bounded variation on a manifold. For more details we refer to \cite{MirPal}.

Let $T^{\star}A$ be the cotangent bundle of $A$ and let $\Gamma(T^{\star}A)$ be the space of $1$-forms on $A$.
Then, given a function $v \in L^1(A)$, define the variation of $v$ by
\begin{equation}\label{eq:def_var_v}
|Dv|(A) := \sup \left\{ \int_A v \,\divf w \, d\Haus{N-1}: w \in \Gamma_c(T^{\star}A), |w|\leq 1\right\}.
\end{equation}
Then $v\in L^1(A)$ is said to be a function of bounded variation, i.e., $v \in BV(A)$ if $|Dv|(A)<\infty$.
Moreover, if $v=\chi_E$ for some set $E \subset A$, then $E$ has finite perimeter if and only if $v \in BV(A)$, and
$$
\Per_A(E) = |Dv|(A)=\Haus{N-2}(E\cap A) < \infty.
$$

\begin{prop}\label{prop:giusti:finiteper}

Let $\Omega \subset \R^N$ be an open bounded set of class $C^2$ and let $E \subset \partial \Omega$ be a set of finite perimeter with respect to $\Haus{N-2}$.
Then there are sets $E_n \subset \partial\Omega$ of class $C^2$ such that
$$
\begin{cases}
\Haus{N-1}(E \triangle E_n) \to 0, \\
\Haus{N-2}(\partial_{\partial\Omega}E \triangle\partial_{\partial\Omega}E_n) \to 0.
\end{cases}
$$
\end{prop}

%
%

\section{Characterization of constants}

\begin{lem}\label{lem:c_compact}

Assume that $V:\R\to[0,\infty)$ satisfies $(H^V_1)-(H^V_3)$.
Then the constant $\ul{c}$ defined in \eqref{eq:def_under_c} belongs to $(0,\infty)$.
\end{lem}

\begin{proof}

Assume by contradiction that $\ul{c}=0$.
Then there exist two sequences $\{f_n\} \subset  H^{\frac{3}{2}}_{\loc}(\R)$ and $\{R_n\} \subset (0,\infty)$ satisfying
\begin{align}
& f_n(-x)= \alpha, \quad f_n(x)=\beta \quad \text{ for all } x \geq R_n, \label{eq:prop1}\\
& \frac{1}{8}\int_{-R_n}^{R_n} \int_{-R_n}^{R_n} \frac{|f_n'(x) - f_n'(y)|^{2}}{|x-y|^{2}} \, dx \, dy + \int_{-R_n}^{R_n} V\bigl(f_n(x)\bigr) \, dx \xto{n\to\infty} 0. \label{eq:cto0}
\end{align}

Let $0 < 2\delta< \beta - \alpha$. Since $f_n(-R_n)=\alpha$, $f_n(R_n)=\beta$, and $f_n$ is continuous, there exists an interval $(S_n,T_n)$ such that
\begin{equation}\label{eq:prop2}
f_n(S_n) = \alpha + \delta < \beta - \delta = f_n(T_n),
\qquad f_n\bigl([S_n,T_n]\bigr) = [\alpha+\delta,\beta-\delta].
\end{equation}

By $(H^V_1)$ and the continuity of $V$ we have that $\ds C_{\delta} := \min_{z \in [\alpha+\delta,\beta-\delta]} V(z) > 0$.
Then by \eqref{eq:cto0}, 
$$
0 
	 = \lim_{n \to \infty} \int_{-R_n}^{R_n} V(f_n(x))\, dx 
	 \geq \lim_{n \to \infty} \int_{S_n}^{T_n} V(f_n(x))\, dx 
	 \geq \liminf_{n \to \infty} C_{\delta} (T_n-S_n),
$$
and so $T_n-S_n \to 0$.
For any $t \in [0,1]$, define
$$
g_n(t) := f_n\bigl(T_n t + S_n(1-t)\bigr).
$$

Then $g_n(0) = \alpha+\delta$ and $g(1)=\beta-\delta$. Changing variables in \eqref{eq:cto0} yields
$$
\int_{S_n}^{T_n} \int_{S_n}^{T_n} \frac{|f'_n(x)-f'_n(y)|^2}{|x-y|^2} \, dx \,dy
	= \frac{1}{(T_n-S_n)^2} \int_{0}^{1} \int_{0}^{1} \frac{|g'_n(s)-g'_n(t)|^2}{|s-t|^2} \, ds \,dt
	\to 0.
$$
This implies that $\left| \frac{g'_n}{T_n-S_n} \right|_{H^{\frac{1}{2}}(0,1)} \to 0$, and so, up to a subsequence (not relabeled), $\frac{g'_n}{T_n-S_n} \to \text{ constant}$ in $L^2(0,1)$.
Since $T_n-S_n \to 0$, this implies that $g'_n \to 0$ in $L^2(0,1)$.

On the other hand, 
$$
0 < \beta - \delta - (\alpha+\delta) = g_n(1) - g_n(0) = \int_0^1 g_n'(t) \, dt.
$$
Letting $n \to \infty$, we obtain a contradiction.
This shows that $\ul{c}>0$.

To prove that $\ul{c}<\infty$, take any function $f \in C^2$ such that $f(t) \leq \alpha$ for $t \leq -1$ and $f(t) = \beta$ for $t\geq 1$. It is easy to verify that the energy is finite.
\end{proof}

\begin{rem}
From the proof of the previous lemma, it follows that for every $0 < \delta < \frac{\beta-\alpha}{2}$, the constant
\begin{multline}\label{eq:cdelta}
c_{\delta} := \inf \biggl\{ \frac{1}{8}\int_{S}^{T} \int_{S}^{T} \frac{|f'(x) - f'(y)|^{2}}{|x-y|^{2}} \, dx \, dy 
+ \int_{S}^{T} V\bigl(f(x)\bigr) \, dx : f \in H_{\loc}^{\frac{3}{2}}(\R),\\
		f(S)= \alpha+\delta, f(T)=\beta-\delta, f_n\bigl((S_n,T_n)\bigr) = [\alpha+\delta,\beta-\delta], \text{ for some } S,T \in \R \biggr\}
\end{multline}
also belongs to $(0,\infty)$.
\end{rem}

\begin{lem}\label{lem:c}

Define the constant $\ol{c}$ as before by
\begin{multline*}
\ol{c} := \inf \left\{ \frac{7}{16}\int_{-\infty}^{\infty}\int_{-\infty}^{\infty} \frac{|f'(x) - f'(y)|^{2}}{|x-y|^{2}} \, dx \, dy 
+ \int_{-\infty}^{\infty} V\bigl(f(x)\bigr) \, dx : \right. \\
	 \left. f \in H_{\loc}^{\frac{3}{2}}(\R), f(-t)= \alpha, f(t)=\beta, \quad \text{ for all } t \geq R, R > 0 \right\},
\end{multline*}
where $V$ satisfies the properties of Theorem \ref{thm:compactness}.

Then $\ol{c} \in (0,\infty)$.

\end{lem}

\begin{prop}

Under the conditions of Theorem \ref{thm:compactness}, $\ul{c} = c_{\star}$,
where $c_{\star}$ is defined by
\begin{multline*}
c_{\star}	:= \inf\left\{ \frac{3}{2^{\frac{5}{3}}} \left(\int_{-1}^1 \int_{-1}^1\frac{|g'(x) - g'(y)|^{2}}{|x-y|^{2}} \, dx \, dy\right)^{\frac{1}{3}} \left(\int_{-1}^{1} V\bigl(g(x)\bigr) \, dx\right)^{\frac{2}{3}} : \right. \\
	 \left. 	 g \in H_{\rm loc}^{\frac{3}{2}}(\R), g' \in H^{\frac{1}{2}}(\R), g(-t)= \alpha, g(t)=\beta \text{ for all } t \geq 1\right\}.
\end{multline*}

\end{prop}

\begin{proof}

First we prove that $\ul{c} \geq c_{\star}$.
Let $\eta >0$, and $f \in H^{\frac{3}{2}}_{\loc}(\R), R>0$ be such that
\begin{align*}
& f' \in H^{\frac{1}{2}}(\R),
	\qquad f(-t)= \alpha, f(t)=\beta, \text{ for all } t \geq R, \\
& \frac{1}{8}\int_{-R}^{R}\int_{-R}^{R} \frac{|f'(x) - f'(y)|^{2}}{|x-y|^{2}} \, dx \, dy 
+ \int_{-R}^{R} V\bigl(f(x)\bigr) \, dx \leq \ul{c} + \eta.
\end{align*}

Then
\begin{align*}
\ul{c}+\eta
	& \geq \frac{1}{8R^2}\int_{-1}^{1}\int_{-1}^{1} \frac{|(f(Rx))' - (f(Ry))'|^{2}}{|x-y|^{2}} \, dx \, dy 
+ R \int_{-1}^{1} V\bigl(f(Rx)\bigr) \, dx \\
	& \geq \frac{1}{8S_R^2}\int_{-1}^{1}\int_{-1}^{1} \frac{|g_R'(x) - g_R'(y)|^{2}}{|x-y|^{2}} \, dx \, dy 
+ S_R \int_{-1}^{1} V\bigl(g_R(x)\bigr) \, dx 
	 \geq c_{\star}
\end{align*}
where $g_R(x) = f(Rx)$ which is admissible for $c_{\star}$, and 
\begin{align*}
S_R 
	& = \argmin_{S>0} \left[\frac{1}{8S^2}\int_{-1}^{1}\int_{-1}^{1} \frac{|g_R'(x) - g_R'(y)|^{2}}{|x-y|^{2}} \, dx \, dy + S \int_{-1}^{1} V\bigl(g_R(x)\bigr) \, dx \right]\\
	& = \left(\frac{\ds \int_{-1}^{1}\int_{-1}^{1} \frac{|g_R'(x) - g_R'(y)|^{2}}{|x-y|^{2}} \, dx \, dy}{\ds 4 \int_{-1}^{1} V\bigl(g_R(x)\bigr) \, dx} \right)^{\frac{1}{3}}.
\end{align*}
Let $\eta \to 0^+$ to deduce that $\ul{c} \geq c_{\star}$.
The converse inequality follows trivially from following the first part of the proof from the end to the beginning.
\end{proof}

\begin{prop}

Under the conditions of Theorem \ref{thm:compactness}, $\ol{c} = c^{\star}$,
where $c^{\star}$ is defined by
\begin{multline*}
c^{\star} := \inf\left\{ \frac{3 \cdot 7^{\frac{1}{3}}}{4} \left(\int_{-\infty}^{\infty} \int_{-\infty}^{\infty} \frac{|g'(x) - g'(y)|^{2}}{|x-y|^{2}} \, dx \, dy\right)^{\frac{1}{3}} \left(\int_{-\infty}^{\infty} V\bigl(g(x)\bigr) \, dx\right)^{\frac{2}{3}} : \right. \\
	 \left. 	 g \in H_{\rm loc}^{\frac{3}{2}}(\R), g(-t)= \alpha, g(t)=\beta, \text{ for all } t \geq 1\right\}.
\end{multline*}

\end{prop}

\section{Two auxiliary one-dimensional problems}\label{sec:critical1D}


\subsection{Compactness for $F_{\eps}$}\label{ssec:bulkF1D_compc}

\begin{thm}\label{thm:compactness_F1D}
Assume that $W:\R \to [0,\infty)$ satisfies $(H^W_1)-(H^W_2)$.
Let $I \subset \R$ be an open, bounded interval, let $\{\eps_n\}$ be a positive sequence converging to $0$, and let $\{u_n\} \subset H^{2}(I)$ be such that
\begin{equation}\label{eq:energyfiniteF}
\sup_{n} F_{\eps_n}(u_n;I) < \infty.
\end{equation}

Then there exist a subsequence (not relabeled) of $\{u_n\}$ and a function $u \in BV\bigl(I;\{a,b\}\bigr)$ such that $u_n \to u$ in $L^2(I)$.
\end{thm}

\begin{proof}

Given a sequence $\{u_n\} \subset H^2(I)$ satisfying \eqref{eq:energyfiniteF}, by the compactness result in \cite{FM} and $(H^W_2)$, we obtain a subsequence $\{u_{n}\}$ (not relabeled) and a function $u \in BV\bigl(I;\{a,b\}\bigr)$ such that $u_n \to u$ in $L^2(I)$.
\end{proof}


\subsection{Lower bound for $F_{\eps}$}\label{ssec:bulkF1D_LB}

\begin{thm}[Lower bound estimate for $F_{\eps}$]\label{thm:interior1D}

Let $I \subset \R$ be an open and bounded interval and let $W:\R\to[0,\infty)$ satisfy $(H^W_1)-(H^W_2)$.
Let $u \in BV\bigl(I;\{a,b\}\bigr)$, let $v \in BV\bigl(\partial I; \{\alpha,\beta\}\bigr)$, and let $\{u_{\eps}\} \subset H^2(I)$ be such that $\ds \sup_{\eps} F_{\eps}(u_{\eps};I) =: C < \infty$, $u_{\eps} \to u$ in $L^2(I)$ and $Tu_{\eps} \rightarrow v$ in $\Haus{0}(\partial I)$.
Then
$$
\liminf_{\eps \to 0^+} F_{\eps}(u_{\eps};I) 
   \geq m \Haus{0}(S(u)) + \int_{\partial I} \sigma\bigl(Tu(x),v(x)\bigr) \, d\Haus{0}(x),
$$
where $m$ and $\sigma$ are defined in \eqref{eq:def_m} and \eqref{eq:def_sigma}, respectively.
\end{thm}

\def\awaybdy{Step 1 of the proof of Theorem \ref{thm:interior1D}}

\begin{proof}

Passing to a subsequence (not relabeled), we can assume that
$$
\liminf_{\eps \to 0^+} F_{\eps}(u_{\eps};I)
	= \lim_{\eps\to 0^+} F_{\eps}(u_{\eps};I).
$$

Since $u_{\eps} \to u$ in $L^1(I)$ and $\|W(u_{\eps})\|_{L^1(I)} \leq C \eps$, by the growth condition $(H^W_2)$, we have that, up to a subsequence (not relabeled), $u_{\eps} \to u$ in $L^2(I)$, and $\sup_{\eps} \|u_{\eps}\|_{L^2(I)} \leq C$.

In turn, by Proposition \ref{prop:gagliardo-nirenberg} and the fact that $\|u_{\eps}''\|_{L^2(I)} \leq C \eps^{-\frac{3}{2}}$, we deduce that $\|u_{\eps}'\|_{L^2(I)} \leq C \eps^{-\frac{3}{4}}$, and so
$$
\lim_{\eps \to 0^+} \int_I |\eps u'_{\eps}(x)|^2 \, dx = 0.
$$

Thus, up to a subsequence (not relabeled), we may assume that
\begin{equation}\label{eq:pointwise}
\eps u_{\eps}'(x) \to 0,
\quad \text{ and }
\quad 
u_{\eps}(x) \to u(x)
\end{equation}
for $\Leb{1}$-a.e. $x \in I$.
Since $u \in BV(I;\{a,b\})$, its jump set is a finite set, so we can write $S(u) := \{s_1, \ldots, s_{\ell}\}$,
where
$$
s_0 := \inf I < s_1 < \cdots < s_{\ell} < s_{\ell+1} := \sup I.
$$
Fix $0< \eta < \frac{\beta-\alpha}{2}$, and $0< \delta_{0} := \frac{1}{2}\min \left\{ s_{i+1} - s_{i}: i=0, \ldots, \ell \right\}$.
Using \eqref{eq:pointwise}, for every $i=1, \ldots, \ell$, we may find $x_i^{\pm} \in (s_i - \delta_0, s_i+\delta_0)$ such that 
\begin{equation}\label{eq:xieps}
|u_{\eps}(x_i^+) - b| < \eta,
\qquad
|u_{\eps}(x_i^-) - a| < \eta,
\quad \text{ and }
\quad 
|\eps u'_{\eps}(x_i^{\pm})| < \eta.
\end{equation}

Moreover, since $u$ is a constant in $(s_0,s_1)$, we assume that $u(x) \equiv a$ in this interval (the case $u(x) \equiv b$ is analogous), and we have that $Tu(s_0)=a$.
Using \eqref{eq:pointwise} once more, we may find a point $x_0^+$ such that
\begin{equation}\label{eq:x0+}
|u_{\eps}(x_0^+) - a| < \eta,
\quad \text{ and }
\quad 
|\eps u'_{\eps}(x_0^{+})| < \eta
\end{equation}
for all $\eps$ sufficiently small.

On the other hand, $Tu_{\eps}(s_0) \to v(s_0)$, and so $|Tu_{\eps}(s_0) - v(s_0)| < \frac{\eta}{2}$ for all $\eps$ sufficiently small.
Since $\lim_{x \to s_0^+} u_{\eps}(x) = Tu_{\eps}(s_0)$, there is $0<\rho_{\eps} < x_0^+<s_0$ such that
$|u_{\eps}(x) - v(s_0)| < \eta$ for all $x \in (s_0,s_0+\rho_{\eps})$.

There are now two cases.
If $u_{\eps}(x_{\eps}) = v(s_0)$ for some $x_{\eps} \in (s_0,x_0^+)$, then take $x_{0,\eps}^- := x_{\eps}$.
If $u_{\eps}(x_{\eps}) \neq v(s_0)$ for all $x_{\eps} \in (s_0,x_0^+)$, then we claim that there exists $x_{0,\eps}^- \in (s_0,x_0^+)$ such that
$$
u_{\eps}'(x_{0,\eps}^-) \bigl(u_{\eps}(x_{0,\eps}^-) - v(s_0) \bigr) > 0.
$$

Indeed, if say $u_{\eps}(x) > v(s_0)$ in $(s_0,x_0^+)$, then for $\eta>0$ such that $|v(s_0)-a| > 2\eta$, we have that
$$
\bigl| u_{\eps}(x_{0}^+) - v(s_0) \bigr|\geq |v(s_0) - a| - |u_{\eps}(x_0^+)-a| > \eta,
$$
and so there exists a first point $x_{\eps} \in (s_0,x_0^+)$ such that
$$
u_{\eps}(x_{\eps}) = v(s_0) + \eta.
$$
Hence, by the mean value theorem, there is $x_{0,\eps}^- \in (s_0,x_{\eps})$ such that 
$$
u'_{\eps}(x_{0,\eps}) = \frac{u_{\eps}(x_{\eps})-Tu_{\eps}(s_0)}{x_{\eps}-s_0} > \frac{v(s_0) + \eta - v(s_0)-\frac{\eta}{2}}{x_{\eps}-s_0} > 0.
$$

Thus, we have found $x_{0,\eps} \in (s_0,x_0^+)$ such that
\begin{equation}\label{eq:x0-}
|u_{\eps}(x_{0,\eps}^-) - v(s_0)| < \eta
\quad \text{ and }
\quad 
u_{\eps}'(x_{0,\eps}^-) \bigl(u_{\eps}(x_{0,\eps}^-) - v(s_0) \bigr) \geq 0,
\end{equation}
for all $\eps$ sufficiently small.
For simplicity of notation, we write $x_0^- := x_{0,\eps}^-$ and $x_{\ell+1}^- := x_{\ell+1,\eps}^-$.
From the facts that the intervals $[x_i^-,x_i^+]$ are disjoint for $i=0,\ldots, \ell+1$, and that $W$ is nonnegative, we have that
\begin{equation}\label{eq:Fsum}
F_{\eps}(u_{\eps};I)
	\geq \sum_{i=0}^{\ell+1} 
	\int_{x_i^-}^{x_i^+} \left( \eps^3 \bigl|u''_{\eps}(x)\bigr|^2 + \frac{1}{\eps} W\bigl(u_{\eps}(x)\bigr)\right) \, dx.
\end{equation}

We claim that
\begin{equation}\label{eq:F_m}
\int_{x_i^-}^{x_i^+} \left( \eps^3 \bigl|u''_{\eps}(x)\bigr|^2 + \frac{1}{\eps} W\bigl(u_{\eps}(x)\bigr)\right) \, dx
	\geq m\ell -O(\eta) - O(\eps)
\end{equation}
for all $i=1,\ldots,\ell$, that
\begin{equation}\label{eq:F_sigma0}
\int_{x_0^-}^{x_0^+} \left( \eps^3 \bigl|u''_{\eps}(x)\bigr|^2 + \frac{1}{\eps} W\bigl(u_{\eps}(x)\bigr)\right) \, dx
	\geq \sigma\bigl(Tu(s_0),v(s_0)\bigr) -O(\eta) - O(\eps),
\end{equation}
and that
\begin{equation}\label{eq:F_sigmaell}
\int_{x_{\ell+1}^-}^{x_{\ell+1}^+} \left( \eps^3 \bigl|u''_{\eps}(x)\bigr|^2 + \frac{1}{\eps} W\bigl(u_{\eps}(x)\bigr)\right) \, dx
	\geq \sigma\bigl(Tu(s_{\ell+1}),v(s_{\ell+1})\bigr) -O(\eta) - O(\eps).
\end{equation}

If \eqref{eq:F_m}, \eqref{eq:F_sigma0}, and \eqref{eq:F_sigmaell} hold, then from \eqref{eq:Fsum} we deduce that
$$
\liminf_{\eps \to 0^+} F_{\eps}(u_{\eps};I)
	\geq m\ell + \sigma\bigl(Tu(s_0),v(s_0)\bigr) + \sigma\bigl(Tu(s_{\ell+1}),v(s_{\ell+1})\bigr) - O(\eta).
$$
Letting $\eta \to 0^+$ yields
$$
\liminf_{\eps \to 0^+} F_{\eps}(u_{\eps};I)
	\geq m\ell + \int_{\partial I} \sigma\bigl(Tu(x),v(x)\bigr) \, d\Haus{0}(x).
$$

The remaining of the proof is devoted to the proof of \eqref{eq:F_m}, \eqref{eq:F_sigma0}, and \eqref{eq:F_sigmaell}.

\paragraph{Step 1. }  Proof of \eqref{eq:F_m}.

Define the functions
\begin{align}
& G(w,z) := \inf \biggl\{ \int_{0}^{1} W\bigl(g(x)\bigr) + \bigl|g''(x)\bigr|^{2} \, dt: g \in C^{2}\bigl([0,1]; \R\bigr), g(0) = w, g(1) = b, g'(0)=z, g'(1)=0  \biggr\}, \label{eq:defG}\\
& H(w,z) := \inf \biggl\{ \int_{0}^{1} W\bigl(h(x)\bigr) + \bigl|h''(x)\bigr|^{2} \, dt: h \in C^{2}\bigl([0,1]; \R\bigr), h(0) = a, h(1) = w, h'(0)=0, h'(1)=z  \biggr\}. \label{eq:defH}
\end{align}

Note that, considering third-order polynomials, one deduces that these functions satisfy
\begin{equation}\label{eq:limGH}
\lim_{(w,z) \rightarrow (b,0)} G(w,z) = 0, \qquad
\lim_{(w,z) \rightarrow (a, 0)} H(w,z) = 0.
\end{equation}

From \eqref{eq:xieps}, for $\eps$ sufficiently small, we have $G\bigl( u_{\eps}( x_{i}^{+}), \eps u_{\eps}'( x_{i}^{+}) \bigr), H\bigl( u_{\eps}( x_{i}^{-}), \eps u_{\eps}'( x_{i}^{-}) \bigr) \leq \eta$.
By \eqref{eq:defG} and \eqref{eq:defH}, we can find admissible functions $\widehat{g_{i}}$ and $\widehat{h_{i}}$ for $G\bigl( u_{\eps}( x_{i}^{+}), \eps u_{\eps}'( x_{i}^{+}) \bigr)$ and $H\bigl( u_{\eps}( x_{i}^{-}), \eps u_{\eps}'( x_{i}^{-}) \bigr)$, respectively, such that
\begin{align}
& \int_{0}^{1} \bigl|\widehat{g_{i}}''(x)\bigr|^{2} + W\bigl(\widehat{g_{i}}(x)\bigr) \, dx 
	\leq G\bigl( u_{\eps}( x_{i}^{+} ), \eps u_{\eps}'( x_{i}^{+} ) \bigr) + \eta \leq 2\eta,  \label{eq:admG} \\
& \int_{0}^{1} \bigl|\widehat{h_{i}}''(x)\bigr|^{2} + W\bigl(\widehat{h_{i}}(x)\bigr) \, dx 
	\leq H\bigl( u_{\eps}( x_{i}^{-}), \eps u_{\eps}'( x_{i}^{-} ) \bigr) + \eta \leq 2\eta. \label{eq:admH}
\end{align}

We now rescale and translate these functions, precisely,
\begin{equation*}
g_{i}(x) := \widehat{g_{i}}\left( x - \frac{x_{i}^{+}}{\eps} \right), \qquad
h_{i}(x) := \widehat{h_{i}}\left( x - \frac{x_{i}^{-}}{\eps} + 1 \right).
\end{equation*}

Define
$$
w_{\eps,i}(x) := \begin{cases}
	b 				& \text{if } x \geq \frac{x_{i}^{+}}{\eps} + 1, \\
	g_{i}(t)			& \text{if } \frac{x_{i}^{+}}{\eps} \leq x \leq \frac{x_{i}^{+}}{\eps} + 1, \\
	u_{\eps}(\eps x)		& \text{if } \frac{x_{i}^{-}}{\eps} \leq x \leq \frac{x_{i}^{+}}{\eps}, \\
	h_{i}(t)			& \text{if } \frac{x_{i}^{-}}{\eps} - 1 \leq x \leq \frac{x_{i}^{-}}{\eps}, \\
	a				& \text{if } x \leq \frac{x_{i}^{-}}{\eps} - 1.
\end{cases}
$$

By construction $w_{\eps,i} \in H^2\bigl(\frac{x_{i}^{-}}{\eps} - 1, \frac{x_{i}^{+}}{\eps}+1 \bigr)$ and $w_{\eps,i}$ is admissible for the constant $m$ given in \eqref{eq:def_m}.
Hence for all $\eps$ sufficiently small,
\begin{multline*}
\int_{x_i^-}^{x_i^+} \left( \eps^3 \bigl|u''_{\eps}(x)\bigr|^2 + \frac{1}{\eps} W\bigl(u_{\eps}(x)\bigr)\right) \, dx 
 = \int_{\frac{x_i^-}{\eps}}^{\frac{x_i^+}{\eps}} \left( \bigl|w''_{\eps,i}(y)\bigr|^2 + W\bigl(w_{\eps,i}(y)\bigr)\right) \, dy \\
	= \int_{\frac{x_i^-}{\eps}-1}^{\frac{x_i^+}{\eps}+1} \left( \bigl|w''_{\eps,i}(y)\bigr|^2 + W\bigl(w_{\eps,i}(y)\bigr)\right) \, dy
		- \int_{0}^1 \left( \bigl|\widehat{g_{i}}''(y)\bigr|^2 + W\bigl(\widehat{g_{i}}(y)\bigr)\right) \, dy \\
	- \int_{0}^1 \left( \bigl|\widehat{h_{i}}''(y)\bigr|^2 + W\bigl(\widehat{h_{i}}(y)\bigr)\right) \, dy
	 \geq m\ell-4\eta,
\end{multline*}
where we used \eqref{eq:admG} and \eqref{eq:admH}.

\paragraph{Step 2. } Proof of \eqref{eq:F_sigma0}.

Define the functions
\begin{align}
& L(w,z) := \inf \biggl\{ \int_{0}^{1} W\bigl(f(x)\bigr) + \bigl|f''(x)\bigr|^{2} \, dt: f \in C^{2}\bigl([0,1]; \R\bigr),
	f(0) = w, f(1) = a, f'(0)=z, f'(1)=0  \biggr\}, \label{eq:defL} \\
& J(w,z) := \inf \biggl\{ \int_{0}^{r} W\bigl(j(x)\bigr) + \bigl|j''(x)\bigr|^{2} \, dt: 	j \in C^{2}\bigl([0,r]; \R\bigr),
	 j(0) = v(0), j(r) = w, j'(r)=z, \text{ for some } r>0  \biggr\}. \label{eq:defJ}
\end{align}

Analogously to \eqref{eq:defG}, $\ds \lim_{(w,z) \rightarrow (a,0)} L(w,z) = 0$, and from \eqref{eq:x0+}, for all $\eps$ sufficiently small we have $L\bigl( u_{\eps}( x_{0}^{+}), \eps u_{\eps}'( x_{0}^{+}) \bigr) \leq \eta$.
Hence we can find an admissible function $\widehat{f_{0}}$ for $L\bigl( u_{\eps}( x_{0}^{+}), \eps u_{\eps}'( x_{0}^{+}) \bigr)$ such that
\begin{equation}\label{eq:f_choose}
\int_{0}^{1} W\bigl(\widehat{f_{0}}(x)\bigr) + \bigl|\widehat{f_{0}}''(x)\bigr|^{2} \, dx 
	\leq L\bigl( u_{\eps}( x_{0}^{+} ), \eps u_{\eps}'( x_{0}^{+} ) \bigr) + \eta \leq 2\eta.
\end{equation}

We now prove that
\begin{equation}\label{eq:Jto0}
\lim_{\substack{w\to v(s_0) \\ z(w-v(s_0))\geq 0}} J(w,z) = 0.
\end{equation}

Fix $\eta >0$ and let $w, z \in \R$ be such that $|w-v(0)|<\eta$ and $z(w-v(0))\geq 0$.
If $|z| \leq \sqrt{\eta}$, then take $j(x):=w+z(x-r)+\frac{(v(0)-w)+rz}{r^2}(x-r)^2$, which is admissible for $J(w,z)$, to obtain 
$$
J(w,z) 
	\leq C \left[ r + \frac{(v(0)-w+rz)^2}{r^3} \right]
	\leq C \left[ r + \frac{\eta^2}{r^3} + \frac{\eta}{r} \right].
$$

Choosing $r = \sqrt{\eta}$, we deduce that $J(w,z) = O(\sqrt{\eta})$.

If $|z| \geq \sqrt{\eta}$, then let $r := \frac{w-v(0)}{z} > 0$, which satisfies $0< r < \sqrt{\eta}$.
Then, let $j(x) := w + z (x-r)$, which is admissible for $J(w,z)$ because $j(0) = w - zr = v(0)$ and $j''(x) = 0$, so $J(w,z) = O(\sqrt{\eta})$.
This proves \eqref{eq:Jto0}.

By \eqref{eq:Jto0} and \eqref{eq:x0-}, we may find a function $j$ admissible for $J\bigl(u_{\eps}(x_{0,\eps}^{-}), \eps u_{\eps}'(x_{0,\eps}^{-})\bigr)$ such that
\begin{equation}\label{eq:j_choose}
\int_{0}^{r} W\bigl(j(t)\bigr) + \bigl|j''(t)\bigr|^{2} \, dt 
	\leq J\bigl( u_{\eps} (x_{0,\eps}^{-}), \eps u_{\eps}'( x_{0,\eps}^{-}) \bigr) + \eta \leq 2\eta,
\end{equation}
for some $r = r(\eta) > 0$.

Set $f_{0}(x) := \widehat{f_{0}}\left( x - r - \frac{x_{0}^{+}-x_{0,\eps}^{-}}{\eps} \right)$, and define
$$
w_{\eps,0}(x) := \begin{cases}
	a				& \text{if } x \geq 1+r+ \frac{x_{0}^{+}-x_{0,\eps}^{-}}{\eps}, \\
	f_{0}(x)			& \text{if } r+ \frac{x_{0}^{+}-x_{0,\eps}^{-}}{\eps} \leq x \leq 1+r+ \frac{x_{0}^{+}-x_{0,\eps}^{-}}{\eps}, \\
	u_{\eps}\left(\eps (x-r) + x_{0,\eps}^{-} \right)	
					& \text{if } r \leq x \leq r+ \frac{x_{0}^{+}-x_{0,\eps}^{-}}{\eps}, \\
	j(x)				& \text{if } 0 \leq x \leq r.
\end{cases}
$$

By construction $w_{\eps,0}$ belongs to $H^2_{\loc}(0,\infty)$ and is admissible for $\sigma\bigl(Tu(s_0),v(s_0)\bigr)$ as defined in \eqref{eq:def_sigma}.
Hence for all $\eps$ sufficiently small, we have that
\begin{multline*}
\int_{x_0^-}^{x_0^+} \left( \eps^3 \bigl|u''_{\eps}(x)\bigr|^2 + \frac{1}{\eps} W\bigl(u_{\eps}(x)\bigr)\right) \, dx 
	 = \int_{r}^{r+\frac{x_0^+-x_0^-}{\eps}} \left( \bigl|w''_{\eps,0}(y)\bigr|^2 + W\bigl(w_{\eps,0}(y)\bigr)\right) \, dy \\
	 = \int_{0}^{1+r+\frac{x_0^+-x_0^-}{\eps}} \left( \bigl|w''_{\eps,0}(y)\bigr|^2 + W\bigl(w_{\eps,0}(y)\bigr)\right) \, dy
		- \int_{0}^1 \left( \bigl|\widehat{f_{0}}''(y)\bigr|^2 +  W\bigl(\widehat{f_{0}}(y)\bigr)\right) \, dy \\
	- \int_{0}^r \left( \bigl|j''(y)\bigr|^2 + W\bigl(j(y)\bigr)\right) \, dy
	 \geq \sigma\bigl(Tu(s_0),v(s_0)\bigr)-4\eta,
\end{multline*}
where we have used \eqref{eq:f_choose} and \eqref{eq:j_choose}.
This proves \eqref{eq:F_sigma0}.
The proof of \eqref{eq:F_sigmaell} is analogous.
\end{proof}


\subsection{Compactness for $G_{\eps}$}\label{ssec:bdyG1D_compc}

To prove compactness for the functional $G_{\eps}$ defined in \eqref{eq:bdyG}, we begin with an auxiliary result.

\begin{lem} \label{lem:thetaBV}
Let $\theta \in L^{1}(J;[0,1])$ and let
$$
X:= \left\{ \ds x \in J: \mint_{J \cap B(x;\delta)} \theta(s) \, ds \in (0,1) \text{ for all } 0< \delta < \delta_0 \text{, for some } \delta_0=\delta_0(x)>0\right\}
$$
be a finite set.
Then $\theta \in BV(J;\{0,1\})$ and $S(\theta) \subset X$.

\end{lem}

\begin{thm}[compactness for $G_{\eps}$]\label{thm:compactness_G1D}

Assume that $V:\R \to [0,\infty)$ satisfies $(H^V_1)-(H^V_3)$.
Let $J \subset \R$ be an open, bounded interval, let $\{\eps_n\}$ be such that $\eps_n \lambda_n^{\frac23} \to L \in(0,\infty)$, and let $\{v_n\} \subset H^{\frac{3}{2}}(J)$ be such that 
\begin{equation}\label{eq:energyfiniteG}
\sup_{n} G_{\eps_n}(v_n;J) < \infty.
\end{equation}

Then there exist a subsequence (not relabeled) of $\{v_n\}$ and a function $v \in BV\bigl(J;\{\alpha,\beta\}\bigr)$ such that $v_n \to v$ in $L^2(J)$.

\end{thm}

\begin{proof}

Since $\lambda_n \to \infty$, by \eqref{eq:energyfiniteG} we have that
$$
C_1 := \sup_n \int_J V(v_n)\, dx < \infty.
$$

By condition $(H^V_3)$ and the fact that $J$ is bounded, we have that
$$
\frac{1}{C} \Leb{1}(J) + C \int_J |v_n|^2 \, dx \leq \int_J V(v_n) \, dx \leq C_1,
$$
and so $\{v_n\}$ is bounded in $L^2(J)$.
Thus by the fundamental theorem of Young measures (for a comprehensive exposition on Young measures, see \cite{Tartar_YM, Ball_YM,MullerLM,FL}), there exists a subsequence (not relabeled) generating a Young measure $\{\nu_x\}_{x\in J}$.
Letting $f(z) := \min \bigl\{ V(z),1\bigr\}$, since $\lambda_n \to \infty$, we have that
$$
0 = \lim_n \int_J f(v_n) \, dx = \int_J \int_{\R} f(z) \, d\nu_x(z) \, dx.
$$
Since $f(z)=0$ if and only if $z \in \{\alpha,\beta\}$, we have that for $\Leb{1}$-a.e. $x \in J$,
\begin{equation}\label{eq:YMnu}
\nu_x = \theta(x) \delta_{\alpha} + \bigl( 1-\theta(x)\bigr) \delta_{\beta}
\end{equation}
for some $\theta \in L^{\infty}\bigl(J;[0,1]\bigr)$.
Define
\begin{equation}\label{eq:X}
X := \left\{ x \in J: \mint_{B(x;\delta)} \theta(s)\, ds \in (0,1) \text{ for all } 0 < \delta < \delta_0 \text{, for some } \delta_0=\delta_0(x)>0 \right\}.
\end{equation}

We claim that $X$ is finite.
To establish this, let $s_1, \ldots, s_{\ell}$ be distinct points of $X$ and let $0 < d_0 < \frac{1}{2} \min \{ |s_i-s_j|:  i\neq j, i,j=1,\ldots,\ell\}$.
Since $s_i \in X$, we may find $d_i>0$ so small that $d_i \leq d_0$ and
\begin{equation}\label{eq:thetasi}
\mint_{B(s_i;d_i)} \theta(s) \, ds > 0,
\qquad 
\mint_{B(s_i;d_i)} \bigl( 1 - \theta(s)\bigr) \, ds > 0.
\end{equation}

Define $d := \min \{ d_1, \ldots, d_{\ell} \}$.
Let $0<\eta< \frac{\beta-\alpha}{2}$, let $\varphi_{\eta} \in C^{\infty}_c\bigl(\R;[0,1]\bigr)$ be such that $\supp \varphi_{\eta} \subset B(\alpha;\eta)$ and $\varphi_{\eta}(\alpha)=1$, and let $\gamma_{\eta} \in C^{\infty}_c\bigl(\R;[0,1]\bigr)$ be such that $\supp \gamma_{\eta} \subset B(\beta;\eta)$ and $\gamma_{\eta}(\alpha)=1$.

Using the fundamental theorem of Young measures with 
$$
f(x,z):= \chi_{B(s_i;d)}(x)\varphi_{\eta}(z),
$$
we obtain
\begin{equation}
\lim_{n \to \infty} \int_{B(s_i;d_i)} \varphi_{\eta}\bigl( v_n(x)\bigr) \, dx 
	 = \int_{\R} \int_{\R} f(x,z) \, d\nu_x(z) \, dx
	 = \int_{B(s_i;d_i)} \theta(x) \, dx > 0, \label{eq:theta1}.
\end{equation}

Similarly,
\begin{equation}\label{eq:theta2}
\lim_{n \to \infty} \int_{B(s_i;d_i)} \gamma_{\eta}\bigl( v_n(x)\bigr) \, dx 
	= \int_{B(s_i;d_i)} \bigl( 1 - \theta(x)\bigr) \, dx > 0.
\end{equation}

In view of \eqref{eq:theta1} and \eqref{eq:theta2}, we may find $x_{n, i}^{\pm} \in (s_i-d,s_i+d)$ such that
$$
|v_{n}(x_{n,i}^{-})-\alpha| < \eta,
\quad \text{ and } \quad |v_{n}(x_{n,i}^{+})-\beta| < \eta.
$$

Let $w_{n}(x) := v_{n}\left(\eps_n \lambda_{n}^{-\frac{1}{3}} x\right)$, which is admissible for the constant $c_{\eta}$ defined in \eqref{eq:cdelta}.
Then by \eqref{eq:energyfiniteG},
\begin{align*}
\infty > C \geq \liminf_{n} G_{\eps_n}(v_{n};J) 
	& \geq \liminf_{n} \sum_{i=1}^{\ell} G_{\eps_n}\bigl(v_{n};(x_{n,i}^-,x_{n,i}^+)\bigr) \\
	& \geq \liminf_{n} \sum_{i=1}^{\ell} \eps_n \lambda_{n}^{\frac{2}{3}} G_{1}\left(w_{n};\left({\ts \frac{x_{n,i}^-}{\eps_n\lambda_{n}^{-\frac{1}{3}}},\frac{x_{n,i}^+}{\eps_n\lambda_{n}^{-\frac{1}{3}}}}\right) \right)
	 \geq c_{\eta} L \ell .
\end{align*}

We conclude that
$$
\Haus{0}(X) \leq \frac{C}{c_{\eta} L} < \infty.
$$

By Lemma \ref{lem:thetaBV}, this implies that $\theta \in BV\bigl(J;\{0,1\}\bigr)$.
In particular, we may write $\theta = \chi_{E}$, and so $\nu_x = \delta_{v(x)}$, where 
$$
v(x) := \begin{cases}
	\alpha 	& \text{ if } x \in E, \\
	\beta 	& \text{ if } x \in J \bs E.
\end{cases}
$$

It follows that $\{v_n\}$ converges in measure to $v$.
By condition $(H^V_2)$, there are $C, T>0$ such that $V(z) \geq C |z|^2$ for all $|z| \geq T$,
and so
$$
\int_{E \cap \{|v_{n}|\geq T\}} |v_{n}(x')|^2  \, dx'
	\leq \frac{1}{C} \int_E V(v_{n}(x')) \, dx'
	\leq \frac{C_1}{C} \frac{1}{\lambda_{n}}.
$$
This implies that $\{v_{n}\}$ is $2$-equi-integrable.
Apply Vitali's convergence theorem to deduce that $v_n \to v$ in $L^2(J)$.
\end{proof}


\subsection{Lower bound for $G_{\eps}$}\label{ssec:bdyG1D_LB}

In this section we prove the following theorem.

\begin{thm}[Lower bound for $G_{\eps}$]\label{thm:boundary1D}

Let $J \subset \R$ be an open and bounded interval and let $V:\R \to [0,\infty)$ satisfy $(H^V_1)-(H^V_3)$. Assume that $\eps \lambda_{\eps}^{\frac{2}{3}} \rightarrow L \in (0,\infty)$.
Let $v \in BV\bigl(J;\{\alpha,\beta\}\bigr)$ and let $\{v_{\eps}\} \subset H^{\frac{3}{2}}(J)$ be such that
\begin{equation}\label{eq:sup}
\sup_{\eps>0} G_{\eps}(v_{\eps};J) =: C < \infty
\end{equation}
and $v_{\eps} \to v$ in $L^2(J)$ as $\eps \to0^+$.
Then
$$
\liminf_{\eps \to 0^+} G_{\eps}(v_{\eps};J) 
   \geq \ul{c}L \Haus{0}(S(v)),
$$
where $\ul{c} \in (0,\infty)$ is the constant defined in \eqref{eq:def_under_c}.
\end{thm}

We begin with some preliminary results.

\begin{lem}\label{lem:estimate}

Let $V:\R \to [0,\infty)$ satisfy $(H^V_1)-(H^V_3)$, and let $v \in H^{\frac{3}{2}}(c,d)$ be such that $Tv(c)=w$ and $Tv'(c)=z$, for some $c,d,z,w \in \R$, with $c<d$ and $|z|+|w-\alpha| \leq 1$.
Let 
$$
f(x) := \begin{cases}
  v(x)		& \text{ if } c\leq x \leq d, \\
  p(x)		& \text{ if } c -1 \leq x \leq c,
\end{cases}
$$
where $p$ is the polynomial given by 
$$p(x) := \alpha  + (3w-3\alpha-z)(x-c+1)^2+(z+2\alpha-2w)(x-c+1)^3.$$

Then $f \in H^{\frac{3}{2}}(c-1,d)$,
\begin{multline}\label{eq:nonlocal_alpha}
\int_{c-1}^d \int_{c-1}^d \frac{|f'(x)-f'(y)|^2}{|x-y|^2} \, dx \, dy
	- \int_c^d \int_c^d \frac{|v'(x)-v'(y)|^2}{|x-y|^2} \, dx \, dy \\
	\begin{aligned}
	& = \int_{c-1}^c \int_{c-1}^c \frac{|p'(x)-p'(y)|^2}{|x-y|^2} \, dx \, dy 
		+ 2 \int_{c-1}^c \int_c^d \frac{|v'(x)-p'(y)|^2}{|x-y|^2} \, dx \, dy\\
	& \leq C \bigl(|z|+|\alpha-w|\bigr)^2 + 2Sv(c), 
	\end{aligned}
\end{multline}
and
\begin{equation}\label{eq:V_alpha}
\int_{c-1}^d V(f(x)) \, dx -\int_c^d V(v(x)) \, dx
		= \int_{c-1}^c V(p(x)) \, dx
	\leq C \bigl(|z|+|\alpha-w|\bigr)^2,
\end{equation}
for some constant $C=C(V,\alpha)>0$, and where 
\begin{equation}\label{eq:Sv_alpha}
Sv(c) := \int_c^d \frac{|v'(x)-v'(c)|^2}{|x-c|^2} \, dx.
\end{equation}

Moreover,
\begin{align}
& \int_{c-1}^{c} \frac{|p'(x)|^2}{|x-c+1|^2} \, dx \leq C \bigl(|z|+|\alpha-w|\bigr)^2, \label{eq:secondpart1}\\
& \int_{c-1}^{c} \frac{|p'(x)|^2}{|x-d-1|^2} \, dx  \leq C \bigl(|z|+|\alpha-w|\bigr)^2. \label{eq:secondpart2}
\end{align}

\end{lem}

\begin{proof}

Since $V \in C^2(\R)$, and $V(\alpha) = V'(\alpha)=0$, by Taylor's formula, for any $t \in \R$, there exists $t_0$ between $\alpha$ and $t$ such that $V(t) = \frac{V''(t_0)}{2} (t-\alpha)^2$.

On the other hand, we have that
$$
|p(x)-\alpha| 
	\leq |3w-3\alpha-z| (x-c+1)^2+|z+2\alpha-2w| |x-c+1|^3
	\leq 5\bigl(|z| + |\alpha-w|\bigr)
$$
for all $x \in [c-1,c]$, and so
$$
\int_{c-1}^c V(p(x)) \, dx
	 \leq \frac{1}{2} \left(\max_{\xi \in [\alpha-5,\alpha+5]} |V''(\xi)| \right) \int_{c-1}^c (p(x)-\alpha)^2 \, dx 
	 \leq  C \bigl(|z|+|\alpha-w|\bigr)^2.
$$

To estimate the first integral in \eqref{eq:nonlocal_alpha}, write $p'(x)$ in the following form
$$
p'(x) = z + 2 (2z+3\alpha-3w)(x-c) + 3(z+2\alpha-2w)(x-c)^2,
$$
for all $x \in [c-1,c]$.
Then, for $x \in [c-1,c]$,
\begin{equation}\label{eq:p'_alpha}
|p'(x)-z| 	\leq 12 \bigl(|z| + |\alpha-w|\bigr) |x-c|,
\end{equation}
while for $x,y \in [c-1,c]$,
\begin{equation}\label{eq:p'2_alpha}
|p'(x)-p'(y)| \leq 18 \bigl(|z| + |\alpha-w|\bigr) |x-y|,
\end{equation}
and so
$$
\int_{c-1}^c \int_{c-1}^c \frac{|p'(x)-p'(y)|^2}{|x-y|^2} \, dx \, dy
	\leq C \bigl( |z|+|\alpha-w|\bigr)^2.
$$

To estimate the second integral in \eqref{eq:nonlocal_alpha}, we have that
$$
\int_{c-1}^c \left[ \int_{c}^d \frac{|v'(x)-p'(y)|^2}{|x-y|^2} \, dx \right] \, dy
	 \leq 2 \int_{c-1}^c \left[ \int_{c}^d \frac{|v'(x)-v'(c)|^2}{|x-y|^2} \, dx \right] \, dy
	 	+ 2\int_{c-1}^c \left[ \int_{c}^d \frac{|p'(y)-z|^2}{|x-y|^2} \, dx \right] \, dy,
$$
where we have used the fact that $v'(c)=z$.

By Fubini's theorem,
$$
\int_{c-1}^c \left[ \int_{c}^d \frac{|v'(x)-v'(c)|^2}{|x-y|^2} \, dx \right] \,dy
 	\leq  \int_{c}^d \frac{|v'(x)-v'(c)|^2}{|x-c|^2} \, dx
	=  Sv(c),
$$
while
$$
\int_{c-1}^c \left[ \int_{c}^d \frac{|p'(y)-z|^2}{|x-y|^2} \, dx \right] \, dy 
	 = (d-c) \int_{c-1}^c \frac{|p'(y)-z|^2}{(c-y)(d-y)}  \, dy
	 \leq \frac{C}{2} \bigl( |z|+|\alpha-w|\bigr)^2.
$$

This concludes the first part of the proof.
To estimate \eqref{eq:secondpart1}, we write $p'(x) = 2(3w-3\alpha-z)(x-c+1)+3 (z+2\alpha-2w)(x-c+1)^2$, so for $x \in (c-1,c)$ we have
$$
|p'(x)|^2 \leq C \bigl( |z|+|\alpha-w|\bigr)^2 (x-c+1)^2.
$$

Hence
$$
\int_{c-1}^{c} \frac{|p'(x)|^2}{|x-c+1|^2} \, dx
	\leq C \bigl(|z|+|\alpha-w|\bigr)^2,
$$
while
$$
\int_{c-1}^{c} \frac{|p'(x)|^2}{|x-d-1|^2} \, dx
	\leq C \bigl(|z|+|\alpha-w|\bigr)^2 \int_{c-1}^c \left| \frac{x-(c-1)}{x-(d+1)}\right|^2 \, dx
	\leq C \bigl(|z|+|\alpha-w|\bigr)^2.
$$
The estimate for \eqref{eq:secondpart2} is analogous.
This completes the proof.
\end{proof}

\begin{cor}\label{cor:estimate}

Let $V:\R \to [0,\infty)$ satisfy $(H^V_1)-(H^V_3)$ and let $v \in H^{\frac{3}{2}}(c,d)$ be such that $Tv(c)=w_1$ and $Tv'(c)=z_1$, $Tv(d)=w_2$ and $Tv'(d)=z_2$, for some $c,d,z_1,z_2,w_1,w_2 \in \R$, with $c<d$ and $|z_1|+|w_1-\alpha| \leq 1$ and $|z_2|+|w_2-\beta| \leq 1$.
Let 
\begin{equation}\label{eq:function_f}
f(x) := \begin{cases}
  p_2(x)		& \text{ if } d \leq x \leq d+1, \\
  v(x)		& \text{ if } c\leq x \leq d, \\
  p_1(x)		& \text{ if } c -1 \leq x \leq c,
\end{cases}
\end{equation}
where $p_1$ and $p_2$ are the polynomials given by
\begin{equation}\label{eq:polynomials}
\begin{aligned}
& p_1(x) := \alpha  + (3w_1-3\alpha-z_1)(x-c+1)^2+(z_1+2\alpha-2w_1)(x-c+1)^3, \\
& p_2(x) := \beta + (3w_2-3\beta+z_2)(d+1-x)^2+(2\alpha-2w_2-z_2)(d+1-x)^3.
\end{aligned}
\end{equation}

Then $f \in H^{\frac{3}{2}}(c-1,d+1)$,
\begin{multline}\label{eq:nonlocal}
\int_c^d \int_c^d \frac{|v'(x)-v'(y)|^2}{|x-y|^2} \, dx \, dy
	\geq \int_{c-1}^{d+1} \int_{c-1}^{d+1} \frac{|f'(x)-f'(y)|^2}{|x-y|^2} \, dx \, dy \\
		- C \bigl(|z_1|+|z_2|+|\alpha-w_1| + |\beta-w_2|\bigr)^2 - C Qv(c,d) - 2Sv(c)-2Sv(d),
\end{multline}
and
\begin{equation}\label{eq:V_ab}
\int_c^d V(v(x)) \, dx
	\geq \int_{c-1}^{d+1} V(f(x)) \, dx
		- C \bigl(|z_1|+|z_2|+|\alpha-w_1| + |\beta-w_2|\bigr)^2,
\end{equation}
where $C=C(V,\alpha,\beta)>0$,
\begin{equation}\label{eq:def_Q}
Qv(c,d) := \frac{|v'(c)-v'(d)|^2}{|c-d|^2},
\end{equation}
and $Sv(\cdot)$ is defined in \eqref{eq:Sv_alpha}.
\end{cor}

\begin{proof}

The estimate \eqref{eq:V_ab} follows by applying twice \eqref{eq:V_alpha} in Lemma \ref{lem:estimate}.

To obtain \eqref{eq:nonlocal}, by Figure \ref{fig:estimate}, it suffices to estimate the double integrals over the sets $S_1$, $S_2$, and $S_0$.

\begin{figure}[!htbp]
\begin{center}
\includegraphics*[width=125pt]{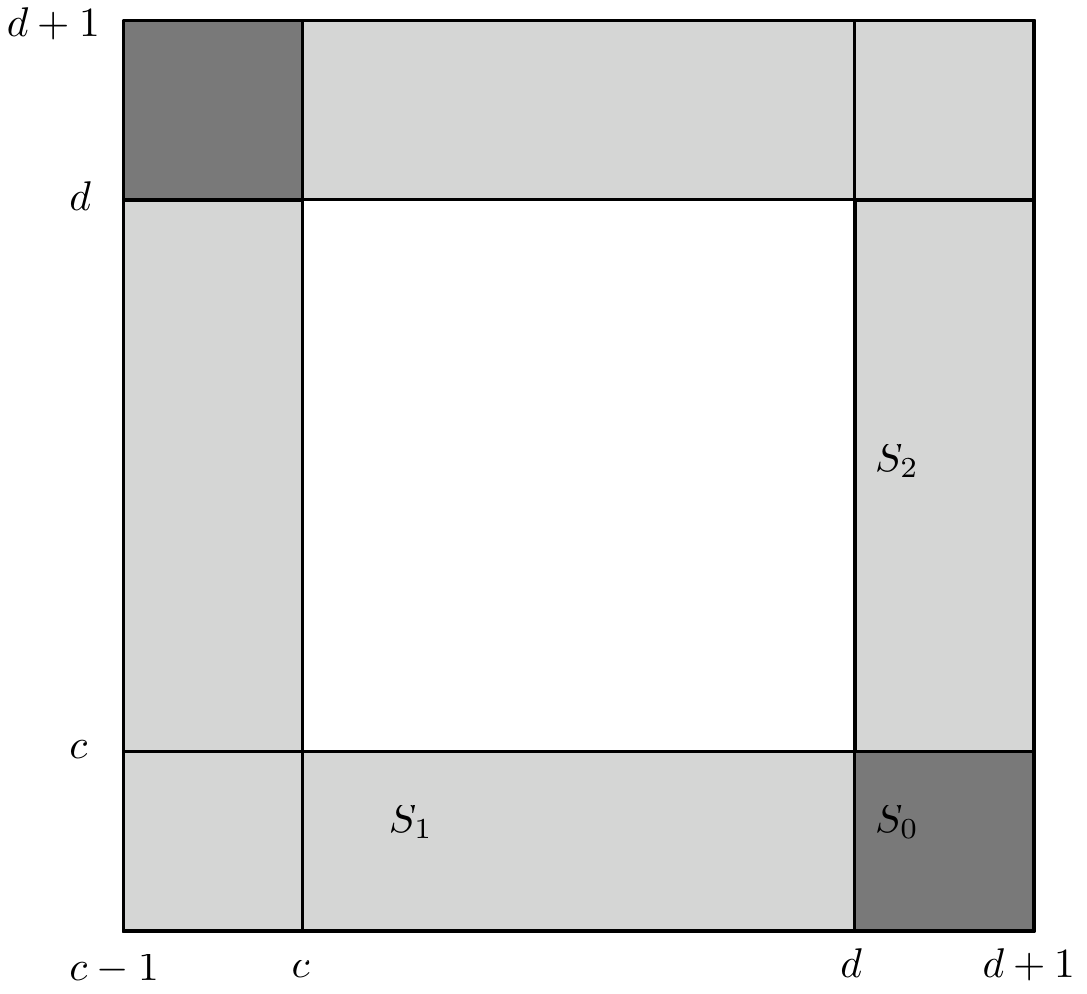}
\caption{Scheme for the estimates.}\label{fig:estimate}
\end{center}
\end{figure}

The estimates on $S_1$ and $S_2$ are a direct consequence of Lemma \ref{lem:estimate}.
To estimate the integral over $S_0$, we observe that
\begin{align*}
& p_1'(x) = z_1 + 2 (2z_1+3\alpha-3w_1)(x-c) + 3(z_1+2\alpha-2w_1)(x-c)^2, \\
& p_2'(x) = z_2 + 2 (-2z_2+3\beta-3w_2)(x-d) + 3(z_2-2\beta+2w_2)(x-d)^2,
\end{align*}
so for $x \in (c-1,c)$ and $y \in (d,d+1)$, we deduce that
$$
|p_1'(x)-p_2'(y)| 
	\leq |z_1-z_2| + C \bigl( |z_1|+|z_2|+|\alpha-w_1|+|\beta-w_2|\bigr) |x-y|.
$$
This implies that
\begin{align*}
\int_{d}^{d+1}\int_{c-1}^{c}  \frac{|p_1'(x)-p_2'(y)|^2}{|x-y|^2} \, dx \, dy
	& \leq C \frac{|z_1-z_2|^2}{|d-c|^2} + C\bigl( |z_1|+|z_2|+|\alpha-w_1|+|\beta-w_2|\bigr)^2.
\end{align*}
\end{proof}

\begin{cor}\label{cor:estimate_infinity}

Let $V:\R \to [0,\infty)$ satisfy $(H^V_1)-(H^V_3)$ and let $v \in H^{\frac{3}{2}}(c,d)$ be such that $Tv(c)=w_1$, d $Tv'(c)=z_1$, $Tv(d)=w_2$, and $Tv'(d)=z_2$, for some $c,d,z_1,z_2,w_1,w_2 \in \R$, with $c<d$, $|z_1|+|w_1-\alpha| \leq 1$, and $|z_2|+|w_2-\beta| \leq 1$.
Let 
\begin{equation}\label{eq:function_ff}
f(x) := \begin{cases}
  \beta	& \text{ if } x \geq d+1, \\
  p_2(x)		& \text{ if } d \leq x \leq d+1, \\
  v(x)		& \text{ if } c\leq x \leq d, \\
  p_1(x)		& \text{ if } c -1 \leq x \leq c, \\
  \alpha	& \text{ if } x \leq c-1,
\end{cases}
\end{equation}
where $p_1$ and $p_2$ are the polynomials defined in \eqref{eq:polynomials}.

Then $f \in H^{\frac{3}{2}}(c-1,d+1)$, $f' \in H^{\frac{1}{2}}(\R)$, and
\begin{multline}\label{eq:nonlocal_infinity}
\int_c^d \int_c^d \frac{|v'(x)-v'(y)|^2}{|x-y|^2} \, dx \, dy
	\geq \int_{-\infty}^{\infty} \int_{-\infty}^{\infty} \frac{|f'(x)-f'(y)|^2}{|x-y|^2} \, dx \, dy
		- C \bigl(|z_1|+|z_2|+|\alpha-w_1| + |\beta-w_2|\bigr)^2 \\
		- C Qv(c,d) 
	- 2( 1 + d-c) \bigl( Sv(c)- 2Sv(d)\bigr) 
	 - 2\log(1+d-c) \bigl( |z_1|^2 + |z_2|^2\bigr),
\end{multline}
where $C=C(V,\alpha,\beta)>0$, $Qv(\cdot,\cdot¥)$ is defined in \eqref{eq:def_Q}, and $Sv(\cdot)$ is defined in \eqref{eq:Sv_alpha}.
\end{cor}

\begin{proof}

By Corollary \ref{cor:estimate}, we know that 
\begin{multline*}
\int_{-c-1}^{c+1} \int_{c-1}^{d+1} \frac{|f'(x)-f'(y)|^2}{|x-y|^2} \, dx \, dy
	- \int_c^d \int_c^d \frac{|v'(x)-v'(y)|^2}{|x-y|^2} \, dx \, dy \\
	\leq C \bigl(|z_1|+|z_2|+|\alpha-w_1| + |\beta-w_2|\bigr)^2 + C Qv(c,d) + 2Sv(c)+2Sv(d),
\end{multline*}
so to prove estimate \eqref{eq:nonlocal_infinity}, it suffices to estimate
$$
\int_{-\infty}^{\infty} \int_{-\infty}^{\infty} \frac{|f'(x)-f'(y)|^2}{|x-y|^2} \, dx \, dy
	- \int_{-c-1}^{c+1} \int_{c-1}^{d+1} \frac{|f'(x)-f'(y)|^2}{|x-y|^2} \, dx \, dy 
	= 2 (I_4^{\pm} + I_5^{\pm} + I_6^{\pm}),
$$
where the $I_{i}$'s are defined by
\begin{align*}
& I_{4}^{-} := \int_{-\infty}^{c-1} \int_{c-1}^{c} 
	\frac{\left| p_1'(x)\right|^{2}}{|x-y|^{2}} \, dx \, dy,
& & I_{4}^{+} := \int_{d+1}^{\infty} \int_{d}^{d+1} 
	\frac{\left| p_2'(x)\right|^{2}}{|x-y|^{2}} \, dx \, dy, \\
& I_{5}^{-} := \int_{d+1}^{\infty} \int_{c-1}^{c} 
	\frac{\left| p_1'(x)\right|^{2}}{|x-y|^{2}} \, dx \, dy,
& & I_{5}^{+} := \int_{-\infty}^{c-1} \int_{d}^{d+1} 
	\frac{\left| p_2'(x)\right|^{2}}{|x-y|^{2}} \, dx \, dy, \\
& I_{6}^{-} := \int_{-\infty}^{c-1} \int_{c}^{d} 
	\frac{\left| v'(x)\right|^{2}}{|x-y|^{2}} \, dx \, dy,
& & I_{6}^{+} := \int_{d+1}^{\infty} \int_{c}^{d} 
	\frac{\left| v'(x)\right|^{2}}{|x-y|^{2}} \, dx \, dy.
\end{align*}

To estimate $I_4^-$, we compute
$$
I_{4}^{-} 
	= \int_{c-1}^{c} 
	\frac{\left| p_1'(x)\right|^{2}}{x-c+1} \, dx
	\leq C \bigl( |z_1| + |\alpha-w_1|\bigr)^2	
$$
by \eqref{eq:p'2_alpha} (with $p$ replaced by $p_1$), and analogously, 
$
I_4^+
	\leq C \bigl( |z_2| + |\beta-w_2|\bigr)^2.
$

For $I_5^{\pm}$, we have that 
$$
I_{5}^{-} 
	= \int_{c-1}^{c} \frac{\left| p_1'(x)\right|^{2}}{d+1-x} \, dx
	\leq \int_{c-1}^{c} \frac{\left| p_1'(x)\right|^{2}}{x-c+1} \, dx
	 = I_4^- \leq C \bigl( |z_1| + |\alpha-w_1|\bigr)^2,
$$
and analogously
$
I_5^+
	\leq C \bigl( |z_2| + |\beta-w_2|\bigr)^2.
$

To estimate $I_6^-$, we write
\begin{align}
I_6^-
	& = \int_{c}^{d} 	\frac{\left| v'(x) \pm v'(c)\right|^{2}}{x-c+1} \, dx
	\leq 2 \int_{c}^{d} \frac{\left| v'(x) - v'(c)\right|^{2}}{x-c+1} \, dx
	+ 2 \int_{c}^{d} \frac{\left| v'(c)\right|^{2}}{x-c+1} \, dx \label{eq:I6} \\
	& \leq 2\int_{c}^{d} \frac{\left| v'(x) - v'(c)\right|^{2}}{(x-c)^2} \frac{(x-c)^2}{x-c+1} \, dx + 2 |z_1|^2 \log(1+d-c) \nonumber\\
	& \leq 2(d-c)Sv(c)+ 2 |z_1|^2 \log(1+d-c). \nonumber
\end{align}

Analogously,
$
I_6^+
	\leq 2(d-c) Sv(d) + 2 |z_2|^2 \log(1+d-c).
$
This completes the proof.
\end{proof}

\begin{prop}\label{prop:almostLB}

Let $J \subset \R$ be an open and bounded interval and let $V:\R \to [0,\infty)$ satisfy $(H^V_1)-(H^V_3)$. Assume that $\eps \lambda_{\eps}^{\frac{2}{3}} \rightarrow L \in (0,\infty)$, and consider a sequence $\{v_{\eps}\} \subset H^{\frac{3}{2}}(J)$ such that $\ds \sup_{\eps>0} G_{\eps}\bigl(v_{\eps};(x^-,x^+)\bigr) < \infty$, for some $x^{\pm} \in J$, with
\begin{equation}\label{eq:points}
\begin{array}{lcl}
\begin{aligned}
& |v_{\eps}(x^{-})-\alpha| \leq \eta,\\
& |v_{\eps}(x^{+})-\beta| \leq \eta,
\end{aligned}
& \qquad &
\begin{aligned}
& \left|\eps v_{\eps}'(x^{\pm})\right| \leq C, \\
&\left|\eps^{3} Sv_{\eps}(x^{\pm})\right| \leq C, \\
&\left|\eps^{3} Qv_{\eps}(x^{-},x^{+}) \right| \leq C,
\end{aligned}
\end{array}
\end{equation}
where $C>0$, $\eta>0$, and $Sv(\cdot)$ and $Qv(\cdot,\cdot)$ are defined in \eqref{eq:Sv_alpha} and \eqref{eq:def_Q}, respectively.

Then
\begin{equation}\label{eq:almostLB}
\liminf_{\eps\to0^+} G_{\eps}\bigl(v_{\eps}; (x^-,x^+)\bigr) 
   \geq \ul{c}L,
\end{equation}
where $\ul{c} \in (0,\infty)$ is the constant defined in \eqref{eq:def_under_c}.

\end{prop}

\begin{proof}

Define $w_{\eps}(t) := v_{\eps}\bigl( \eps \lambda_{\eps}^{-\frac{1}{3}} t\bigr)$ for $x \in J$.
By the change of variables $x = \eps \lambda_{\eps}^{-\frac{1}{3}} t$, $y = \eps \lambda_{\eps}^{-\frac{1}{3}} s$, we have
\begin{align}
G_{\eps}\bigl(v_{\eps}; (x^-,x^+)\bigr) 
	& = \frac{\eps}{8} \int_{x^-}^{x^+} \int_{x^-}^{x^+} \frac{\bigl| v_{\eps}'(x)-v_{\eps}'(y)\bigr|^2}{|x-y|^2} \, dx\,dy + \lambda_{\eps} \int_{x^-}^{x^+} V\bigl(v_{\eps}(x)\bigr) \, dx \nonumber\\
	& = \eps \lambda_{\eps}^{\frac{2}{3}} \left[ \frac{1}{8} \int_{\frac{x^-}{\eps\lambda_{\eps}^{-
	\frac{1}{3}}}}^{\frac{x^+}{\eps\lambda_{\eps}^{-
	\frac{1}{3}}}} \int_{\frac{x^-}{\eps\lambda_{\eps}^{-
	\frac{1}{3}}}}^{\frac{x^+}{\eps\lambda_{\eps}^{-
	\frac{1}{3}}}} \frac{\bigl| w_{\eps}'(t)-w_{\eps}'(s)\bigr|^2}{|t-s|^2} \, dt\,ds + \int_{\frac{x^-}{\eps\lambda_{\eps}^{-
	\frac{1}{3}}}}^{\frac{x^+}{\eps\lambda_{\eps}^{-\frac{1}{3}}}} V\bigl(w_{\eps}(t)\bigr) \, dt \right] 
	\label{eq:Gnew}
\end{align}

Let $f_{\eps}$ be the function given in \eqref{eq:function_ff} with the choice of parameters
\begin{align*}
& v := w_{\eps}, \\
& c := \frac{x^-}{\eps \lambda_{\eps}^{-\frac{1}{3}}},
& & d := \frac{x^+}{\eps \lambda_{\eps}^{-\frac{1}{3}}}, \\
& w_1:= w_{\eps} \left( \frac{x^-}{\eps\lambda_{\eps}^{-\frac{1}{3}}} \right)= v_{\eps}(x^-),
& & w_2:= w_{\eps} \left( \frac{x^+}{\eps\lambda_{\eps}^{-\frac{1}{3}}} \right)= v_{\eps}(x^+), \\
& z_1 := w_{\eps}' \left( \frac{x^-}{\eps\lambda_{\eps}^{-\frac{1}{3}}} \right)= \eps \lambda_{\eps}^{-\frac{1}{3}} v_{\eps}'(x^-),
& & z_2 := w_{\eps}' \left( \frac{x^+}{\eps\lambda_{\eps}^{-\frac{1}{3}}} \right)= \eps \lambda_{\eps}^{-\frac{1}{3}} v_{\eps}'(x^+).
\end{align*}

By Corollary \ref{cor:estimate}, \eqref{eq:Gnew}, and the fact that $\eps \lambda_{\eps}^{\frac{2}{3}} \to L$, we have that
\begin{multline*}
G_{\eps}\bigl(v_{\eps};(x^-,x^+)\bigr) 
	\geq (L + o(1))
		\left[ \frac{1}{8} \int_{\frac{x^-}{\eps\lambda_{\eps}^{-	\frac{1}{3}}}-1}^{\frac{x^+}{\eps\lambda_{\eps}^{-	\frac{1}{3}}}+1}	 \int_{\frac{x^-}{\eps\lambda_{\eps}^{-\frac{1}{3}}}-1}^{\frac{x^+}{\eps\lambda_{\eps}^{-	\frac{1}{3}}}+1} \frac{\left| f_{\eps}'(t) - f_{\eps}'(s)  \right|^{2}}{|t-s|^{2}} \, dt \, ds 
		+ \int_{\frac{x^-}{\eps\lambda_{\eps}^{-	\frac{1}{3}}}-1}^{\frac{x^+}{\eps\lambda_{\eps}^{-	\frac{1}{3}}}+1} V(f_{\eps}(t)) \, dt \right] \\
	\quad - C \left[ \eps \lambda_{\eps}^{-\frac{1}{3}} \left( |v_{\eps}'(x^-)| + |v_{\eps}'(x^+)|\right) + |\alpha-v_{\eps}(x^-| + |\beta - v_{\eps}(x^+)| \right. \\
	\quad \left. + Qw_{\eps}\left( \frac{x^-}{\eps \lambda_{\eps}^{-\frac{1}{3}}},  \frac{x^+}{\eps \lambda_{\eps}^{-\frac{1}{3}}}\right)
	+ 2Sw_{\eps}\left(  \frac{x^-}{\eps \lambda_{\eps}^{-\frac{1}{3}}} \right) + 2Sw_{\eps}\left( \frac{x^+}{\eps \lambda_{\eps}^{-\frac{1}{3}}}\right) \right].
\end{multline*}

We claim that $f_{\eps}' \in H^{\frac{1}{2}}(\R)$.
If the claim holds, since $f_{\eps}$ is admissible for the constant $\ul{c}$ defined in \eqref{eq:def_under_c}, and by \eqref{eq:points}, we have that
\begin{equation}\label{eq:A1}
G_{\eps}\bigl(v_{\eps};(x^-,x^+)\bigr) 
	\geq (L + o(1)) \ol{c} - C (2 \lambda_{\eps}^{-\frac{1}{3}} + 2\eta )^2 \\
	- C Qw_{\eps}\left( \frac{x^-}{\eps \lambda_{\eps}^{-\frac{1}{3}}},  \frac{x^+}{\eps \lambda_{\eps}^{-\frac{1}{3}}}\right)
	- C Sw_{\eps}\left(  \frac{x^-}{\eps \lambda_{\eps}^{-\frac{1}{3}}} \right) - CSw_{\eps}\left( \frac{x^+}{\eps \lambda_{\eps}^{-\frac{1}{3}}}\right).
\end{equation}

Since $\lambda_{\eps} \to \infty$, to conclude that the first part of the proof, it remains to estimate the last three terms on the right-hand side of \eqref{eq:A1}.
By \eqref{eq:points},
\begin{equation}\label{eq:A2}
Qw_{\eps}\left( \frac{x^-}{\eps \lambda_{\eps}^{-\frac{1}{3}}},  \frac{x^+}{\eps \lambda_{\eps}^{-\frac{1}{3}}}\right)
	 = \left| \frac{w_{\eps}'\left( \frac{x^+}{\eps \lambda_{\eps}^{-\frac{1}{3}}} \right) - w_{\eps}'\left(\frac{x^-}{\eps \lambda_{\eps}^{-\frac{1}{3}}}\right)}{\frac{x^+}{\eps \lambda_{\eps}^{-\frac{1}{3}}}-\frac{x^-}{\eps \lambda_{\eps}^{-\frac{1}{3}}}} \right|^2
	\leq C \eps \lambda_{\eps}^{-\frac{4}{3}}, 
\end{equation}
while
\begin{equation} \label{eq:A3}
Sw_{\eps}\left(  \frac{x^-}{\eps \lambda_{\eps}^{-\frac{1}{3}}} \right)
	 = \int_{\frac{x^-}{\eps \lambda_{\eps}^{-\frac{1}{3}}}}^{\frac{x^+}{\eps \lambda_{\eps}^{-\frac{1}{3}}}} \frac{\left| w_{\eps}'(t) - w_{\eps}'\left(\frac{x^-}{\eps \lambda_{\eps}^{-\frac{1}{3}}}\right)\right|^2}{\left|t-\frac{x^-}{\eps \lambda_{\eps}^{-\frac{1}{3}}},\right|^2} \, dt 
	\leq C \lambda_{\eps}^{-1},
\end{equation}
and similarly 
\begin{equation}\label{eq:A4}
Sw_{\eps}\left(  \frac{x^+}{\eps \lambda_{\eps}^{-\frac{1}{3}}} \right) \leq C \lambda_{\eps}^{-1}.
\end{equation}
Thus, by \eqref{eq:A1}--\eqref{eq:A4},
$$
G_{\eps}\bigl(v_{\eps};(x^-,x^+)\bigr)
	\geq (L-o(1)) \ul{c} - C (2 \lambda_{\eps}^{-\frac{1}{3}} + 2\eta)^2 - C \eps \lambda_{\eps}^{-\frac{4}{3}} - C\lambda_{\eps}^{-1}.
$$
Letting first $\eps \to 0^+$ and then $\eta \to 0^+$ we obtain \eqref{eq:almostLB}.
To complete the proof, we show that 
$$
\sup_{\eps} |f_{\eps}'|_{H^{\frac{1}{2}}(\R)} \leq C.
$$

Starting again from \eqref{eq:Gnew}, but using Corollary \ref{cor:estimate_infinity} in place of Corollary \ref{cor:estimate}, we obtain
\begin{equation}\label{eq:Gnew_infinity}
\begin{aligned}
C & \geq G_{\eps}\bigl(v_{\eps};(x^-,x^+)\bigr) 
	 \geq (L + o(1))
		\frac{1}{8} \int_{-\infty}^{\infty}\int_{-\infty}^{\infty} \frac{\left| f_{\eps}'(t) - f_{\eps}'(s)  \right|^{2}}{|t-s|^{2}} \, dt \, ds \\
	& \quad - C \Biggl[ \eps \lambda_{\eps}^{-\frac{1}{3}} \left( |v_{\eps}'(x^-)| + |v_{\eps}'(x^+)|\right) + |\alpha-v_{\eps}(x^-| + |\beta - v_{\eps}(x^+)| \\
	& \quad + Qw_{\eps}\left( \frac{x^-}{\eps \lambda_{\eps}^{-\frac{1}{3}}},  \frac{x^+}{\eps \lambda_{\eps}^{-\frac{1}{3}}}\right)
	+ \left( 1 + \frac{x^+-x^-}{\eps \lambda_{\eps}^{-\frac{1}{3}}} \right)
	 \left( Sw_{\eps}\left(  \frac{x^-}{\eps \lambda_{\eps}^{-\frac{1}{3}}} \right) + Sw_{\eps}\left( \frac{x^+}{\eps \lambda_{\eps}^{-\frac{1}{3}}}\right)\right) \\
	& \quad  + \eps^2 \lambda_{\eps}^{-\frac{2}{3}} \log \left( 1 + \frac{x^+-x^-}{\eps \lambda_{\eps}^{-\frac{1}{3}}} \right) \left( |v_{\eps}'(x^-)|^2 + |v_{\eps}'(x^+)|^2\right) \Biggr].
\end{aligned}
\end{equation}

By \eqref{eq:points}, and \eqref{eq:A1}--\eqref{eq:A4}, we have
\begin{multline}\label{eq:Gnew_infinity2}
C 
	\geq G_{\eps}\bigl(v_{\eps};(x^-,x^+)\bigr) 
	 \geq (L + o(1)) |f_{\eps}'|_{H^{\frac{1}{2}}(\R)}^2
	- C (2 \lambda_{\eps}^{-\frac{1}{3}} + 2\eta)^2 - C \eps \lambda_{\eps}^{-\frac{4}{3}} \\ 
	- \frac{C}{\lambda_{\eps}}\left( 1 + \frac{x^+-x^-}{\eps\lambda_{\eps}^{-\frac{1}{3}}} \right) - C \lambda_{\eps}^{-\frac{2}{3}} \log \left( 1 + \frac{x^+-x^-}{\eps \lambda_{\eps}^{-\frac{1}{3}}} \right).
\end{multline}

Since $\eps \lambda_{\eps}^{\frac{2}{3}} \to L$, it follows that for all $\eps>0$ sufficiently small, we have
$$
(L-o(1)) |f_{\eps}'|_{H^{\frac{1}{2}}(\R)}^2
	\leq C(1+L) + C\eta,
$$
where $C$ depends also on $x^+-x^-$.
This proves that $f_{\eps}' \in H^{\frac{1}{2}}(\R)$, which completes the proof.
\end{proof}

\begin{proof}[Proof of Theorem \ref{thm:boundary1D}]

Passing to a subsequence (not relabeled), we can assume that
$$
\liminf_{\eps \to 0^+} G_{\eps}(v_{\eps};J) 
	= \lim_{\eps\to 0^+} G_{\eps}(v_{\eps};J).
$$

This will allow us to take further subsequences (not relabeled).
By Proposition \ref{prop:gagliardo}, \eqref{eq:sup}, and the growth condition $(H^V_2)$, we know that 
$
\|v_{\eps}\|_{H^1(J)} 
	\leq C \eps^{-1}$.
	
Since $v \in BV\bigl(J;\{\alpha,\beta\}\bigr)$, its jump set $S(v)$ is finite, and we write
$$
S(v) = \{s_1,\ldots,s_{\ell}\},
$$
where $s_1 < \cdots < s_{\ell}$.
Let $0 < d < \frac{1}{2} \min \left\{ s_i - s_{i-1}: i=2,\ldots,\ell\right\}$, and assume that $v= \alpha$ in $(s_{2j},s_{2j+1})$ for $j=0,\ldots$, where $s_0, s_{\ell+1}$ are the endpoints of $J$.
Then
$$
\lim_{k \to \infty} \liminf_{\eps \to 0^+} \int_{s_1-d}^{s_1} \left[ k |v_{\eps}(x)-\alpha| + \frac{1}{k} \eps |v_{\eps}'(x)| + \frac{\eps^3}{k} \int_{J} \frac{\bigr|v_{\eps}'(x) - v_{\eps}'(y)\bigr|^2}{|x-y|^2} \, dy \right] \, dx = 0.
$$

Hence, we may find $k_0 \in \N$ such that for all $k \geq k_0$,
$$
\liminf_{\eps \to 0^+} \int_{s_1-d}^{s_1} \left[ k |v_{\eps}(x)-\alpha| + \frac{1}{k} \eps |v_{\eps}'(x)| + \frac{\eps^3}{k} \int_{J} \frac{\bigr|v_{\eps}'(x) - v_{\eps}'(y)\bigr|^2}{|x-y|^2} \, dy \right] \, dx \leq d.
$$

By Fatou's lemma, we have that for $k \geq k_0$,
$$
\frac{1}{d}\int_{s_1-d}^{s_1} \liminf_{\eps \to 0^+} \left[ k |v_{\eps}(x)-\alpha| + \frac{1}{k} \eps |v_{\eps}'(x)| + \frac{\eps^3}{k} \int_{J} \frac{\bigr|v_{\eps}'(x) - v_{\eps}'(y)\bigr|^2}{|x-y|^2} \, dy \right] \, dx \leq 1
$$

Fix $k_1 > \max \bigl\{ k_0,\frac{1}{\eta} \bigr\}$.
By the mean value theorem, there exists $x_1^- \in (s_1-d,s_1)$ such that
$$
\liminf_{\eps \to 0^+} \left[ |v_{\eps}(x_1^-)-\alpha| + \frac{1}{k_1^2} \eps |v_{\eps}'(x_1^-)| + \frac{\eps^3}{k_1^2} \int_{J} \frac{\bigr|v_{\eps}'(x_1^-) - v_{\eps}'(y)\bigr|^2}{|x_1^--y|^2} \, dy \right]  < \eta.
$$

So, up to a subsequence (not relabeled),
\begin{equation}\label{eq:xis-}
|v_{\eps}(x_1^-)-\alpha| < \eta,
\qquad 
\eps |v_{\eps}'(x_1^-)| < \eta k_1^2,
\quad  \text{ and }\quad
\eps^3\int_{J} \frac{\bigr|v_{\eps}'(x_1^-) - v_{\eps}'(y)\bigr|^2}{|x_1^--y|^2} \, dy  < \eta k_1^2.
\end{equation}

Analogously, considering
$$
\lim_{k \to \infty} \liminf_{\eps \to 0^+} \int_{s_1}^{s_1+d} \left[ k |v_{\eps}(x)-\alpha| + \frac{1}{k} \eps |v_{\eps}'(x)| + \frac{\eps^3}{k} \int_{J} \frac{\bigr|v_{\eps}'(x) - v_{\eps}'(y)\bigr|^2}{|x-y|^2} \, dy  + \frac{\eps^3}{k} \frac{\bigr|v_{\eps}'(x) - v_{\eps}'(x_1^-)\bigr|^2}{|x-x_1^-|^2}
\right] \, dx = 0,
$$
we may find $x_1^+ \in (s_1,s_1+d)$ such that (up to a further subsequence)
\begin{equation}\label{eq:xis+}
|v_{\eps}(x_1^+)-\beta| < \eta,
\qquad 
\eps |v_{\eps}'(x_1^+)| < \eta k_2^2,
\quad  \text{ and }\quad
\eps^3\int_{J} \frac{\bigr|v_{\eps}'(x_1^+) - v_{\eps}'(y)\bigr|^2}{|x_1^+-y|^2} \, dy  < \eta k_2^2,
\end{equation}
and
\begin{equation}\label{eq:xis+-}
\eps^3 \frac{\bigr|v_{\eps}'(x_1^+) - v_{\eps}'(x_1^-)\bigr|^2}{|x_1^+-x_1^-|^2} < \eta k_2^2.
\end{equation}

We now repeat the process to find points $x_i^{\pm}$ in $(s_i-d,s_i+d)$ with the properties \eqref{eq:xis-}--\eqref{eq:xis+-}.

By Proposition \ref{prop:almostLB}, we deduce that
$$
\liminf_{\eps \to 0^+} G_{\eps}(v_{\eps};J)  
	\geq \sum_{i=1}^{\ell} \liminf_{\eps \to 0^+} G_{\eps}\bigl(v_{\eps};(x_i^-,x_i^+)\bigr)
	 \geq \ell L \ul{c} = \ul{c} L \Haus{0}(S(v)).
$$
\end{proof}

\section{The $N$-dimensional case}\label{sec:critical}

In this section we prove Theorems \ref{thm:compactness} and \ref{thm:critical}.

%
%

\subsection{Compactness}\label{ssec:Compactness}

In this subsection we prove Theorem \ref{thm:compactness}.
We follow the argument of \cite{FM}, which we reproduce for the convenience of the reader.

\begin{thm}[Compactness in the interior]
\label{thm:compactness_interior}

Let $\Omega$, $W$, and $V$ satisfy the hypotheses of Theorem \ref{thm:compactness}, and let $\eps \lambda_{\eps}^{\frac{2}{3}} \to L \in (0,\infty)$.
Consider a sequence $\{u_{\eps}\} \subset H^2(\Omega)$ such that
$$
C_1 := \sup_{\eps} \mathcal{F}_{\eps}(u_{\eps}) < \infty,
$$
where $\mathcal{F}_{\eps}$ is the functional defined in \eqref{eq:F}.
Then there exist a subsequence of $\{u_{\eps}\}$ (not relabeled) and a function $u \in BV(\Omega;\{a,b\})$ such that $u_{\eps} \to u$ in $L^2(\Omega)$.
\end{thm}

\begin{proof}

For simplicity of notation, we suppose $N=2$. The higher dimensional case is treated analogously.

\paragraph{Step 1.} Assume that $\Omega = I \times J$, where $I,J \subset \R$ are open bounded intervals.

For $x \in \Omega$, we write $x=(y,z)$, with $y\in I, z\in J$. For every function $u$ defined on $\Omega$ and every $y \in I$ we denote by $u^y$ the function on $J$ defined by $u^y(z) := u(y,z)$, and for every $z \in J$ we denote by $u^z$ the function on $I$ defined by $u^z(y) := u(y,z)$. The functions $u^y$ and $u^z$ are called one-dimensional slices of $u$.

We recall that by slicing, if $u\in H^2(\Omega)$, then $u^y \in H^2(J)$ for $\Leb{1}$-a.e. $y \in I$, $u^z \in H^2(I)$ for $\Leb{1}$-a.e. $z \in J$, and 
$$
\fpar{^2 u}{z^2}(y,z) = \frac{d^2u^y}{dz^2}(z),
\qquad
\fpar{^2 u}{y^2}(y,z) = \frac{d^2u^z}{dy^2}(y),
\quad
\text{ for $\Leb{1}$-a.e. } y \in I \text{ and for $\Leb{1}$-a.e. } z \in J.
$$

Since $|\grad^2 u|^2 \geq \max \left\{ \left| \fpar{^2 u}{z^2}\right|, \left|\fpar{^2 u}{y^2}\right| \right\}$, we immediately obtain that
\begin{equation}\label{eq:fm_slicing}
C_1 \geq \mathcal{F}_{\eps}(u) \geq \int_I F_{\eps}(u^y;J) \, dy,
\qquad 
C_1 \geq \mathcal{F}_{\eps}(u) \geq \int_J F_{\eps}(u^z;I) \, dz,
\end{equation}
where $F_{\eps}$ is the functional defined in \eqref{eq:bulkF}.

Consider a family $\{u_{\eps}\} \subset H^2(\Omega)$ such that $\mathcal{F}_{\eps}(u_{\eps}) \leq C_1 < \infty$.
Then we have that $W(u_{\eps}) \to 0$ in $L^1(\Omega)$. 
From condition $(H^W_2)$, we have the existence of $C,T>0$ such that for all $|z|\geq T$, $W(z) \geq C |z|^2$.
This implies that $\{u_{\eps}\}$ is  2-equi-integrable and, in particular, it is equi-integrable.
Therefore, fix $\delta>0$ and let $\eta>0$ be such that for any measurable set $E \subset \R$, with $\Leb{2}(E) \leq \eta$,
\begin{equation}\label{eq:equiintegr}
\sup_{\eps>0} \int_E \bigl( |u_{\eps}(x)| + |b| \bigr) \, dx \leq \delta.
\end{equation}
For $\eps>0$ we define $v_{\eps}:\Omega \to \R$ by
$$
v{\eps}(y,z) := \begin{cases}
u^y_{\eps}(z)
	& \text{ if } y \in I, z \in J, \text{ and }F_{\eps}(u^y_{\eps};J) \leq \frac{C\Leb{1}(J)}{\eta}, \\
b	& \text{ otherwise.}
\end{cases}
$$

We claim that $\{v_{\eps}\}$ and $\{u_{\eps}\}$ are $\delta$-close, i.e., $\|u_{\eps}-v_{\eps}\|_{L^1(\Omega)} < \delta$.

Indeed, let $Z_{\eps} := \{ y \in I: u^y_{\eps} \neq v^y_{\eps} \}$.
By \eqref{eq:fm_slicing}, we have
$$
C_1 \geq \int_I F_{\eps}(u^y;J) \, dy,
$$
and so
$$
\Leb{1}(Z_{\eps}) 
	\leq \Leb{1}\left( \bigl\{ y \in I: F_{\eps}(u^y_{\eps};J) > {\ts \frac{C_1\Leb{1}(J)}{\eta}}\bigr\}\right)
	\leq \frac{\eta}{C_1\Leb{1}(J)} \int_I F_{\eps}(u^y;J) \, dy \leq \frac{\eta}{\Leb{1}(J)}.
$$
It follows that $\Leb{2}(Z_{\eps} \times J) \leq \eta$.
Thus, by \eqref{eq:equiintegr},
$$
\| u_{\eps} - v_{\eps}\|_{L^1(\Omega)} 
	\leq \int_{Z_{\eps} \times J} |u_{\eps}(x) - b| \, dx
	\leq \int_{Z_{\eps} \times J} \bigl(|u_{\eps}(x)| + |b| \bigr) \, dx
	\leq \delta.
$$

Moreover, for every $y\in I$ we have $F_{\eps}(v^y_{\eps};J) \leq \frac{C_1\Leb{1}(J)}{\eta}$, where we have used the face that $F_{\eps}(b;J)=0$, and therefore Theorem \ref{thm:compactness_F1D}, yields $L^2(J)$ precompactness of $\{v^y_{\eps}\}$.
Similarly, we can construct a  sequence $\{w_{\eps}\}$ $\delta$-close to $\{u_{\eps}\}$ so that $\{w^z_{\eps}\}$ is precompact in $L^2(I)$ for every $z \in J$.

Using Proposition \ref{prop:slicing} we conclude that the sequence $\{u_{\eps}\}$ is precompact in $L^2(\Omega)$.

\paragraph{Step 2.} General case.

This case can be proved by decomposing $\Omega$ into a countable union of closed rectangles with disjoint interiors.
The fact that the limit $u$ belongs to $BV(\Omega;\{a,b\}))$ is a direct consequence of Theorem \ref{thm:slicing_BV}.
\end{proof}

\begin{thm}[compactness at the boundary]
\label{thm:compactness_boundary}

Let $\Omega$, $W$, and $V$ satisfy the hypotheses of Theorem \ref{thm:compactness}, and let $\eps \lambda_{\eps}^{\frac{2}{3}} \to L \in (0,\infty)$.
Consider a sequence $\{u_{\eps}\} \subset H^2(\Omega)$ such that
$$
C := \sup_{\eps} \mathcal{F}_{\eps}(u_{\eps}) < \infty,
$$
where $\mathcal{F}_{\eps}$ is the functional defined in \eqref{eq:F}.
Then there exist a subsequence of $\{u_{\eps}\}$ (not relabeled) and a function $v \in BV(\partial\Omega;\{\alpha,\beta\})$ such that $Tu_{\eps} \to v$ in $L^2(\partial\Omega)$.
\end{thm}

To prove this theorem we introduce the localization of the functionals $\mathcal{F}_{\eps}$: for every open set $A \subset \Omega$ with boundary of class $C^2$, for every Borel set $E \subset \partial A$, and for every $u \in H^2(A)$, we set
$$
\mathcal{F}_{\eps}(u;A,E)
	:= \int_A \left( \eps^2 \left| \grad^2 u\right|^2 + \frac{1}{\eps} W(u) \right) \, dx + \lambda_{\eps} \int_E V(Tu) \, d\Haus{N-1}.
$$

Note that for $u \in H^2(\Omega)$,
$
\mathcal{F}_{\eps}(u) = \mathcal{F}_{\eps}(u;\Omega,\partial\Omega).
$

We begin by proving compactness on the boundary in the special case in which $A = \Omega \cap B$, where $B$ is a ball centered on $\partial \Omega$ and $E = B \cap \partial \Omega$ is a flat disk.
Later on we will show that this flatness assumption can be dropped when $B$ is sufficiently small.

\begin{prop}\label{prop:compactness_bdy_flat}

For every $r>0$, let $D_r$ be the open half-ball
$$
D_r := \{ x = (x',x_N) \in \R^N: |x| < r, x_N>0 \}
$$
and let 
$$
E_r := \{ x = (x', 0) \in \R^N: |x| < r\}.
$$
Let $W$ and $V$ satisfy the hypotheses of Theorem \eqref{thm:compactness}, and let $\eps\lambda_{\eps}^{\frac{2}{3}} \to L \in (0,\infty)$.
Consider a sequence $\{u_{\eps}\} \subset H^2(D_r)$ such that
$$
C_1 := \sup_{\eps>0} \mathcal{F}_{\eps}(u_{\eps};D_r,E_r) < \infty.
$$
Then there exist a subsequence of $\{u_{\eps}\}$ (nor relabeled) and a function $v \in BV(E_r;\{\alpha,\beta\})$ such that $Tv_{\eps} \to v$ in $L^2(E_r)$.

\end{prop}

\begin{proof}

To simplify the notation, we write $D$ and $E$ in place of $D_r$ and $E_r$.

The idea of the proof is to reduce to the statement of Theorem \ref{thm:compactness_G1D} via a suitable slicing argument.

Fix $i=1,\ldots, N-1$ and let $E_{e_i} := \{ y \in \R^{N-2}: (y,x_i,0) \in E \text{ for some } x_i \in\R\}$.
For every $y \in E_{e_i}$, define the sets
\begin{align*}
& D^y := \{ (x_i,x_N) \in \R^2: (y,x_i,x_N) \in D\}, \\
& E^y := \{ x_i \in \R: (y,x_i,0) \in E\}.
\end{align*}

For every $ y \in E_{e_i}$ and every function $u:D\to \R$, let $u^y:D^y \to \R$ be the function defined by 
$$
u^y(x_i,x_N) := u(y,x_i,x_N),
\qquad 
(x_i,x_N) \in D^y,
$$
and for every function $v:E\to \R$, let $v^y:E^y \to \R$ be defined by
$$
v^y(x_i) := v(y,x_i),
\qquad 
x_i \in E^y.
$$

If $u \in H^2(D)$, then by the slicing theorem in \cite{Ziemer} for $\Leb{N-2}$-a.e. $y \in E_{e_i}$, the function $u^y$ belongs to $H^2(D^y)$, for $\Leb{2}$-a.e. $(x_i,x_N) \in D$,
$$
\fpar{u}{x_k}(y,x_i,x_N) = \fpar{u^y}{x_k}(x_i,x_N),
\qquad \text{ for } k=i,N,
$$
and
$$
\fpar{^2u}{x_k\partial x_j}(y,x_i,x_N) = \fpar{^2u^y}{x_k\partial x_j}(x_i,x_N),
\qquad \text{ for } k,j=i,N,
$$
and the trace of $u^y$ on $E^y$ agrees $\Leb{1}$-a.e. in $E^y$ with $(Tu)^y$.
Taking into account these facts and Fubini's theorem, for every $\eps>0$ we get
\begin{multline*}
\mathcal{F}_{\eps}(u;D,E)
	 \geq \eps^3 \int_D |D^2u(x)|^2 \, dx + \lambda_{\eps} \int_E V(Tu(x',0)) \, dx' \\ 
	 \geq \int_{E_{e_i}} \left[ \eps^3 \int_{D^y} |D^2_{x_i,x_N} u^y(x_i,x_N)|^2 \, dx_i\,dx_N + \lambda_{\eps} \int_{E^y} V(Tu^y(x_i,0)) \, dx_i \right] \, dy.
\end{multline*}

We apply the trace inequality \eqref{eq:lifting} to each function $u^y$ to obtain
\begin{equation}\label{eq:Fcal_geq_G}
\mathcal{F}_{\eps}(u;D,E)
	\geq \int_{E_{e_i}} G_{\eps}(Tu^y;E^y) \, dy,
\end{equation}
where $G_{\eps}$ is the functional defined in \eqref{eq:bdyG}
To prove that the sequence $\{Tu_{\eps}\}$ is precompact in $L^2(E)$, it is enough to show that it satisfies the conditions of Proposition \ref{prop:slicing}.
Since
\begin{equation}\label{eq:C1}
C_1 = \sup_{\eps>0} \mathcal{F}_{\eps}(u_{\eps};D,E) < \infty,
\end{equation}
we have that
$$
\sup_{\eps>0} \lambda_{\eps} \int_E V\bigl(Tu_{\eps}(x',0)\bigr) \, dx' \leq C_1.
$$

From condition $(H^V_2)$, we may find $C,T>0$ such that for all $|z| \geq T$, $V(z)  \geq C |z|^2$,
and so
$$
\int_{E \cap \{|Tu_{\eps}|\geq t\}} \bigl|Tu_{\eps}(x',0)\bigr|^2 \, dx'
	\leq \frac{2}{C} \int_E V\bigl(Tu_{\eps}(x',0)\bigr) \, dx'
	\leq \frac{2C_1}{C} \frac{1}{\lambda_{\eps}}.
$$
This implies that $\{Tu_{\eps}\}$ is 2-equi-integrable. In particular, it is equi-integrable.
Thus to apply Proposition \ref{prop:slicing}, it remains to show that for every $\delta >0$ there is a sequence $\{v_{\eps}\} \subset L^1(E)$ that is $\delta$-close to $Tu_{\eps}$, in the sense of Definition \ref{def:deltaclose}, and such that $\{v^y_{\eps}\}$ is precompact in $L^1(E^y)$ for $\Leb{N-2}$-a.e. $y\in E_{e_i}$.

Fix $\delta>0$, let $\eta>0$ be a constant that will be fixed later, and let
\begin{equation}\label{eq:def_v_slicing}
v_{\eps}(y,x_i):=\begin{cases}
	Tu_{\eps}^y(x_i)
		& \text{ if } y \in E_{e_i}, x \in E^y, \text{ and } G_{\eps}(Tu^y_{\eps};E^y) \leq \frac{C_1}{\eta}, \\
	\alpha
		& \text{ otherwise.}
\end{cases}
\end{equation}

Note that although $v_{\eps}$ is no longer in $H^{\frac{3}{2}}(E)$, for every $y \in E_{e_i}$, either $v^y_{\eps} = Tu^y_{\eps} \in H^{\frac{3}{2}}(E^y)$, or $v^y_{\eps} \equiv \alpha$, and so $v^y_{\eps}$ always belongs to $H^{\frac{3}{2}}(E^y)$.
We claim that $\{v_{\eps}\}$ is $\delta$-close to $\{Tu_{\eps}\}$.
Indeed, by Fubini's theorem,
$$
\|Tu_{\eps} - v_{\eps} \|_{L^1(E)}
	\leq \int_{Z_{e_i}} \int_{E^y} |Tu^y_{\eps}(x_i)-\alpha| \, dx_i \, dy
	\leq \int_{Z_{e_i}} \int_{E^y} \left( |Tu^y_{\eps}(x_i)| + |\alpha| \right) \, dx_i \, dy,
$$
where $Z_{e_i} 
	:= \{y \in E_{e_i}: Tu^y_{\eps} \neq v^y_{\eps}\} 
	=  \bigl\{y \in E_{e_i}: G_{\eps}(Tu^y_{\eps};E^y) > {\ts \frac{C_1}{\eta}} \bigr\}$.
Since $\{Tu_{\eps}\}$ is equi-integrable, to prove that the right-hand side of the previous inequality is less than $\delta$, it suffices to show that the $\Leb{N-1}$ measure of the set $H:=\{(y,x_i): y \in Z_{e_i}, x_i \in E^y\}$ can be made arbitrarily small.
Again by Fubini's theorem and the definition of $Z_{e_i}$,
$$
\Leb{N-1}(H)
	= \int_{Z_{e_i}} \Leb{1}(E^y) \, dy
	\leq 2r \Leb{N-2}(Z_{e_i})
	\leq \frac{\eta}{C_1} \int_{Z_{e_i}} G_{\eps}(Tu^y_{\eps};E^y) \, dy
	\leq \eta,
$$
where we have used \eqref{eq:C1} and the fact that $\Leb{1}(E^y) \leq 2r \leq 1$ for $r \leq \frac{1}{2}$.
Thus if $\eta$ is chosen sufficiently small, we have that $\{v_{\eps}\}$ is $\delta$-close to $\{Tu_{\eps}\}$.

To prove that $\{v^y_{\eps}\}$ if precompact for $\Leb{N-2}$-a.e. $y \in E_{e_i}$, it suffices to consider only those $y \in E_{e_i}$ such that $G_{\eps}(Tu^y_{\eps};E^y) \leq \frac{C_1}{\eta}$ (since otherwise $v^y_{\eps}(x_i) \equiv \alpha$ and there is nothing to prove).
For these $y \in E_{e_i}$, the precompactness follows from Theorem \ref{thm:compactness_G1D}.

Hence we are in a position to apply Proposition \ref{prop:slicing} to conclude that $\{Tu_{\eps}\}$ is precompact in $L^1(E)$.
Thus, up to a subsequence (not relabeled), we may assume that there exists a function $v \in L^1(E)$ such that $Tu_{\eps} \to v$ in $L^1(E)$.
Note that since $\{Tu_{\eps}\}$ is 2-equi-integrable, it follows by Vitali's convergence theorem that $Tu_{\eps} \to v$ in $L^2(E)$. 

It remains to show that $v \in BV(E;\{\alpha,\beta\})$.
Indeed, replacing $u$ by $u_{\eps}$ in \eqref{eq:Fcal_geq_G}, and passing to the limit as $\eps \to 0^+$, by Fatou's lemma we deduce that
$$
\infty> \liminf_{\eps \to 0^+} \mathcal{F}_{\eps}(u_{\eps};D,E)
	\geq \int_{E_{e_i}} \liminf_{\eps \to 0^+} G_{\eps}(Tu^y_{\eps};E^y) \, dy,
$$
which implies that $\ds \liminf_{\eps \to 0^+} G_{\eps}(Tu^y_{\eps};E^y)$ is finite for $\Leb{N-2}$-a.e. $y \in E_{e_i}$.
Since $Tu_{\eps} \to v$ in $L^2(E)$, up to a subsequence (not relabeled), we have that $Tu^y_{\eps} \to v^y$ in $L^2(E)$ for $\Leb{N-2}$-a.e. $y \in E_{e_i}$.
Then Proposition \ref{prop:Bv_bdy} yields $v^y \in BV(E^y;\{\alpha,\beta\})$ and
\begin{equation}\label{eq:jumpv_bdd}
\liminf_{\eps \to 0^+} \mathcal{F}_{\eps}(u_{\eps};D,E)
	\geq \int_{E_{e_i}} \ul{c} L \Haus{0}(Sv^y) \, dy.
\end{equation}

The right-hand side of \eqref{eq:jumpv_bdd} is finite, so Proposition \ref{prop:Bv_bdy} implies that $v \in BV(E;\{\alpha,\beta\})$, and that $Sv^y$ agrees with $Sv \cap E^y$ for a.e. $y \in E_{e_i}$.
\end{proof}

To prove compactness in the general case, i.e., where $\Omega$ is not flat, we introduce the notion of isometry defect following \cite{ABS}.


\begin{defi}[isometry defect]\label{def:isom_def}

Given $A_1,A_2 \subset \R^N$ open sets and a bi-Lipschitz homeomorphism $\psi: \ol{A_1} \rightarrow \ol{A_2}$ of class $C^2(\ol{A}_i;\R^N)$, the isometry defect $\delta(\psi)$ of $\psi$ is the smallest constant $\delta$ such that
$$
\esssup_{x \in A_1} \left\{\dist \bigl(D\psi(x),O(N)\bigr) + \dist \bigl(D^2\psi(x),0\bigr) \right\} \leq \delta,
$$
where $O(N) := \left\{ A : \R^N \to \R^N \text{ linear mappings}, A A^T = \I_N \right\}$. 

\end{defi}

\begin{prop}\label{prop:flatbdy}

Let $\Omega$, $W$, and $V$ satisfy the hypotheses of Theorem \ref{thm:compactness}.
Given $A_1,A_2 \subset \R^N$ open sets and a bi-Lipschitz homeomorphism $\psi: \ol{A_1} \rightarrow \ol{A_2}$ of class $C^2(\ol{A}_i;\R^N)$ such that $\psi$ has finite isometry defect and maps a set $A_1' \subset \partial A_1$ onto $A_2' \subset \partial A_2$.
Then for every $u \in H^2(A_2)$ there holds
\begin{equation}\label{eq:flatbdy}
\mathcal{F}_{\eps}(u;A_2,A_2')
       \geq \bigl(1-\delta(\psi)\bigr)^{N+4} {\cal F}_{\eps}(u \circ \psi;A_1,A_1')
       - \delta(\psi) \bigl(1-\delta(\psi)\bigr)^{2} \eps^3 \int_{A_2} \left( (\bigl| D^2u\bigr|\bigl|Du\bigr| + \delta(\psi) \bigl|Du\bigr|^2\right) dx.
\end{equation}
     
\end{prop}

\begin{prop}\label{prop:flatbdy2}

Let $\Omega \subset \R^N$ be an open and bounded set of class $C^2$ and let $D_r := \{ x \in \R^N: |x|<r, x_N>0\}$. 
Then for every $x \in \partial \Omega$, there exists $r_x>0$ such that for every $0<r<r_x$, there exists a bi-Lipschitz homeomorphism $\psi_r: \ol{D_r} \rightarrow \ol{\Omega\cap B(x;r)}$ such that
\begin{enumerate}
\item[(i)] $\psi_r$ maps $D_r$ onto $\Omega \cap B(x;r)$ and $E_r:= B_r \cap \{x_N=0\}$ onto $\partial \Omega \cap B(x;r)$;

\item[(ii)] $\psi_r$ is of class $C^2$ in $D_r$ and $\|D\psi_r - \I_N\|_{\infty}+\|D^2\psi_r\|_{\infty} \leq \delta_r$, where $\delta_r \xrightarrow{r\rightarrow 0^+} 0^+$.
\end{enumerate}
\end{prop}

For a proof of Propositions \ref{prop:flatbdy} and \ref{prop:flatbdy2} we refer to \cite{ABS}.
We now turn to the proof of Theorem \ref{thm:compactness_boundary}.

\begin{proof}[Proof of Theorem \ref{thm:compactness_boundary}]

In view of Proposition \ref{prop:flatbdy2} and a simple compactness argument we can cover $\partial \Omega$ with finitely many balls $B^i$ centered on $\partial \Omega$ so that $\Omega \cap B^i$ is the image of a half-ball under a map $\psi^i$ with isometry defect smaller than $1$.
Hence it suffices to show that the sequence $\{Tu_{\eps}\}$ is precompact in $L^2(\partial\Omega \cap B^i)$ for every $i$.

Fix $i$ and let $\til{u}_{\eps} := u_{\eps} \circ \psi^i$.
Since the isometry defect of $\psi^i$ is smaller than $1$, Proposition \ref{prop:flatbdy} implies that $\sup_{\eps} \mathcal{F}_{\eps}(\til{u}_{\eps};D_r,E_r) < \infty$.
Hence the precompactness of the traces $Tu_{\eps}$ in $L^2(\partial\Omega \cap B^i)$ is a consequence of the precompactness of the traces $T\til{u}_{\eps}$ in $L^2(E_r)$, which follows from Proposition \ref{prop:compactness_bdy_flat}.
This completes the proof.
\end{proof}

We are finally ready to prove Theorem \ref{thm:compactness}.

\begin{proof}[Proof of Theorem \ref{thm:compactness}]

Let $\{u_{\eps}\} \subset H^2(\Omega)$ be a sequence such that $C := \sup_{\eps} \mathcal{F}_{\eps}(u_{\eps}) < \infty$.
Then, by Theorem \ref{thm:compactness_interior}, we may find a subsequence $u_{\eps_n} \in H^2(\Omega)$ and a function $u \in BV(\Omega;\{a,b\})$ such that $u_{\eps_n} \to u \text{ in } L^2(\Omega)$.

On the other hand, by applying Theorem \ref{thm:compactness_boundary} to the sequences $\{\eps_{n}\}$ and $\{u_{\eps_n}\}$, which still satisfy $C = \sup_{n} \mathcal{F}_{\eps_{n}}(u_{\eps_n}) < \infty$,
we may find a further subsequence $\{u_{\eps_{n_k}}\}$ of $\{u_{\eps_n}\}$ and a function $v \in BV(\partial \Omega;\{\alpha,\beta\})$ such that $Tu_{\eps_{n_k}} \to v \text{ in } L^2(\partial\Omega)$.
Note that we still have $u_{\eps_{n_k}} \to u \text{ in } L^2(\Omega)$.
This completes the proof.
\end{proof}


\subsection{Lower bound in $\R^N$}\label{ssec:LB_ND}

Before proving the lower bound estimate in the general $N$-dimensional case, we state an auxiliary result.

\begin{lem}\label{lem:measures}

Let $\mu, \mu^1$, and $\mu^2$ be nonnegative finite Radon measures on $\R^N$, such that $\mu^1$ and $\mu^2$ are mutually singular, and $\mu(B) \geq \mu^i(B)$ for $i=1,2$, and for any open ball $B$ such that $\mu(\partial B)=0$.

Then for any Borel set $E$, $\mu(E) \geq \mu^1(E) + \mu^2(E)$.
\end{lem}


\begin{proof}[Proof of Theorem \ref{thm:critical}(i)]

We now have all the necessary auxiliary results to prove the lower bound estimate for the critical regime.

Consider a sequence $\{u_{\eps}\} \subset H^2(\Omega)$ and two functions $u \in BV(\Omega;\{a,b\})$ and $v \in BV(\partial \Omega;\{\alpha,\beta\})$ such that $u_{\eps} \to u$ in $L^2(\Omega)$ and $Tu_{\eps} \to v$ in $L^2(\partial\Omega)$.

We claim that
\begin{equation}\label{eq:LB}
\liminf_{\eps \to 0^+} {\cal F}_{\eps}(u_{\eps};\Omega) 
    \geq m \Per_{\Omega}(E_a) 
    + \sum_{z=a,b} \sum_{\xi=\alpha,\beta} \sigma(z,\xi) \Haus{N-1}\bigl( \{Tu=z\} \cap \{v=\xi\}\bigr)
    + \ul{c} L \Per_{\partial \Omega}(F_{\alpha}).
\end{equation}

Without loss of generality, we may assume that
\begin{equation}\label{eq:finiteF}
\infty > \liminf_{\eps\to 0^+} \mathcal{F}_{\eps}(u_{\eps};\Omega)
	= \lim_{\eps\to 0^+} \mathcal{F}_{\eps}(u_{\eps};\Omega).
\end{equation}

For every $\eps>0$ we define a measure $\mu_{\eps}$ for all Borel sets $E \subset \R^N$ by
$$
\mu_{\eps}(E) := 
     \eps^{3} \int_{\Omega\cap E} |D^{2}u_{\eps}|^{2} \, dx 
     + \frac{1}{\eps} \int_{\Omega\cap E} W (u_{\eps}) \, dx
     + \lambda_{\eps} \int_{\partial \Omega \cap E} V(Tu_{\eps}) \, d\Haus{N-1}.
$$

Since $\mu_{\eps} = {\cal F}_{\eps}(u_{\eps})$, it follows by \eqref{eq:finiteF} that by taking a subsequence (not relabeled), we obtain a finite measure $\mu$ such that $\mu_{\eps} \stackrel{\star}{\rightharpoonup} \mu$ in the sense of measures.

For every Borel set $E \subset \R^N$ define the measures:
\begin{align*}
& \mu^1(E) := m \Per_{\Omega \cap E}(E_a);\\
& \mu^2(E) := \sum_{z=a,b} \sum_{\xi=\alpha,\beta} \sigma(z,\xi) \Haus{N-1}\bigl( \{Tu=z\} \cap \{v=\xi\} \cap E\bigr); \\
&\mu^3(E) := \ul{c} L \Per_{\partial \Omega\cap E}(F_{\alpha}). 
\end{align*}

These three measures are mutually singular and so, by Lemma \ref{lem:measures}, \eqref{eq:LB} is a consequence of $\mu(B) \geq \mu^i(B)$ for $i=1,2,3$ for any ball $B$ with $\mu(\partial B)=0$, which we prove next.


Take $B$ an open ball such that $\mu(\partial B)=0$.

Using a slicing argument as in Theorem \ref{thm:compactness_interior} (see \eqref{eq:fm_slicing} for $N=2$) and Fatou's lemma, we have
\begin{align*}
\mu(B) 
  & = \lim_{\eps\to 0^+} \mu_{\eps}(B) 
  \geq \int_{\Omega_e \cap B} \liminf_{\eps\to 0^+} F_{\eps}(u_{\eps}^y;B^y) \, d\Haus{N-1}(y) \\
  & \geq \int_{\Omega_e \cap B} \left[ m \Haus{0}(Su_{\eps}^y \cap B^y) + \int_{\partial B^y}\sigma(Tu^y(s),v^y(s)) d\Haus{0}(s) \right] \, d\Haus{N-1}(y) \\
  & \geq m \Per_{\Omega \cap B}(E_a) + \int_{\partial\Omega \cap B} \sigma(Tu(s),v(s)) d\Haus{N-1}(s)
   = \mu^1(B) + \mu^2(B),
\end{align*}
where we have used Theorem \ref{thm:interior1D} and Proposition \ref{prop:Bv_bdy}.

By Section \ref{ssec:manifold}, the jump set of $v$, $Sv$, is $(N-2)$-rectifiable.
Hence by the Lebesgue decomposition theorem, the Radon-Nikodym theorem, and the Besicovitch derivation theorem, for $\Haus{N-2}$-a.e. $x\in Sv$,
\begin{equation}\label{eq:dermuHaus}
\frac{d\mu}{d\Haus{N-2}\restr{Sv}}(x) = \lim_{r \to 0^+} \frac{\mu\bigl(B(x;r)\bigr)}{\Haus{N-2}\bigl(B(x;r)\cap Sv\bigr)} \in \R.
\end{equation}

Fix a point $x\in Sv$ for which \eqref{eq:dermuHaus} holds and that has density $1$ for $Sv$ with respect to the $\Haus{N-2}$ measure.
Take $r>0$ such that $\mu\bigl(\partial B(x;r)\bigr) = 0$.
Find $\psi_r$ as in Proposition \ref{prop:flatbdy2} and set $\ol{u_{\eps}} := u_{\eps} \circ \psi_r$ and $\ol{v} := v \circ \psi_r$.
Then $\ol{v} \in BV(E_r;\{\alpha,\beta\})$ and $Tu_{\eps} \to \ol{v}$ in $L^2(E_r)$, where $E_r$ is defined in Proposition \ref{prop:flatbdy2}.
Since $\mu\bigl(\partial B(x;r)\bigr)=0$, we have
\begin{align*}
\mu\bigl(B(x;r)\bigr) 
   & = \lim_{\eps} \mu_{\eps}\bigl(B(x;r)\bigr)
   = \lim_{\eps} {\cal F}_{\eps}\bigl(u_{\eps};\Omega \cap B(x;r), \partial \Omega \cap B(x;r) \bigr) \\
   & \geq \left(1-\delta(\psi_r)\right)^{N+4} \liminf_{\eps} \int_{(E_r)_e} G_{\eps}(Tu_{\eps}^y;E^y_r) \, d\Haus{N-2}(y) \\
   & \geq \ul{c} L \left(1-\delta(\psi_r)\right)^{N+4} \int_{(E_r)_e} \Haus{0}(Sv \cap E^y_r) \, d\Haus{N-2}(y).
\end{align*}

Hence,
\begin{align*}
\frac{d\mu}{d\Haus{N-2}\restr{Sv}}(x) 
   		\geq \lim_{r\rightarrow 0^+} \frac{\mu\bigl(B(x;r)\bigr)}{\alpha_{N-2}r^{N-2}}
   		 \geq \ul{c} L \lim_{r\rightarrow 0^+} \mint_{(E_r)_e} \Haus{0}(Sv \cap E^y_r) \, d\Haus{N-2}(y)
   		= \ul{c} L,
\end{align*}
and so
$$
\mu(B)
   \geq \int_{Sv \cap B} \frac{d\mu}{d\Haus{N-2}\restr{Sv}}(x) \, d\Haus{N-2}(x)  
   \geq \ul{c} L \Per_{\partial \Omega \cap B}(F_{\alpha}) 
    = \mu^3(B).
$$

This concludes the proof of the theorem.
\end{proof}

%
%

\subsection{Upper bound}\label{ssec:UB}

In this subsection we will obtain an estimate for the upper bound.

First we prove the result on a smooth setting, i.e., assuming that both $Su$ and $Sv$ are of class $C^2$.
We define a recovery sequence separately in the different regions of Figure \ref{fig:partition}. In Proposition \ref{prop:interior}, we define it on $A_2$, then we construct the recovery sequence on $A_1$ in Proposition \ref{prop:interior_boundary} and in Corollary \ref{cor:interior_boundary} we glue the last two sequences together to make $\{\ol{u}\}_n$.
Then in Proposition \ref{prop:boundary_flat}, on the setting of a flat domain where $Sv$ has also been flattened, we first construct the recovery sequence on $T_1$ and then glue it to the previously constructed sequence $\{\ol{u}_n\}$ on $T_2$.
In Proposition \ref{prop:recoveryC2} we adapt the sequence of Proposition \ref{prop:boundary_flat} to a general domain, but still under smooth assumptions.

Finally, using a diagonalization argument, we prove the upper bound result without regularity conditions.

\begin{figure}[!htbp]
\begin{center}
\includegraphics*[width=200pt]{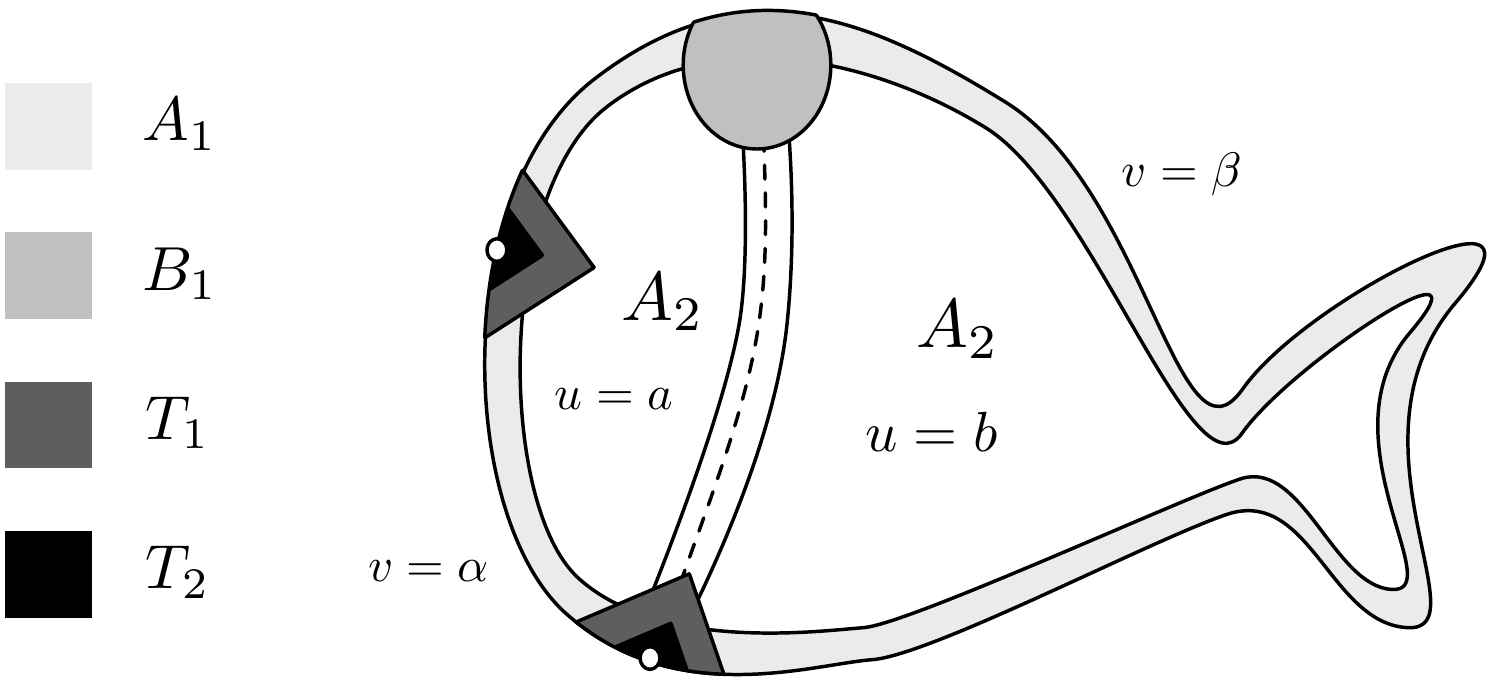}
\caption{Partition of $\Omega$ for the construction of the recovery sequence.}\label{fig:partition}
\end{center}
\end{figure}

In what follows, given a set $E \subset \R^N$ and $\rho>0$ we denote by $E_{\rho}$ the set $E_{\rho} := \{x \in \R^N : \dist(x,E)<\rho\}$.

\begin{prop}\label{prop:interior}

Let $W:\R \to [0,\infty)$ satisfy $(H^W_1)-(H^W_2)$, let $\eps_n \to 0^+$, let $\eta>0$, and let $u \in BV(\Omega;\{a,b\})$ be such that $Su$ is an $N-1$ dimensional manifold of class $C^2$.
Then there exists a sequence  $\{z_{n}\} \subset  H^2(\Omega)$ such that $z_{n} \to u$  in $L^2(\Omega)$,
\begin{align}
& z_n = u \text{ in } \Omega \bs (Su)_{C\eps_n},\label{eq:zn_eq_u} \\
& \|z_n\|_{\infty} \leq C, \qquad
\|\grad z_n\|_{\infty} \leq \frac{C}{\eps_n}, \qquad
\|\grad^2 z_n\|_{\infty} \leq \frac{C}{\eps_n^2}, \label{eq:zn_bounds} \\
\intertext{and}
& \mathcal{F}_{\eps_n}(z_{n};\Omega,\emptyset) \leq (m+\eta) \Haus{N-1}(Su) + o(1),\label{eq:energy_mz}
\end{align}
where $m$ is the constant defined in \eqref{eq:def_m} and $C>0$.
\end{prop}

\begin{proof}

By the definition of $m$, we may find $R>0$ and a function $f \in H_{\loc}^{2}(\R)$ such that $f(-t)=a$ and $f(t) = b$ for all $t \geq R$, and
\begin{equation}\label{eq:func_m}
\int_{-R}^{R} |f''(t)|^{2} + W\bigl(f(t)\bigr) \, dt \leq m + \eta.
\end{equation}

Since $Su$ is a manifold of class $C^2$ in $\R^N$, there exists $\delta_0>0$ such that for all $0<\delta\leq\delta_0$ the points in the tubular neighborhood $U_{\delta}:= \{ x \in \R^N: \dist(x,Su)<\delta\}$ of the manifold $Su$ admit a unique smooth projection onto $Su$.
Define the function $z_n: \Omega \to \R$ by
$$
z_n(x) := \begin{cases}
f\left( \frac{d_u(x)}{\eps_n}\right)
	& \text{ if } x \in U_{R\eps_n} \cap \Omega, \\
a	& \text{ if } x \in E_a \bs U_{R\eps_n}, \\
b	& \text{ if } x \in \Omega \bs (E_a \cup U_{R\eps_n}),
\end{cases}
$$
where $d_u: \R^N \to \R$ is the signed distance to $Su$, negative in $E_a$ and positive outside $E_a$ and where we recall that $E_a := \{ x \in \Omega: u(x) = a\}$.

We then have
$$
\mathcal{F}_{\eps_n}(z_n; \Omega, \emptyset)
	 = \int_{\Omega} \left[ \eps_n^3 \left| \frac{1}{\eps_n^2} f''\bigl( {\ts \frac{d_u(x)}{\eps_n}}\bigr) \grad d_u (x)\times \grad d_u(x) + \frac{1}{\eps_n} f' \bigl( {\ts \frac{d_u(x)}{\eps_n}}\bigr)H_u(x) \right|^2 
	 	+ \frac{1}{\eps_n} W\left(f\bigl( {\ts \frac{d_u(x)}{\eps_n}}\bigr) \right)\right] \, dx,
$$
where $H_u$ is the Hessian matrix of $d_u$.
Change variable via the diffeomorphism $x := \psi_1(y,t)$, where $\psi_1: Su \times (-\delta_0,\delta_0) \to U_{\delta_0}$ is defined by $\psi_1(y,t) := y + t \nu_u(y)$, with $\nu_u(y)$ the normal vector to $Su$ at $y$ pointing away from $E_a$.
Let $J_u(y,t)$ denote the Jacobian of this map.
Then
\begin{multline*}
\mathcal{F}_{\eps_n}(z_n; \Omega, \emptyset)
	 \leq \frac{1}{\eps_n} \int_{Su} \int_{-R\eps_n}^{R\eps_n} \left[ \left|f''\bigl( {\ts \frac{t}{\eps_n}}\bigr)\right|^2 |\grad d_u(\psi_1(y,t)) |^2  + W\left(f\bigl( {\ts \frac{t}{\eps_n}}\bigr) \right) \right] J_u(y,t) \, dt \, d\Haus{N-1}(y) \\
	 \qquad 	+  \eps_n \int_{Su} \int_{-R\eps_n}^{R\eps_n} \left|f'\bigl( {\ts \frac{t}{\eps_n}}\bigr) \right|^2 |H_u(\psi_1(y,t))|^2 J_u(y,t) \, dt \,d\Haus{N-1}(y) \\
	 \qquad + C \int_{Su} \int_{-R\eps_n}^{R\eps_n} \left|f''\bigl( {\ts \frac{t}{\eps_n}}\bigr)\right| \left|f'\bigl( {\ts \frac{t}{\eps_n}}\bigr) \right| |\grad d_u(\psi_1(y,t))|^2 |H_u(\psi_1(y,t))| J_u(y,t) \, dt \, d\Haus{N-1}(y),
\end{multline*}
which reduces to
\begin{multline*}
\mathcal{F}_{\eps_n}(z_n; \Omega, \emptyset)
	 \leq \frac{1}{\eps_n} \int_{Su} \int_{-R\eps_n}^{R\eps_n} \left[ \left|f''\bigl( {\ts \frac{t}{\eps_n}}\bigr)\right|^2  + W\left(f\bigl( {\ts \frac{t}{\eps_n}}\bigr) \right) \right] J_u(y,t) \, dt \, d\Haus{N-1}(y) \\
	 \qquad 	+  C  \int_{Su} \int_{-R\eps_n}^{R\eps_n} \left[ \eps_n \left|f'\bigl( {\ts \frac{t}{\eps_n}}\bigr) \right|^2  + \left|f''\bigl( {\ts \frac{t}{\eps_n}}\bigr)\right|^2 \left|f'\bigl( {\ts \frac{t}{\eps_n}}\bigr) \right|^2 \right] \, dt \, d\Haus{N-1}(y)
	=: I_1 + I_2,
\end{multline*}
where we took into account the facts that the gradient of the distance is $1$, and the Jacobian $J_u$ and the Hessian $H_u$ of the distance are uniformly bounded.
We have
\begin{align*}
I_1
	& \leq \Biggl(\sup_{\substack{y \in Su, \\t \in (-R\eps_n,R\eps_n)}} J_u(y,t) \Biggr) \frac{1}{\eps_n} \int_{Su} \int_{-R\eps_n}^{R\eps_n} \left[ \left|f''\bigl( {\ts \frac{t}{\eps_n}}\bigr)\right|^2  + W\left(f\bigl( {\ts \frac{t}{\eps_n}}\bigr) \right) \right] \, dt \, d\Haus{N-1}(y) \\
	& = \bigl(1+o(1)\bigr) \int_{Su} \int_{-R}^{R} \left[ \left|f''(s)\right|^2  + W\left(f(s) \right) \right] \, ds \, d\Haus{N-1}(y), \\
	& \leq \bigl(1+o(1)\bigr) (m+\eta) \Haus{N-1}(Su),
\end{align*}
where we used \eqref{eq:func_m} and the fact that since $Su$ is a compact manifold, $J_u(y,t)$ converges to $1$ uniformly as $t \to 0$.

On the other hand, by \eqref{eq:func_m},
$$
I_2
	\leq C  \eps_n \int_{-R}^{R} \left[ \eps_n \left|f'(s) \right|^2 
	+ \left|f''(s)\right| \left|f'(s) \right| \right] \, ds \leq C \eps_n.
$$

We conclude that $\mathcal{F}_{\eps_n}(z_n; \Omega, \emptyset) \leq (m+\eta) \Haus{N-1}(Su) + o(1)$.
This completes the proof.
\end{proof}

\begin{prop}\label{prop:interior_boundary}

Let $W:\R \to [0,\infty)$ satisfy $(H^W_1)-(H^W_2)$, let $V:\R \to [0,\infty)$ satisfy $(H^V_1)-(H^V_3)$ Let $\eps_n \to 0^+$ be such that $\eps_n\lambda_n^{\frac23} \to L \in (0,\infty)$, let $\eta>0$, let $\Omega_{\delta} := \{ x \in \Omega: \dist(x,\partial\Omega)<\delta\}$ for $\delta>0$, and let $u \in BV(\Omega;\{a,b\})$ and $v \in BV(\partial\Omega;\{\alpha,\beta\})$, with $Su$ an $N-1$ manifold of class $C^2$ such that $\Haus{N-1}(\partial\Omega \cap \ol{Su}) = 0$ and $Sv$ an $N-2$ manifold of class $C^2$ .
Then there exist $R=R(\eta)>0$ and a sequence $\{v_{n}\} \subset H^2(\Omega_{R\eps_n})$ such that $Tv_n \to v$ in $L^2(\partial\Omega)$,
\begin{align}
& \Leb{N}\left( \left\{ x \in \Omega_{R\eps_n} \bs \ol{\Omega}_{\frac{R\eps_n}{2}}: v_n(x) \neq u(x) \right\} \right) \leq C \eps_n^2, \label{eq:small_meas_inters} \\
& \|v_n\|_{\infty} \leq C, \qquad
\|\grad v_n\|_{\infty} \leq \frac{C}{\eps_n}, \qquad
\|\grad^2 v_n\|_{\infty} \leq \frac{C}{\eps_n^2}, \label{eq:vn_bounds}
\end{align}
and
\begin{equation}\label{eq:F_sigma1}
\mathcal{F}_{\eps_n}(v_{n};\Omega_{R\eps_n},\emptyset) 
	\leq \sum_{z=a,b} \sum_{\xi=\alpha,\beta} (\sigma(z,\xi) + \eta) \Haus{N-1}(\{ x \in \partial\Omega: Tu(x)=z, v(x)=\xi\}) + o(1),
\end{equation}
where $\sigma(z,\xi)$ is the constant defined in \eqref{eq:def_sigma}.
\end{prop}

\begin{proof}

By the definition of $\sigma(\cdot,\cdot)$, for every $z \in \{a,b\}$ and $\xi \in \{\alpha,\beta\}$ there exist $R_{z\xi}>0$ and $g_{z\xi} \in H^2_{\loc}(\R)$ such that $g_{z\xi}(0)=z$, $g_{z\xi}(x) = \xi$ for all $x \geq R_{z\xi}$, and
\begin{equation}\label{eq:sigma_gs}
\int_0^{R_{z\xi}} \left[ |g_{z\xi}''(x)|^2 + W(g_{z\xi}(x)) \right] \, dx \leq \sigma(z,\xi) + \eta.
\end{equation}

Define $R := \max \{ \ol{R}, R_{a\alpha}, R_{b\alpha}, R_{a\beta}, R_{b\beta}\}$, where $\ol{R}$ is the number $R$ given in the previous proposition.
Since $\partial\Omega$ is an $N-1$ manifold of class $C^2$, there exists $\delta_0>0$ such that every point $x \in \Omega_{\delta_0}$ admits a unique projection $\pi(x)$ onto $\partial \Omega$ and the map $x \in \Omega_{\delta_0} \mapsto \pi(x)$ is of class $C^2$.
Hence we may partition $\Omega_{\delta_0}$ as follows
$$
\Omega_{\delta_0} = \left( \bigcup_{z=a,b} \bigcup_{\xi=\alpha,\beta} A_{z\xi} \right) \cup Su \cup \pi^{-1}(Sv),
$$
where $A_{z\xi} := \left\{ x \in \Omega_{\delta_0} \bs \left( Su \cup \pi^{-1}(Sv) \right): (Tu)(\pi(x))=z, v(\pi(x)) = \xi \right\}$.
Let $n$ be so large that $R\eps_n \leq \delta_0$ and define $g_n:\Omega_{R\eps_n}\to \R$ as follows
$$
g_n(x) := \begin{cases}
g_{z\xi}\left( \frac{d(x)}{\eps_n}\right)
	& \text{ if } x \in A_{z\xi} \cap \Omega_{R\eps_n} \text{ for some } z \in \{a,b\} \text{ and } \xi \in \{\alpha,\beta\}, \\
0	& \text{ if } x \in \left( Su \cup \pi^{-1}(Sv) \right) \cap \Omega_{R\eps_n},
\end{cases}
$$
where, as before, $d: \ol{\Omega} \to [0,\infty)$ is the distance to $\partial\Omega$.

Note that the functions $g_n$ are discontinuous across $ \left( Su \cup \pi^{-1}(Sv) \right) \cap \Omega_{R\eps_n}$, and so they are not admissible for $\mathcal{F}_{\eps_n}$.
To solve this problem, let $\varphi \in C^{\infty}\left( (0,\infty);[0,1]\right)$ be such that $\varphi \equiv 0$ in $\bigl(0,\frac13\bigr)$ and $\varphi \equiv 1$ in $\bigl(\frac12,\infty\bigr)$, and let $d_u:\ol{\Omega}_{\delta_0} \to [0,\infty)$ and $d_v:\ol{\Omega}_{\delta_0} \to [0,\infty)$ denote the distance to $Su$ and to $\pi^{-1}(Sv)$, respectively.
Since $Su$ is an $N-1$ manifold of class $C^2$, it follows that $d_u$ is of class $C^2$ in a neighborhood $P_1 := \{x \in \Omega_{\delta_0}: d_u(x) < \delta_1\}$ of $Su$.
Similarly, since $Sv$ is an $N-2$ manifold of class $C^2$ by taking $\delta_0$ smaller, if necessary, we may assume that $\pi^{-1}(Sv)$ is an $N-1$ dimensional manifold of class $C^2$ and thus $d_v$ is of class $C^2$ in a neighborhood $P_2:=\{x \in \Omega_{\delta_0}: d_v(x) < \delta_2\}$ of $\pi^{-1}(Sv)$.
Let $n$ be so large that $R\eps_n < \frac 13 \min \{\delta_1,\delta_2\}$ and for $x \in \Omega_{R\eps_n}$ define
$$
v_n(x) := \varphi\left( \frac{d_u(x)}{R\eps_n}\right)\varphi\left( \frac{d_v(x)}{R\eps_n}\right) g_n(x).
$$
Since $\varphi \equiv 0$ in $\bigl(0,\frac13\bigr)$, it follows that $v_n(x) = 0$ for all $x \in \Omega_{R\eps_n}$ such that $d_u(x) < \frac13R\eps_n$ or $d_v(x) < \frac13R\eps_n$. As $g_n$ is regular away from $Su \cup \pi^{-1}(Sv)$, it follows that $v_n \in H^2(\Omega_{R\eps_n})$.

We claim that $Tv_n \to v$ in $L^2(\partial\Omega)$.
Indeed, since $\Haus{N-1}(\partial\Omega \cap \ol{Su}) = 0$, we know that 
\begin{equation}\label{eq:meas_small}
\Haus{N-1}\left( \left\{ x \in \partial \Omega: d_u(x) < \frac12 R \eps_n\right\}\right) \leq C\eps_n,
\end{equation}
and similarly, since $Sv$ is an $N-2$ manifold contained in $\partial\Omega$, 
\begin{equation}\label{eq:meas_small2}
\Haus{N-1}\left( \left\{ x \in \partial \Omega: d_v(x) < \frac12 R \eps_n\right\}\right) \leq C\eps_n.
\end{equation}

On the other hand, if $x \in \partial\Omega$ is such that $d_u(x) \geq \frac12 R\eps_n$ and $d_v(x) \geq \frac12 R\eps_n$, then $v_n = g_n$ in a neighborhood of $x$, and so by the definition of the sets $A_{z\xi}$ and the fact that $g_{z\xi}(0)=z$, it follows that $v_n(x)=v(x)$.
Hence by \eqref{eq:meas_small} and \eqref{eq:meas_small2}, $\|v_n-v\|_{L^2(\partial\Omega)} \to 0$,
which proves the claim.

It remains to prove \eqref{eq:F_sigma1}.
Let 
\begin{align*}
& L_n := \left\{ x \in \Omega_{R\eps_n}: d_u(x) < {\ts \frac12}R\eps_n \right\},
& M_n := \left\{ x \in \Omega_{R\eps_n}: d_v(x) < {\ts \frac12}R\eps_n \right\}.
\end{align*}

\paragraph{Step 1. }

We begin by estimating $\mathcal{F}_{\eps}$ in the set $\Omega_{R\eps_n} \bs (L_n \cup M_n)$.
Since in this set $v_n=g_n$, we have that
$$
\mathcal{F}_{\eps_n}(v_n;\Omega \bs (L_n \cup M_n), \emptyset)
	\leq \sum_{z=a,b} \sum_{\xi=\alpha,\beta} \mathcal{F}_{\eps_n}\left(g_n; A_{z\xi} \cap \Omega_{R\eps_n} ,\emptyset\right).
$$
Thus it suffices to estimate $\mathcal{F}_{\eps_n}(g_n;A_{z\xi} \cap \Omega_{R\eps_n})$.

Let $A'_{z\xi} := \ol{A}_{z\xi} \cap \partial\Omega$, which satisfies $A_{z\xi}' = \{ x \in \partial\Omega: Tu(x) =z, v(x) =\xi\}$.
We have
$$
\mathcal{F}_{\eps_n}(g_n; A_{z\xi}, \emptyset)
	 = \int_{A_{z\xi}} \left[ \eps_n^3 \left| \frac{1}{\eps_n^2} g_{z\xi}''\bigl( {\ts \frac{d(x)}{\eps_n}}\bigr) \grad d(x) \times \grad d(x) + \frac{1}{\eps_n} g_{z\xi}'\bigl( {\ts \frac{d(x)}{\eps_n}}\bigr) H(x) \right|^2
	  + \frac{1}{\eps_n} W\left(g_{z\xi}\bigl( {\ts \frac{d(x)}{\eps_n}}\bigr) \right)\right] \, dx,
$$
where $H$ is the Hessian matrix of $d$.
Change variable via the diffeomorphism $x := \psi_2(y,t)$, where $\psi_2: \partial\Omega \times \bigl(0,\delta_1\bigr) \to \Omega_{\delta_1}$, defined by $\psi_2(y,t) := y + t \nu(y)$, with $\nu(y)$ the normal vector to $\partial\Omega$ at $y$ pointing to the inside of $\Omega$.
We write $J(y,t)$ the Jacobian of this map.
Then
\begin{multline*}
\mathcal{F}_{\eps_n}(g_n; A_{z\xi}, \emptyset)
	\leq \int_{A'_{z\xi}} \int_{0}^{R\eps_n} \biggl[ \frac{1}{\eps_n}  \left|g_{z\xi}''\bigl( {\ts \frac{t}{\eps_n}}\bigr)\right|^2 |\grad d(\psi_2(y,t)) |^2  + \frac{1}{\eps_n}  W\left(g_{z\xi}'\bigl( {\ts \frac{t}{\eps_n}}\bigr) \right)
	+  \eps_n \left|g_{z\xi}''\bigl( {\ts \frac{t}{\eps_n}}\bigr) \right|^2 |H(\psi_2(y,t))|^2 \\
	+  C \left|g_{z\xi}''\bigl( {\ts \frac{t}{\eps_n}}\bigr)\right|\left|g_{z\xi}'\bigl( {\ts \frac{t}{\eps_n}}\bigr) \right| |\grad d(\psi_2(y,t))|^2 |H(\psi(y,t))|  \biggr] J(y,t) \, dt \, d\Haus{N-1}(y),
\end{multline*}
which reduces to
\begin{multline*}
\mathcal{F}_{\eps_n}(g_n; A_{z\xi}, \emptyset)
	\leq \biggl\{ \frac{1}{\eps_n} \int_{A_{z\xi}'} \int_{0}^{R\eps_n} \left[ \left|g_{z\xi}''\bigl( {\ts \frac{t}{\eps_n}}\bigr)\right|^2  + W\left(g_{z\xi}\bigl( {\ts \frac{t}{\eps_n}}\bigr) \right) \right] J(y,t) \, dt \, d\Haus{N-1}(y) \\
	+  C  \int_{A_{z\xi}'} \int_{0}^{R\eps_n} \left[ \eps_n \left|g_{z\xi}'\bigl( {\ts \frac{t}{\eps_n}}\bigr) \right|^2  + \left|g_{z\xi}'\bigl( {\ts \frac{t}{\eps_n}}\bigr)\right| \left|g_{z\xi}'\bigl( {\ts \frac{t}{\eps_n}}\bigr) \right| \right] \, dt \, d\Haus{N-1}(y) \biggr\}
	 =: I_1 + I_2,
\end{multline*}
where we took into account the facts that the gradient of the distance is $1$, and the Jacobian $J$ and the Hessian $H$ of the distance are uniformly bounded.
We have
\begin{align*}
I_1
	& \leq \Biggl(\sup_{\substack{y \in A'_{z\xi}, \\t \in (0,R\eps_n)}} J(y,t) \Biggr) \frac{1}{\eps_n} \int_{A'_{z\xi}} \int_{0}^{R\eps_n} \left[ \left|g_{z\xi}''\bigl( {\ts \frac{t}{\eps_n}}\bigr)\right|^2  + W\left(g_{z\xi}\bigl( {\ts \frac{t}{\eps_n}}\bigr) \right) \right] \, dt \, d\Haus{N-1}(y) \\
	& \leq \bigl( 1 + o(1) \bigr) (\sigma(z,\xi)+\eta) \Haus{N-1}(\{Tu=z,v=\xi\}),
\end{align*}
where we used the fact that since $\partial\Omega$ is a compact manifold, $J(y,t)$ converges to $1$ uniformly as $t \to 0$.
On the other hand
$$
I_2
	\leq C  \eps_n \int_{0}^{R} \left[ \eps_n \left|g_{z\xi}'(s) \right|^2 
	+ \left|g_{z\xi}''(s)\right|^2 \left|g_{z\xi}'(s) \right|^2 \right] \, ds \leq C \eps_n.
$$

We conclude that $\mathcal{F}_{\eps_n}(g_n; A_{z\xi}, \emptyset)\leq (\sigma(z,\xi)+\eta) \Haus{N-1}(\{Tu=z,v=\xi\}) + o(1)$.

\paragraph{Step 2. } We estimate the energy in $L_n \cup M_n$.

We have
\begin{multline*}
\mathcal{F}_{\eps_n}(v_n; L_n \bs M_n, \emptyset)
	=  \int_{L_n \bs M_n} \biggl[ \eps_n^3 \left| \varphi\bigl( {\ts \frac{d_u(x)}{R\eps_n}}\bigr)g_n''(x) + \frac{2}{R\eps_n}g_n'(x) \varphi'\bigl( {\ts \frac{d_u(x)}{R\eps_n}}\bigr) \grad d_u \times \grad d_u \right. \\
	  \left. + \frac{1}{R^2\eps_n^2} g_n(x) \varphi''\bigl( {\ts \frac{d_u(x)}{R\eps_n}}\bigr) H_u \right|^2 + \frac{1}{\eps_n} W\left( \varphi\bigl({\ts \frac{d_u(x)}{R\eps_n}}\bigr) g_n(x)\right)\biggr] \, dx,
\end{multline*}
where $H_u$ is the Hessian matrix of $d_u$.
Then
\begin{multline*}
 \mathcal{F}_{\eps_n}(v_n; L_n \bs M_n, \emptyset)
	\leq  C\int_{L_n \bs M_n} \biggl[ \eps_n^3 |g_n''(x)|^2 + \eps_n |g_n'(x)|^2 + \frac{1}{R^4\eps_n} |\ol{v}_n(x)|^2 \\
	 + \limsup_n \frac{1}{\eps_n} W\left( \varphi\bigl({\ts \frac{d_u(x)}{R\eps_n}}\bigr) g_n(x)\right)\biggr] \, dx
	 \leq C \frac{1}{\eps_n} |L_n| \leq C \eps_n,
\end{multline*}
where we took into account the facts that the Hessian $H_u$ is uniformly bounded, and that $\ol{v}_n$ is uniformly bounded, $g_n'$ is bounded by $\frac{C}{\eps_n}$, and $g_n''$ is bounded by $\frac{C}{\eps_n^2}$.

We conclude that $\mathcal{F}_{\eps_n}(v_n; L_n \bs M_n, \emptyset) = o(1)$.
Similarly, we may prove that $\mathcal{F}_{\eps_n}(v_n; M_n, \emptyset) = o(1)$.
This concludes the proof.
\end{proof}

\begin{cor}\label{cor:interior_boundary}

Let $W:\R \to [0,\infty)$ satisfy $(H^W_1)-(H^W_2)$, let $V:\R \to [0,\infty)$ satisfy $(H^V_1)-(H^V_3)$. Let $\eps_n \to 0^+$ be such that $\eps_n \lambda_n^{\frac23} \to L \in (0,\infty)$, let $\eta>0$, and let $u \in BV(\Omega;\{a,b\})$ and $v \in BV(\partial\Omega;\{\alpha,\beta\})$, with $Su$ an $N-1$ manifold of class $C^2$ such that $\Haus{N-1}(\partial\Omega \cap \ol{Su}) = 0$ and $Sv$ an $N-2$ manifold of class $C^2$ .
Then there exists a sequence $\{\ol{u}_{n}\} \subset H^2(\Omega)$ such that $\ol{u}_{n} \to u$ in $ L^2(\Omega)$, $T\ol{u}_{n} \to v$ in $L^2(\partial\Omega)$,
\begin{equation}\label{eq:un_bar_bounds}
\|\ol{u}_n\|_{\infty} \leq C, \qquad
\|\grad \ol{u}_n\|_{\infty} \leq \frac{C}{\eps_n}, \qquad
\|\grad^2 \ol{u}_n\|_{\infty} \leq \frac{C}{\eps_n^2}, 
\end{equation}
and
\begin{equation}\label{eq:F_sigma2}
\mathcal{F}_{\eps_n}(\ol{u}_{n};\Omega,\emptyset) \leq (m+\eta) \Haus{N-1}(Su)
+  \sum_{z=a,b} \sum_{\xi=\alpha,\beta} (\sigma(z,\xi) + \eta) \Haus{N-1}(\{ x \in \partial\Omega: Tu(x)=z, v(x)=\xi\}) + o(1)
\end{equation} 
where $m$ and $\sigma(z,\xi)$ are the constant defined, respectively, in \eqref{eq:def_m} and \eqref{eq:def_sigma}.
\end{cor}

\begin{proof}

Let $\varphi \in C^{\infty}((0,\infty);[0,1])$ be such that $\varphi \equiv 0$ in $\bigl(0,\frac12\bigr)$ and $\varphi \equiv 1$ in $(1,\infty)$ and let
$$
\ol{u}_n(x) := \varphi\bigl( {\ts \frac{d(x)}{R\eps_n}}\bigr) z_n(x)  + \left( 1 - \varphi\bigl( {\ts \frac{d(x)}{R\eps_n}}\bigr) \right) v_n(x),
$$
for $x \in \Omega$, where the functions $z_n$ and $v_n$ are defined, respectively, in Propositions \ref{prop:interior} and \ref{prop:interior_boundary}, $R$ is the number given in the previous proposition, and $d$ is the distance to the boundary.

\begin{figure}[!htbp]
\begin{center}
\includegraphics*[width=125pt]{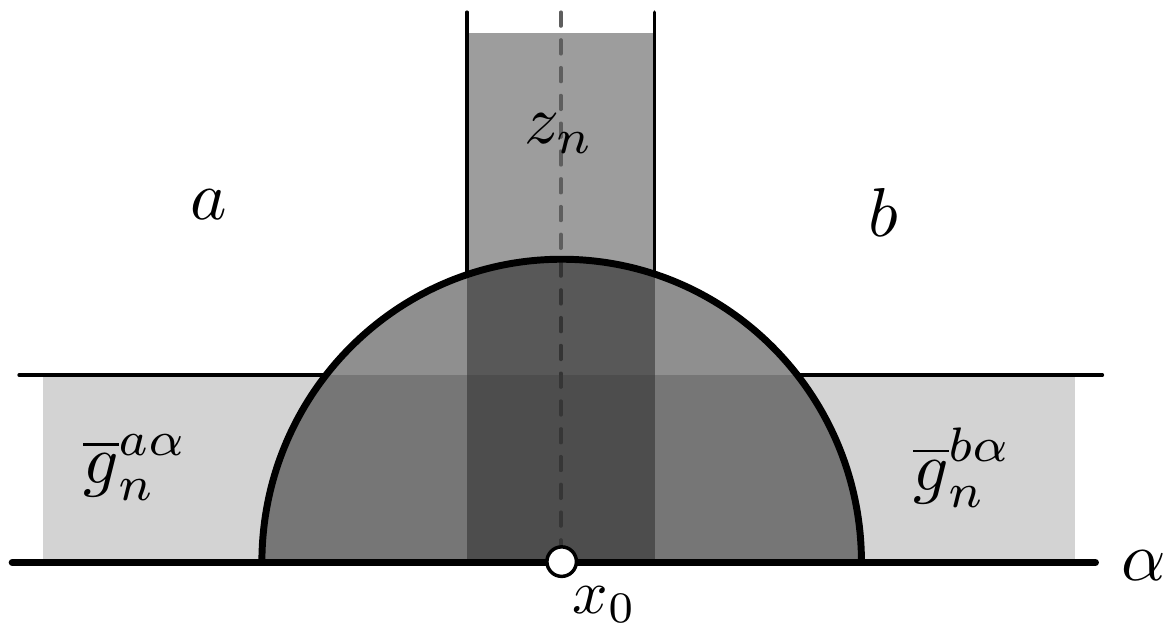}
\caption{Scheme for the gluing of the discontinuity set of $u$ to the boundary $\partial \Omega$ when there is no discontinuity in $v$.}\label{fig:bdy_extension_B}
\end{center}
\end{figure}

Since $T\ol{u}_n = Tv_n$, it follows that $T\ol{u}_n \to v$ in $L^2(\partial\Omega)$.

On the other hand, since $\|v_n\|_{\infty} \leq C$, $\Leb{N}\left( \{ x \in \Omega: d(x) \leq R\eps_n\}\right) \to 0$, and $z_n \to u$ in $L^2(\Omega)$, we have that $\ol{u}_n \to u$ in $L^2(\Omega)$.
Moreover, by \eqref{eq:energy_mz} and \eqref{eq:F_sigma1},
\begin{align*}
\mathcal{F}_{\eps_n}(\ol{u}_n;\Omega;\emptyset)
	& \leq \mathcal{F}_{\eps_n}(z_n;\Omega \bs \Omega_{2R\eps_n};\emptyset)
	+ \mathcal{F}_{\eps_n}(v_n;\Omega_{R\eps_n};\emptyset)
	+ \mathcal{F}_{\eps_n}(\ol{u}_n;P_n;\emptyset) \\
	& \leq (m+\eta) \Haus{N-1}(Su)
      + \sum_{z=a,b} \sum_{\xi=\alpha,\beta} (\sigma(z,\xi) + \eta) \Haus{N-1}(\{ x \in \partial\Omega: Tu(x)=z, v(x)=\xi\}) \\
	& \qquad + \limsup_n \mathcal{F}_{\eps_n}(\ol{u}_n;P_n;\emptyset) + o(1),
\end{align*}
where $P_n := \left\{ x \in \Omega: \frac12 R \eps_n < d(x) < 2 R\eps_n\right\}$.

To estimate the last term, note that by \eqref{eq:energy_mz} and \eqref{eq:small_meas_inters}, 
$
\Leb{N}\bigl( \{ x \in P_n: \ol{u}_n(x) \neq u(x) \} \bigr) \leq C \eps_n^2,
$
and so by the continuity of $W$,
$$
\frac{1}{\eps_n} \int_{P_n} W(\ol{u}_n) \, dx
	= \frac{1}{\eps_n} \left( \max_{B(0;L)} W \right) \Leb{N}\bigl( \{ x \in P_n: \ol{u}_n(x) \neq u(x) \}\bigr)
	\leq C \eps_n 
	\to 0,
$$
where $L:= \sup_n \|\ol{u}_n\|_{\infty}$.

On the other hand, we have that $\grad \ol{u}_n(x) = 0$ and $\grad^2 \ol{u}_n(x) = 0$ for $\Leb{N}$-a.e. $x \in E_n := \{ x \in P_n : \ol{u}_n(x) = u(x) \}$, while for $x \in P_n \bs E_n$,
$$
|\grad^2 \ol{u}_n(x)|^2 
	\leq C \biggl[ \frac{1}{\eps_n^4} | \bigl( |z_n(x)|^2 + |v_n(x)|^2 \bigr)
		+ \frac{1}{\eps_n^2} | \bigl( |\grad z_n(x)|^2 + |\grad v_n(x)|^2 \bigr)
		+ | \bigl( |\grad^2 z_n(x)|^2 + |\grad^2 v_n(x)|^2  \biggr]
	\leq \frac{C}{\eps_n^4},
$$
where we used the bounds on $z_n$ and $v_n$ given in \eqref{eq:zn_bounds} and \eqref{eq:vn_bounds}.
Hence
$$
\eps_n^3 \int_{P_n} |\grad^2 \ol{u}_n|^2 \, dx
	= \eps_n^3 \int_{P_n \bs E_n} |\grad \ol{u}_n|^2 \, dx
	\leq \frac{C}{\eps_n} \Leb{N}(P_n \bs E_n)
	\leq C\eps_n,
$$
which completes the proof.
\end{proof}

\begin{prop}\label{prop:boundary_flat}

Let $W:\R \to [0,\infty)$ satisfy $(H^W_1)-(H^W_2)$, let $V:\R \to [0,\infty)$ satisfy $(H^V_1)-(H^V_3)$.
Let $\eps_n \to 0^+$ be such that $\eps_n \lambda_n^{\frac23} \to L \in (0,\infty)$, let $\eta>0$, let $D_r := \{ x \in \R^N: |x|<r, x_N>0\}$, and let $E_r := \{ (x',0) \in \R^{N-1}\times \R: |x|<r\}$.
Also let $u \in BV(D_r;\{a,b\})$ and $v \in BV(E_r;\{\alpha,\beta\})$, with $Su$ an $N-1$ manifold of class $C^2$ such that $\Haus{N-1}(E_r \cap \ol{Su}) = 0$ and $Sv = \{ x \in E_r: x_{N-1}=0\}$.
Then there exists $\{u_{n}\} \subset  H^2(D_r)$ such that $u_{n} \to u$ in $L^2(D_r)$, $Tu_{n} \to v$ in $L^2(E_r)$, and
$$
\limsup_n \mathcal{F}_{\eps_n}(u_{n};D_r, E_r)
	\leq (m + \eta) \Haus{N-1}(Su)
	+ \sum_{z=a,b} \sum_{\xi=\alpha,\beta} (\sigma(z,\xi)+ \eta) \Haus{N-1}(\{Tu=z,v=\xi\}) + (\ol{c} + \eta) L \Haus{N-2}(Sv),
$$
where $m$, $\sigma$, and $\ol{c}$ are the constants defined in \eqref{eq:def_m}, \eqref{eq:def_sigma}, and \eqref{eq:def_over_c}, respectively.
\end{prop}

\begin{proof}

First we prove the result for $N=2$ and then treat the $N$-dimensional case.

\paragraph{Step 1. } Assume that $N=2$.

\subpar{Substep 1a. }
By the definition of $\ol{c}$ there exists $R>0$ and a function $h \in H_{\loc}^{\frac{3}{2}}(\R)$ satisfying $h(-t) = \alpha$ and $h(t)= \beta $ for all  $t \geq R$ and
\begin{equation}\label{eq:hbdy}
\frac{7}{16} \iint_{\R^2} \frac{\bigl| h'(t) - h'(s)\bigr|^2}{|t-s|^2}  \, dt \, ds + \int_{\R} V\bigl(h(t)\bigr) \, dt  \leq \ol{c} + \eta.
\end{equation}
Define
\begin{equation}\label{eq:def_w}
\ol{w}(t,s) = \frac{1}{2s} \int_{t-s}^{t+s} h(\tau)\,d\tau.
\end{equation}
By Proposition \ref{prop:lifting}, we have that $\ol{w} \in H^2_{\loc}(\R \times (0,\infty))$, $T\ol{w} = h$, and
$$
\iint_{\triangle_R} \bigl| D^2 \ol{w}(t,s)\bigr|^2  \, dt\,ds \leq \frac{7}{16} \int_{-R}^R \int_{-R}^R \frac{\bigl| h'(t) - h'(s)\bigr|^2}{|t-s|^2}  \, dt \, ds,
$$
where $\triangle_{R} := T_{2R}^+-(R,0)$ and $T_{2R}^+ := \bigl\{ (t,s) \in \R^2: 0<s<R, s<t<2R-s\bigr\}$.
For $(x,y) \in \triangle_{R\rho_n}$ define $w_{n}(x,y) := \ol{w}\left( \frac{x}{\rho_{n}}, \frac{y}{\rho_{n}}\right)$, where $\rho_n = \eps_n \lambda_n^{-\frac{1}{3}}$.

Then
\begin{align*}
\mathcal{F}_{\eps}(w_{n}, \triangle_{R\rho_n}, (-R\rho_{n},R\rho_{n}) \times \{0\}) 
	& = \iint_{\triangle_{R\rho_n}}  \left[ \eps_n^3 |\grad^2_{(x,y)} w_{n}(x,y)|^2  + \frac{1}{\eps_n} W(w_{n}(x,y))\right]  \, dx \, dy  + \lambda_{n} \int_{-R\rho_{n}}^{R\rho_{n}} V(Tw_{n}(x)) \, dx \\
	& = \iint_{\triangle_{R\rho_n}}  \left[ \frac{\eps_n^3}{\rho_n^4} \left|\grad^2_{(t,s)} \ol{w}\bigl({\ts \frac{x}{\rho_n},\frac{y}{\rho_n}}\bigr)\right|^2  + \frac{1}{\eps_n} W\left(\ol{w}\bigl({\ts \frac{x}{\rho_n},\frac{y}{\rho_n}}\bigr)\right) \right]\, dx \, dy  \\
& \hspace{7cm} + \lambda_{n} \int_{-R\rho_{n}}^{R\rho_{n}} V\left(T\ol{w}\bigl({\ts \frac{x}{\rho_n},\frac{y}{\rho_n}}\bigr)\right) \, dx \\	
	& = \iint_{\triangle_{R}} \left[  \eps_n \lambda_{n}^{\frac23} |\grad^2_{(t,s)} \ol{w}(t,s)|^2  + \eps_n \lambda_{n}^{-\frac23} W(\ol{w}(t,s)) \right] \, dt \, ds +  \eps_n \lambda_{n}^{\frac23} \int_{-R}^{R} V(T\ol{w}(t)) \, dt \\
	& \leq (L+o(1)) \left[ \frac{7}{16} \int_{-R}^R \int_{-R}^R \frac{\bigl| h'(t) - h'(s)\bigr|^2}{|t-s|^2}  \, dt \, ds + \int_{-R}^{R} V(Th(t)) \, dt \right] + C\eps_n^2,
\end{align*}
where we used the fact that $W$ is continuous and $\|\ol{w}\|_{\infty} \leq C$.
Thus
\begin{equation}\label{eq:estimate1}
\mathcal{F}_{\eps}(w_{n}, \triangle_{R\rho_{n}}, (-R\rho_{n},R\rho_{n}) \times \{0\}) 
	\leq L \left[ \frac{7}{16} \int_{-R}^R \int_{-R}^R \frac{\bigl| h_{\eta}'(t) - h_{\eta}'(s)\bigr|^2}{|t-s|^2}  \, dt \, ds + \int_{-R}^{R} V(Th_{\eta}(t)) \, dt\right] + o(1).
\end{equation}

\subpar{Substep 1b. }
To complete this step, we need to match the function $w_n$ to the function $\ol{u}_n$ given in Corollary \ref{cor:interior_boundary} (with $N=2$ and $\Omega := D_r$).

\begin{figure}[!htbp]
\begin{center}
\begin{tabular}{ccc}
\includegraphics*[width=125pt]{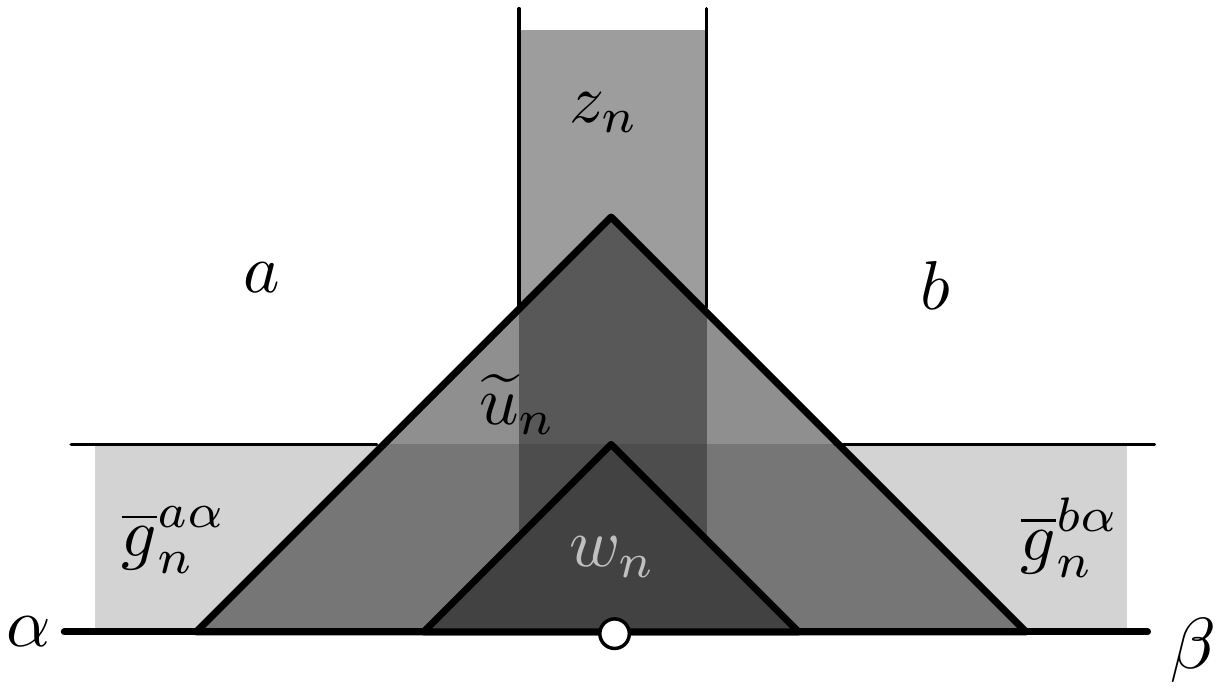}
& \hspace{2cm} & \includegraphics*[width=100pt]{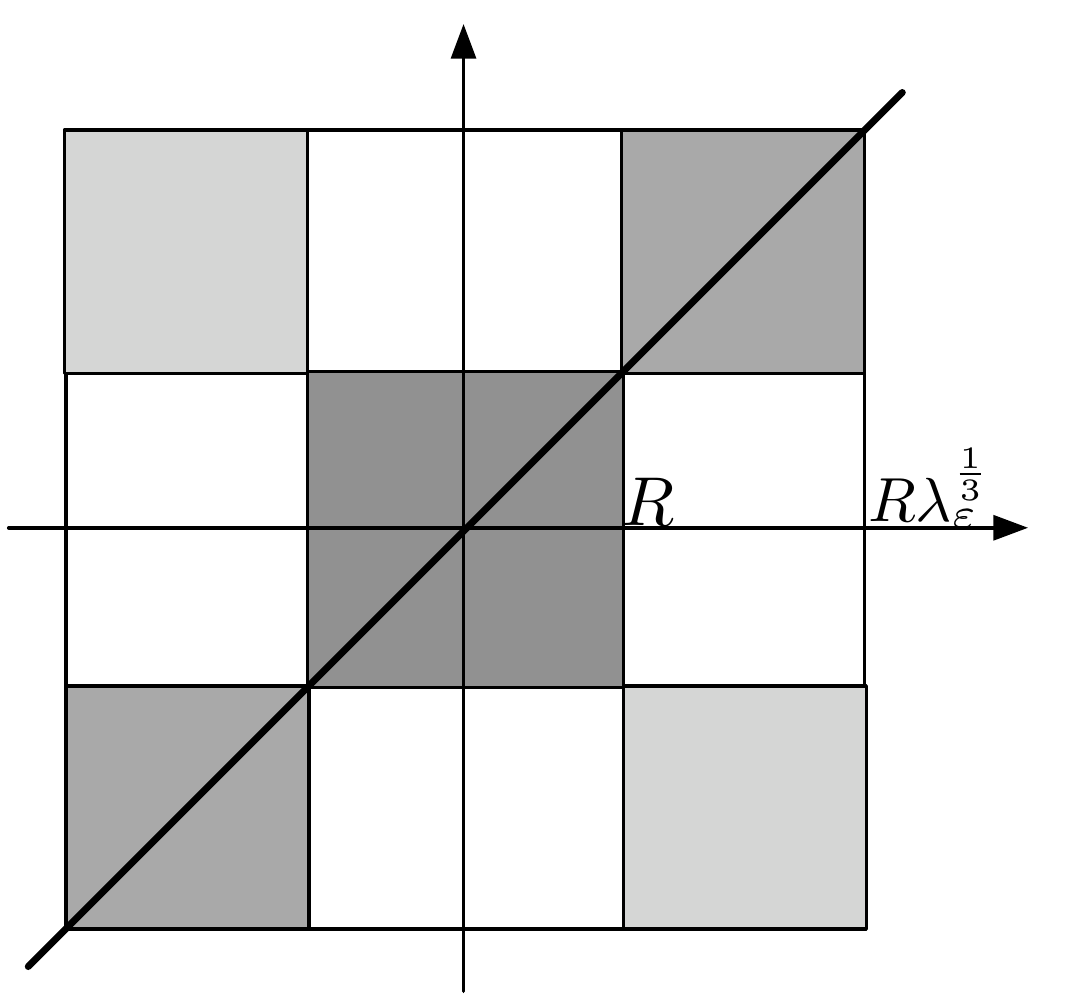} \\
(a) & & (b) \;\;
\end{tabular}
\caption{(a) Close-up view of $T_{2R\eps_n}$ and $T_{2R\rho_n}$.; (b) Domain of integration after change of variables and divided in regions.}
\label{fig:domint}
\end{center}
\end{figure}

Consider the function $
\til{u}_n(x,y) := \psi_n(x,y) w_n(x,y) 
		+ \bigl( 1 - \psi_n(x,y)\bigr) \ol{u}_n(x,y)$ for $(x,y) \in \R^2$,
where $\psi_n \in C^{\infty}(\R\times(0,\infty);[0,1])$ satisfies $\psi_n \equiv 1$ in $T^+_{R\rho_n}$, $\psi_n \equiv 0$ in $\R^N \bs\ol{T}^+_{R\eps_n}$, and
\begin{equation}\label{eq:bounds_cutoff}
\|\grad \psi_n\|_{\infty} \leq \frac{C}{\eps_n}
\quad \text{ and } \quad
\|\grad^2 \psi_n\|_{\infty} \leq \frac{C}{\eps_n^2}.
\end{equation}
Since $Tw_n = T\ol{u}_n = v$ in $\bigl(\ol{\triangle}_{R\eps_n} \bs \triangle_{R\rho_n}\bigr) \cap E_r$, we have that $T\til{u}_n = v$ on $\bigl( \ol{\triangle}_{R\eps_n} \bs \triangle_{R\rho_n}\bigr) \cap E_r$.
Hence $V\bigl( T\til{u}_n\bigr) = 0$ in $\bigl( \ol{\triangle}_{R\eps_n} \bs \triangle_{R\rho_n}\bigr) \cap E_r$.
Thus, it suffices to estimate 
\begin{equation}\label{eq:energyLn}
\mathcal{F}_{\eps}(\til{u}_{n}; L_n, \emptyset)
	= \int_{L_n} \left[ \eps_n^3 |\grad^2 \til{u}_{n}(x,y)|^2 + \frac{1}{\eps_n} W(\til{u}_{n}(x,y)) \right] \,dx\,dy,
\end{equation}
where $L_n := \triangle_{R\eps_n} \bs \ol{\triangle}_{R\rho_n}$.


By Young's inequality and \eqref{eq:bounds_cutoff}, for $(x,y) \in L_n$£ we have
\begin{multline}\label{eq:grad2u} 
|\grad^2 \til{u}_{n}(x,y)|^2
	\leq (1+\eta) |\grad^2 w_{n}(x,y)|^2 + C_{\eta} \left[|\grad^2 \ol{u}_{n}(x,y)|^2\right.\\ 
	  \left.  + \frac{1}{\eps_n^2} \bigl[ |\grad w_{n}(x,y)|^2 + |\grad \ol{u}_{n}(x,y)|^2 \bigr] + \frac{1}{\eps_n^4}\bigl[|w_{n}(x,y)|^2 + |\ol{u}_{n}(x,y)|^2 \bigr] \right],
\end{multline}
and, so
\begin{align}
\eps^3 \iint_{L_n} |\grad^2 \til{u}_{n}(x,y)|^2 \, dx\,dy
  & \leq (1+\eta) \eps_n^3 \iint_{L_n} |\grad^2 w_{n}(x,y)|^2 \, dx\,dy
   + C  \iint_{L_n} \biggl[ \eps_n |\grad w_{n}(x,y)|^2 +  \frac{1}{\eps_n} |w_{n}(x,y)|^2\nonumber\\
  & + \eps_n^3 |\grad^2 \ol{u}_{n}(x,y)|^2  
  + \eps_n |\grad \ol{u}_{n}(x,y)|^2  + \frac{1}{\eps_n}  |\ol{u}_{n}(x,y)|^2 \biggr] \, dx\,dy \nonumber\\
  & =: I_1 + I_2 = I_3. \label{eq:I1I2I3}
\end{align}
To estimate $I_1$, note  that 
\begin{equation}\label{eq:int_grad2w}
\eps_n^3 \iint_{L_n} |\grad^2_{(x,y)} w_{n}(x,y)|^2 \, dx\,dy
	 = \frac{\eps_n^3}{\rho_n^4} \iint_{L_n} \left|\grad^2_{(t,s)} w_{n}\bigl({\ts \frac{x}{\rho_n},\frac{y}{\rho_n}}\bigr)\right|^2 \, dx\,dy 
	 = \eps_n \lambda_n^{\frac23} \iint_{\frac{1}{\rho_n}L_n} | \grad^2 \ol{w}(t,s) | ^2 \, dt\,ds 
\end{equation}

Extend $\ol{w}$ to $T_{R\frac{\eps_n}{\rho_n}}$ using \eqref{eq:def_w}.
Since $\ol{w}(t', \cdot)$ is even, by Proposition \ref{prop:lifting} and \eqref{eq:int_grad2w}, we have
\begin{align*}
\eps_n^3 \iint_{L_n} |\grad^2_{(x,y)} w_{n}(x,y)|^2 \, dx\,dy
	& = \eps_n \lambda_n^{\frac23} \iint_{\frac{1}{\rho_n} L_n} | \grad^2_{(t,s)} \ol{w}(t,s) | ^2 \, dt\,ds \\
	& \leq \frac{7}{16}  \eps_n \lambda_n^{\frac23}\iint_{\frac{1}{\rho_n} L_n} \left| \frac{h'(s+t)-h'(s-t)}{2t}\right| ^2 \, ds \, dt \\
	& \leq \frac{7}{32}  \eps_n \lambda_n^{\frac23}\iint_{\frac{1}{\rho_n} L_n} \left| \frac{h'(s+t)-h'(s-t)}{2t}\right| ^2 \, ds \, dt \\
  & = \frac{7}{64}  \eps_n \lambda_n^{\frac23} \iint_{\bigl(-R\lambda_{n}^{\frac{1}{3}} ,R\lambda_{n}^{\frac{1}{3}}\bigr)^2 \bs [-R,R]^2} \left| \frac{h'(w)-h'(z)}{w-z}\right| ^2 \, dz \, dw.
\end{align*}

The integral we are estimating is the integral over the the ``square annulus'' in Figure \ref{fig:domint}(b).
Note that on all four corner squares of Figure \ref{fig:domint}(b), we have $h'(z) = h'(w) = 0$, so the integral reduces to
\begin{align}
I_1
	& \leq \frac{7}{32} \eps_n \lambda_n^{\frac23} \left[ \int_{-R}^R \int_{R}^{R\lambda_{n}^{\frac{1}{3}}} 
		\left| \frac{h'(w)}{w-z}\right| ^2 \, dz \, dw
			+ \int_{-R}^R \int_{-R\lambda_{n}^{\frac{1}{3}}}^{-R} 
		\left| \frac{h'(w)}{w-z}\right| ^2 \, dz \, dw \right] \label{eq:wgrad2} \\
	& \leq \frac{7}{32} (L + o(1)) \left[ \int_{-R}^R \int_{R}^{\infty} 
		\left| \frac{h'(w)}{w-z}\right| ^2 \, dz \, dw
			+ \int_{-R}^R \int_{-\infty}^{-R} 
		\left| \frac{h'(w)}{w-z}\right| ^2 \, dz \, dw \right]. \nonumber
\end{align}

To estimate $I_2$, note that since $\|w_n\|_{\infty}\leq C$, we have
\begin{equation}\label{eq:wpart}
\frac{1}{\eps_n}  \int_{L_n}  |w_{n}(t,s)|^2\,dt\,ds
	\leq C \eps_n
\end{equation}
and by H\u{o}lder inequality, Proposition \ref{prop:gagliardo-nirenberg} and \eqref{eq:wgrad2}, we obtain that
\begin{align}
\eps_n  \iint_{L_n}  |\grad w_{n}(t,s)|^2\,dt\,ds
	& \leq C \eps_n \left( \| w_n\|_{L^2(L_n)} \|\grad^2 w_n\|_{L^2(L_n)} + \| w_n\|_{L^2(L_n)}^2 \right) \nonumber \\
	& \leq C \left[ \left(\eps_n^{-\frac{1}{2}} \| w_n\|_{L^2(L_n)}\right) \left( \eps_n^{\frac{3}{2}} \|\grad^2 w_n\|_{L^2(L_n)} \right)+ \eps_n \| w_n\|_{L^2(L_n)}^2 \right] \nonumber \\
	& \leq C (\sqrt{\eps_n} + \eps_n^3). \label{eq:gagl}
\end{align}
Combining \eqref{eq:wpart} and \eqref{eq:gagl} yields 
\begin{equation}\label{eq:I2}
I_2 \leq C \sqrt{\eps_n}.
\end{equation}

We estimate $I_3$ using \eqref{eq:un_bar_bounds}.
Precisely,
\begin{equation}\label{eq:I3}
I_3
	\leq \frac{C}{\eps_n} \Leb{2}(L_n)
	\leq C \eps_n.
\end{equation}

Finally, using the fact that $\til{u}_n$ is bounded in $L^{\infty}(L_n)$, we have
\begin{equation}\label{eq:Wu}
\frac{1}{\eps_n} \iint_{L_n} W\bigl(\til{u}_n(x,y)\bigr) \, dx\,dy
	\leq \frac{C}{\eps_n} \Leb{2}(L_n) 
	\leq C \eps_n.
\end{equation}

Using \eqref{eq:energyLn}, \eqref{eq:I1I2I3}, \eqref{eq:wgrad2}, \eqref{eq:I2}, \eqref{eq:I3}, and \eqref{eq:Wu}, we obtain that
\begin{equation}\label{eq:estimate2}
\mathcal{F}_{\eps}(\til{u}_{n}, L_n, \emptyset)
	\leq \frac{7L}{32}  \left[ \int_{-R}^R \int_{R}^{\infty} 
		\left| \frac{h'(w)}{w-z}\right| ^2 \, dz \, dw
			+ \int_{-R}^R \int_{-\infty}^{-R} 
		\left| \frac{h'(w)}{w-z}\right| ^2 \, dz \, dw \right] + o(1).
\end{equation}

Combine \eqref{eq:estimate1} and \eqref{eq:estimate2} to obtain
\begin{multline}
\mathcal{F}_{\eps}(\til{u}_{n}, \triangle_{R\eps_n}, (-R\eps_n,R\eps_n) \times \{0\}) \\
	\leq L \left[ \frac{7}{16} \int_{-\infty}^{\infty} \int_{-\infty}^{\infty} 
		\left| \frac{h'(w)-h'(z)}{w-z}\right| ^2 \, dz \, dw
		+ \int_{-\infty}^{\infty} V(h(z))\,dz \right] + o(1)
	\leq \ol{c} + \eta + o(1). \label{eq:estimate3}
\end{multline}

On the other hand, from Corollary \ref{cor:interior_boundary} we know that
\begin{multline}
\mathcal{F}_{\eps}(\til{u}_{n}, D_r \bs T^+_{R\eps_n}, \emptyset)
	= \mathcal{F}_{\eps}(\ol{u}_{n}, D_r \bs T^+_{R\eps_n}, \emptyset) \\
	\leq (m+\eta) \Haus{1}(Su) + \sum_{z=a,b} \sum_{\xi=\alpha,\beta} (\sigma(z,\xi) + \eta) \Haus{1}(\{Tu=z,v=\xi\}) + o(1). \label{eq:estimate4}
\end{multline}

The result follows by combining \eqref{eq:estimate3} and \eqref{eq:estimate4}.


\paragraph{Step 2. } General $N$-dimensional problem.

In this case, we define $u_n(x) := \til{u}_n(x_{N-1},x_N)$ for $x=(x'',x_{N-1},x_N) \in D_r$.
By Fubini's theorem  and Step 1, we deduce that
\begin{align*}
\mathcal{F}_{\eps}(u_{n}, B_r^+, E_r)
	& = \int_{B_r^{N-2}} \mathcal{F}_{\eps_n}(\til{u}_n; D_r^{x''}, E_r^{x''}) \, dx''\\
	& \leq \int_{B_r^{N-2}}  \biggl[ (m+\eta) \Haus{1}(Su \cap (\{x''\} \times \R^2)) \\
	& \hspace{2cm}+ \sum_{z=a,b} \sum_{\xi=\alpha,\beta} (\sigma(z,\xi) + \eta) \Haus{1}(\{Tu=z,v=\xi\} \cap (\{x''\} \times \R^2))  \\
	& \hspace{4cm} + (\ol{c}+\eta)L\Haus{0}( Sv\cap (\{x''\} \times \R^2))\biggr] \, dx'' + o(1).
\end{align*}
Using Theorem \ref{prop:Bv_bdy}, we then deduce that
$$
\limsup_{n\to\infty} \mathcal{F}_{\eps}(u_{n}, B_r^+, E_r) 
	\leq (m+\eta) \Haus{N-1}(Su)
	+ \sum_{z=a,b} \sum_{\xi=\alpha,\beta} (\sigma(z,\xi) + \eta) \Haus{N-1}(\{Tu=z,v=\xi\})
	+ (\ol{c}+\eta)L\Haus{N-2}(Sv).
$$
This completes the proof.
\end{proof}

\begin{prop}\label{prop:recoveryC2}

Let $W:\R \to [0,\infty)$ satisfy $(H^W_1)-(H^W_2)$, let $V:\R \to [0,\infty)$ satisfy $(H^V_1)-(H^V_3)$ Let $\eps_n \to 0^+$ with $\eps_n \lambda_n^{\frac23} \to L\in(0,\infty)$, let $\eta>0$, and let $u \in BV(\Omega;\{a,b\})$ and $v \in BV(\partial\Omega;\{\alpha,\beta\})$, with $Su$ an $N-1$ manifold of class $C^2$ such that $\Haus{N-1}(\partial\Omega \cap \ol{Su}) = 0$ and $Sv$ an $N-2$ manifold of class $C^2$ .
Then there exists $\{u_{n}\} \subset  H^2(\Omega)$ such that $u_{n} \to u$ in $L^2(\Omega)$, $Tu_{n} \to v$ in $L^2(\partial\Omega)$, and
$$
\limsup_n \mathcal{F}_{\eps_n}(u_{n}) 
\leq (m + \eta) \Haus{N-1}(Su)
+ \sum_{z=a,b} \sum_{\xi=\alpha,\beta} (\sigma(z,\xi)+ \eta) \Haus{N-1}(\{Tu=z,v=\xi\}) + (\ol{c} + \eta) L \Haus{N-2}(Sv),
$$
where $m$, $\sigma$, and $\ol{c}$ are the constants defined in \eqref{eq:def_m}, \eqref{eq:def_sigma}, and \eqref{eq:def_over_c}, respectively.
\end{prop}

\begin{proof}

From Corollary \ref{cor:interior_boundary}, it suffices to prove that
\begin{multline}\label{eq:claim}
\limsup_n \mathcal{F}_{\eps_n}(u_{n};\Omega \cap B(x_0;r), \partial\Omega \cap B(x_0;r))
\leq (m + \eta) \Haus{N-1}(Su \cap B(x_0;r)) \\
+ \sum_{z=a,b} \sum_{\xi=\alpha,\beta} (\sigma(z,\xi)+ \eta) \Haus{N-1}(\{Tu=z,v=\xi\}\cap B(x_0;r)) + (\ol{c} + \eta) L \Haus{N-2}(Sv\cap B(x_0;r)),
\end{multline}
for each point $x_0 \in Sv$ and for some neighborhood $A$ of $x_0$.

\begin{figure}[!htbp]
\begin{center}
\includegraphics*[width=300pt]{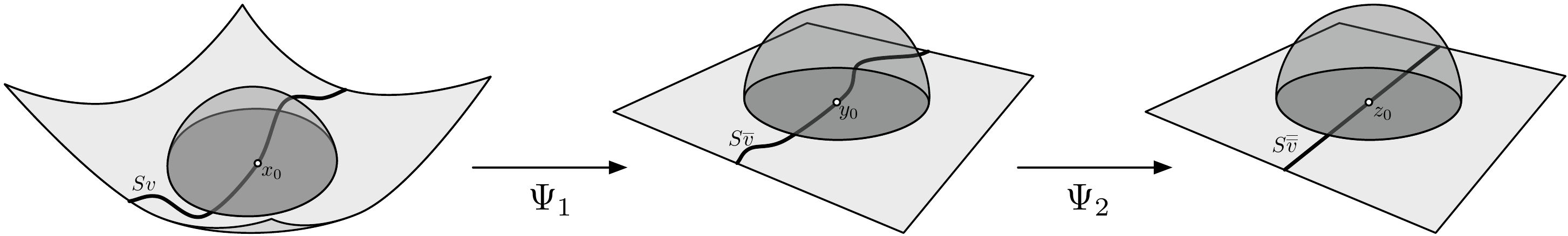}
\caption{Scheme for the flattening of the boundary and of $Sv$.}
\label{fig:change_variables}
\end{center}
\end{figure}

First we fix a point $x_0 \in Sv$. 
Since the domain is of class $C^2$, we can find $\ol{r}>0$ such that, up to a rotation,
\begin{equation}\label{eq:bdy_graph}
\partial \Omega \cap B(x_0;\ol{r}) = \bigl\{ x \in \R^N: x_N = \gamma_1(x') \bigr\},
\end{equation}
for some function $\gamma_1 \in C^2(\R^{N-1})$. 
So we define $\Psi_1(x) := \bigl(x', x_N - \gamma_1(x') \bigr)$, $\ol{u}(y) := (u \circ \Psi_1^{-1}) (y)$, and $\ol{v}(y) := (v \circ \Psi_1^{-1}) (y)$.

Moreover, $S\ol{v}$ is also of class $C^2$, so we can find $0<r<\ol{r}$ such that, up to a ``horizontal rotation'', i.e., $R = \left(\begin{tabular}{c|c} $R'$ & $0$ \\ \hline $0$ & $1$ \end{tabular}\right)$, with $R' \in SO(N-1)$, we have
$$
S\ol{v} \cap B(x_0;r) = \bigl\{ y \in \R^{N-1}\times\{0\}: y_{N-1} = \gamma_2(y'') \bigr\},
$$
for some function $\gamma_2 \in C^2(\R^{N-2})$.
Let $\Psi_2(y) := \bigl(y'', y_{N-1} - \gamma_2(y''), y_N \bigr)$, $\ol{\ol{u}}(z) := \bigl(\ol{u} \circ \Psi_2^{-1}\bigr) (z) = \left( u \circ ( \Psi_2 \circ \Psi_1)^{-1} \right)(z)$, $\ol{\ol{v}}(z) := \left(v \circ ( \Psi_2 \circ \Psi_1)^{-1}\right)(z)$.

Let $\Phi := \Psi_2 \circ \Psi_1:\R^N \to \R^N$, which is a bi-Lipschitz homeomorphism. Moreover, its isometry defect $\delta_r$ vanishes as $r \to 0$ due to the regularity of both $\partial\Omega$ and $Sv$.

Let $z_0 := \Phi(x_0) \in S\ol{\ol{z_0}}$. Note that $D_r:=\Phi\bigl(\Omega \cap B(x_0;r)\bigr)$ is a neighborhood of $z_0$, and set $E_r := \Phi\bigl(\partial\Omega \cap B(x_0;r)\bigr)$. 
Let $\{u_n\} \subset H^2\left(D_r\right)$ be defined as in Proposition \ref{prop:boundary_flat} with $\ol{\ol{u}}$ and $\ol{\ol{v}}$.
Then from Proposition \ref{prop:flatbdy}, we have that
\begin{multline*}
\mathcal{F}_{\eps_n}\left(u_n \circ \Phi; \Omega \cap B(x_0;r),  \partial\Omega \cap B(x_0;r)\right)
	\leq (1-\delta_r)^{-(N+4)} \mathcal{F}_{\eps_n}\left(u_n; D_r, E_r)\right) \\
	+ \frac{\delta_r}{(1-\delta_r)^{N+2}} \eps_n^3 \int_{D_r} \left(|\grad^2 u_n(z)| \, |\grad u_n(z)| + \delta_r |\grad u_n(z)|^2 \right) \, dz.
\end{multline*}

On the other hand, by H\"{o}lder and Propositions \ref{prop:gagliardo-nirenberg} and \ref{prop:boundary_flat}, we have that 
$$
\eps_n^3 \int_{B^+(z_0;r)} \left(|\grad^2 u_n(z)| \, |\grad u_n(z)| + \delta_r |\grad u_n(z)|^2 \right) \, dz
\leq C \eps_n,
$$
and that
\begin{multline*}
\mathcal{F}_{\eps_n}\left(u_n; D_r, E_r\right)
	\leq (m + \eta) \Haus{N-1}(S\ol{\ol{u}} \cap D_r) + \sum_{z=a,b} \sum_{\xi=\alpha,\beta} (\sigma(z,\xi)+ \eta) \Haus{N-1}(\{T\ol{\ol{u}}=z,\ol{\ol{v}}=\xi\} \cap E_r) \\
		+ (\ol{c} + \eta) L \Haus{N-2}(S\ol{\ol{v}} \cap E_r) + o(1).
\end{multline*}

Moreover,
\begin{align*}
\Haus{N-1}(S\ol{\ol{u}} \cap D_r)
	& = \Haus{N-1}(S(u \circ \Phi^{-1}) \cap D_r)
	 = \Haus{N-1}\left(\Phi(Su) \cap \Phi\bigl(\Omega \cap B(x_0;r)\bigr)\right)\\
	& \leq \Haus{N-1}\left(\Phi\bigl( Su \cap B(x_0;r)\bigr)\right)
	 \leq \text{Lip}(\Phi)^{N-1} \Haus{N-1}\left(Su \cap B(x_0;r)\bigr)\right) \\
	& = \Haus{N-1}\bigl(Su \cap B(x_0;r)\bigr),
\end{align*}
because $\Phi$ is an isomorphism.
Analogously, we deduce that 
\begin{align*}
& \Haus{N-1}(\{T\ol{\ol{u}}=z,\ol{\ol{v}}=\xi\} \cap D_r)
	\leq \Haus{N-1}\left(\{Tu=z,v=\xi\} \cap B(x_0;r)\right), \\
& \Haus{N-2}(S\ol{\ol{v}} \cap D_r)
	\leq \Haus{N-2}\left(Sv \cap B(x_0;R\eps_n)\right).
\end{align*}

Hence
\begin{multline*}
\limsup_n \mathcal{F}_{\eps_n}\left(u_n \circ \Phi; \Omega \cap B^+(x_0;r)\bigr);  \partial \Omega \cap B(x_0;r)\right)
	\leq (1-\delta_r)^{-(N+4)} \biggl[ (m + \eta) \Haus{N-1}\left(Su \cap B(x_0;r)\right) \\
	+ \sum_{z=a,b} \sum_{\xi=\alpha,\beta} (\sigma(z,\xi)+ \eta) \Haus{N-1}\left(\{Tu=z,v=\xi\} \cap B(x_0;r)\right)
	 + (\ol{c} + \eta) L \Haus{N-2}\left(Sv \cap B(x_0;r)\right) \biggr]
\end{multline*}

This proves the result.
\end{proof}

%
%

\begin{proof}[Proof of Theorem \ref{thm:critical}(ii)]

Since $u \in BV(\Omega; \{a, b\})$, we may write $u$ as
$$
u(x) =
\begin{cases}
	a	& \text{ if } x \in E_{a}, \\
	b	& \text{ if } x \in \Omega \bs E_{a},
\end{cases}
$$
where $E_{a}$ is a set of finite perimeter in $\Omega$.
Similarly, since $v \in BV(\partial\Omega;\{\alpha,\beta\})$, we may write $v$ as
$$
v(x) =
\begin{cases}
	\alpha	& \text{ if } x \in F_{\alpha}, \\
	\beta	& \text{ if } x \in \partial\Omega \bs F_{\alpha},
\end{cases}
$$
where $F_{\alpha}$ is a set of finite perimeter in $\partial\Omega$.
Apply Proposition \ref{prop:giusti} to the set $E_a$ to obtain a sequence of sets $E_k$ of class $C^2$ such that $\Leb{N}(E_a \triangle E_k) \to 0$ and $\Haus{N-1}(\partial E_a \cap \partial E_k) \to 0$.
By slightly modifying each $E_k$, we may assume that $\Haus{N-1}(\partial \Omega \cap \partial E_k) = 0$.
Similarly, by Proposition \ref{prop:giusti:finiteper} applied to the set $F_{\alpha}$, we may find a sequence of sets $F_k \subset \partial\Omega$ of class $C^2$ such that $\Haus{N-1}(F_{\alpha} \triangle F_k) \to 0$ and $\Haus{N-2}(\partial_{\partial\Omega} F_{\alpha} \triangle \partial_{\partial\Omega} F_k) \to 0$.
Define the sequences of functions
$$
u_{k}(x) :=
\begin{cases}
	a	& \text{ if } x \in \Omega \cap E_{k}, \\
	b	& \text{ if } x \in \Omega \bs E_{k},
\end{cases} \qquad 
 v_{k}(x) :=
\begin{cases}
	\alpha	& \text{ if } x \in \partial \Omega \cap F_{k}, \\
	\beta	& \text{ if } x \in \partial \Omega \bs F_{k}.
\end{cases}
$$

Apply Proposition \ref{prop:recoveryC2} to find $\{u_{k,n}\} \subset H^2(\Omega)$ such that $u_{k,n} \xto{n} u_k$ in $L^2(\Omega)$, $Tu_{k,n} \xto{n} v_k$ in $L^2(\partial\Omega)$, and 
$$
\limsup_{n}\mathcal{F}_{\eps_n}(u_{k,n}) 
	\leq \bigl(m+{\ts \frac{1}{k}}\bigr) \Per_{\Omega}(E_k)  
	+ \sum_{z =a,b} \sum_{\xi=\alpha,\beta} \bigl(\sigma(z,\xi)+{\ts \frac{1}{k}}\bigr) \Haus{N-1}\bigl( \{ Tu_k = z\} \cap \{v_k = \xi\} \bigr)
	+ \bigl(\ol{c} + {\ts \frac{1}{k}}\bigr) L \Per_{\partial\Omega}(F_{k}).
$$

Since $u_k \to u$ in $L^2(\Omega)$ and $v_k \to v$ in $L^2(\partial\Omega)$, we have 
\begin{gather*}
\lim_{k} \lim_{n} \|u_{k,n} - u\|_{L^{2}(\Omega)} = 0,
\qquad
\lim_{k} \lim_{n} \|Tu_{k,n} - v\|_{L^{2}(\partial\Omega)} = 0, \\
\limsup_{k} \limsup_{n}\mathcal{F}_{\eps_n}(u_{k,n}) 
	\leq m \Per_{\Omega}(E_a)
	+ \sum_{z =a,b} \sum_{\xi=\alpha,\beta} \sigma(z,\xi) \Haus{N-1}\bigl( \{ Tu = z\} \cap \{v = \xi\} \bigr)
	+ \ol{c} L \Per_{\partial\Omega}(F_{\alpha}).
\end{gather*}
Diagonalize to get a subsequence $k_n \to \infty $ and 
obtain $u_n:= u_{k_n,n} \to u$ in $L^2(\Omega)$, $Tu_n \to v$ in $L^2(\partial\Omega)$, and
$$
\limsup_{n} \mathcal{F}_{\eps_n}(u_{n}) 
	\leq m \Per_{\Omega}(E_{a}) 
	+ \sum_{z =a,b} \sum_{\xi=\alpha,\beta} \sigma(z,\xi) \Haus{N-1}\bigl( \{ Tu = z\} \cap \{v = \xi\} \bigr) 
	+ \ol{c} L \Per_{\partial\Omega}(F_{\alpha}).
$$
This completes the proof.
\end{proof}


\ack

This research was partially funded by Funda\c{c}\~{a}o para a Ci\^{e}ncia e a Tecnologia under grant SFRH/BD/8582/2002, the Department of Mathematical Sciences of Carnegie Mellon University and its Center for Nonlinear Analysis (NSF Grants No. DMS-0405343 and DMS-0635983), Irene Fonseca (NSF Grant DMS-0401763) and Giovanni Leoni (NSF Grants No. DMS-0405423 and DMS-0708039).

The author thanks Vincent Millot and Dejan Slep\v{c}ev for the fruitful conversations, Luc Tartar for useful conversations on Proposition \ref{prop:lifting}, and is indebted to Irene Fonseca and Giovanni Leoni for uncountable discussions and advice as the work progressed that largely influenced its course.


\bibliographystyle{amsplain}
\bibliography{refs}

\end{document}